\documentclass[reqno,10pt]{amsart}
\usepackage{amscd,amssymb,amsmath,amsthm}
\usepackage{graphicx}
\usepackage{color}
\usepackage[all]{xy}
\usepackage{enumitem}

\usepackage[perpage]{footmisc}

\def\fix#1{{{#1}}}
\def\fixtwo#1{{{#1}}}
\theoremstyle{plain}
\newtheorem{proposition}{Proposition}
\newtheorem{theorem}[proposition]{Theorem}

\newtheorem*{conjecture*}{Conjecture}
\newtheorem{definition}[proposition]{Definition}
\newtheorem{corollary}[proposition]{Corollary}
\newtheorem{lemma}[proposition]{Lemma}
\newtheorem{remark}[proposition]{Remark}

\newtheorem{example}[proposition]{Example}

\newtheorem{proposition-definition}[proposition]{Proposition/Definition}
\newtheorem*{proposition*}{Proposition}
\newtheorem*{theorem*}{Theorem}

\newtheorem*{openproblem*}{Open Problem}
\newtheorem*{maintheorem*}{Main Theorem}
\newtheorem*{maincorollary*}{Main Corollary}
\newtheorem*{corollary*}{Corollary}
\newtheorem*{lemma*}{Lemma}
\newtheorem*{remark*}{Remark}
\newtheorem*{definition*}{Definition}
\newtheorem*{example*}{Example}
\newtheorem*{examples*}{Examples}
\newtheorem*{criterion*}{Generation Criterion}
\newtheorem*{explanation*}{Explanation}

\numberwithin{proposition}{section}
\numberwithin{equation}{section}

\def\co{\colon\thinspace}

\newcommand{\N}{\mathbb{N}}
\newcommand{\Z}{\mathbb{Z}}
\newcommand{\Q}{\mathbb{Q}}
\newcommand{\R}{\mathbb{R}}
\newcommand{\C}{\mathbb{C}}
\newcommand{\K}{\mathbb{K}}
\newcommand{\F}{\mathbb{F}}
\newcommand{\D}{\mathbb{D}}

\renewcommand{\P}{\mathbb{P}}

\newlist{CZ}{enumerate}{1}
\setlist[CZ]{label=(CZ\arabic*)}

\begin{document}

\title[McKay correspondence]
{The McKay correspondence for isolated singularities via Floer theory}

\author{Mark McLean}
\thanks{The research of Mark McLean was partially supported by the NSF grant DMS-1508207.}
\address{M. McLean, Mathematics Department, Stony Brook University, NY, U.S.A.}
\email{markmclean@math.stonybrook.edu}
\author{Alexander F. Ritter}
\address{A. F. Ritter, Mathematical Institute, University of Oxford, England.}
\email{ritter@maths.ox.ac.uk}

\date{version: \today}

\begin{abstract}%
We prove the generalised McKay correspondence for isolated singularities using Floer theory. Given an isolated singularity $\C^n/G$ for a finite subgroup $G\subset SL(n,\C)$ and any crepant resolution $Y$, we prove that the rank of positive symplectic cohomology $SH^*_+(Y)$ is the number $|\mathrm{Conj}(G)|$ of conjugacy classes of $G$, and that twice the age grading on conjugacy classes is the $\Z$-grading on $SH^{*-1}_+(Y)$ by the Conley-Zehnder index. The generalized McKay correspondence follows as $SH^{*-1}_+(Y)$ is naturally isomorphic to ordinary cohomology $H^*(Y)$, due to a vanishing result for full symplectic cohomogy. In the Appendix we 
construct a novel filtration on the symplectic chain complex for any non-exact convex symplectic manifold, which yields both a Morse-Bott spectral sequence and a construction of positive symplectic cohomology.
\end{abstract} 
\maketitle
%
\section{Introduction}
%
\subsection{The classical McKay correspondence}
\label{Section The classical McKay correspondence}
The classical McKay correspondence is a description of the representation theory of finite subgroups $G\subset SL(2,\C)$ in terms of the geometry of the minimal resolution $\pi: Y \to \C^2/G$. Recall a resolution consists of a non-singular quasi-projective variety $Y$ \fix{together with a proper, birational morphism}
$\pi$ which is a biholomorphism away from the singular locus.
In the case of $\C^2/G$, there is only an isolated singularity at the origin.
Minimality means other resolutions factor through it, and in this case it is equivalent to the absence of rational holomorphic $(-1)$-curves in $Y$. The exceptional locus $E=\pi^{-1}(0)\subset Y$ is a tree of transversely intersecting exceptional divisors $E_j$, where each $E_j$ is a rational holomorphic $(-2)$-curve. Finite subgroups $G\subset SL(2,\C)$, up to conjugation, are in 1-to-1 correspondence with ADE Dynkin diagrams. The diagram for $G$ can be recovered by assigning a vertex to each $E_j$, and an edge between vertices whenever the corresponding divisors intersect. For example, the real picture for $D_4$ is\footnote{$X=\C^2/\widetilde{\mathbb{D}_4}$. 
The binary dihedral group $\widetilde{\mathbb{D}_4}$ is the quaternion group; it has size $8$ and double covers via $SU(2)\to SO(3)$ a size $4$ dihedral group $C_2\times C_2\cong \mathbb{D}_4\subset SO(3)$. Circles depicting $E$ represent copies of $\C \P^1$. The quaternion group has four non-trivial conjugacy classes: $-1$, $\pm i$, $\pm j$, $\pm k$.}
\\[2mm]
\strut\hspace{16mm}\begin{picture}(0,0)%
\includegraphics{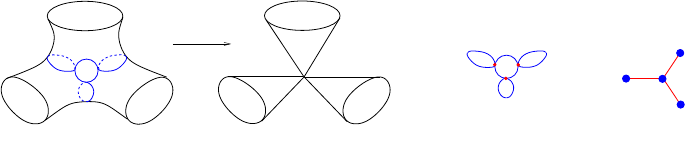}%
\end{picture}%
\setlength{\unitlength}{1657sp}%
\begingroup\makeatletter\ifx\SetFigFont\undefined%
\gdef\SetFigFont#1#2#3#4#5{%
  \reset@font\fontsize{#1}{#2pt}%
  \fontfamily{#3}\fontseries{#4}\fontshape{#5}%
  \selectfont}%
\fi\endgroup%
\begin{picture}(13048,2771)(-2414,-5422)
\put(-944,-3031){\makebox(0,0)[lb]{\smash{{\SetFigFont{7}{8.4}{\rmdefault}{\mddefault}{\updefault}{\color[rgb]{0,0,0}$Y$}%
}}}}
\put(1306,-3301){\makebox(0,0)[lb]{\smash{{\SetFigFont{7}{8.4}{\rmdefault}{\mddefault}{\updefault}{\color[rgb]{0,0,0}$\pi$}%
}}}}
\put(3241,-3031){\makebox(0,0)[lb]{\smash{{\SetFigFont{7}{8.4}{\rmdefault}{\mddefault}{\updefault}{\color[rgb]{0,0,0}$X$}%
}}}}
\put(1801,-5326){\makebox(0,0)[lb]{\smash{{\SetFigFont{7}{8.4}{\rmdefault}{\mddefault}{\updefault}{\color[rgb]{0,0,0}$\{x^2+zy^2+z^3=0\}\subset \C^3$}%
}}}}
\put(6391,-2941){\makebox(0,0)[lb]{\smash{{\SetFigFont{7}{8.4}{\rmdefault}{\mddefault}{\updefault}{\color[rgb]{0,0,0}$\textrm{Exceptional}$}%
}}}}
\put(6661,-3301){\makebox(0,0)[lb]{\smash{{\SetFigFont{7}{8.4}{\rmdefault}{\mddefault}{\updefault}{\color[rgb]{0,0,0}$\textrm{divisors}$}%
}}}}
\put(9451,-3301){\makebox(0,0)[lb]{\smash{{\SetFigFont{7}{8.4}{\rmdefault}{\mddefault}{\updefault}{\color[rgb]{0,0,0}$\textrm{Diagram}$}%
}}}}
\put(9496,-2941){\makebox(0,0)[lb]{\smash{{\SetFigFont{7}{8.4}{\rmdefault}{\mddefault}{\updefault}{\color[rgb]{0,0,0}$\textrm{Dynkin}$}%
}}}}
\put(9766,-5011){\makebox(0,0)[lb]{\smash{{\SetFigFont{7}{8.4}{\rmdefault}{\mddefault}{\updefault}{\color[rgb]{0,0,0}$D_4$}%
}}}}
\put(6526,-5011){\makebox(0,0)[lb]{\smash{{\SetFigFont{7}{8.4}{\rmdefault}{\mddefault}{\updefault}{\color[rgb]{0,0,0}$E=\pi^{-1}(0)$}%
}}}}
\end{picture}%

\\[2mm]
The correspondence \cite{McKay} states that the non-trivial irreducible representations $V_i$ of $G$ can be labelled by the vertices of the Dynkin diagram so that the adjacency matrix $A_{ij}$ of the diagram determines the tensor products $\C^2\otimes V_i\cong \oplus\, A_{ij}\, V_j$ with the canonical representation. 
The cohomology $H^*(Y,\C)$ consists of $H^0(Y)=\C\cdot 1$, $H^2(Y)=\oplus\, \C\cdot \mathrm{PD}[E_j]$. So the dimension of $H^*(Y)$, or the Euler characteristic $\chi(Y)$, is the number of irreducible representations. As $G$ is finite, this is the number of conjugacy classes, namely the dimension of the representation ring $\mathrm{Rep}(G)$ (although there is no natural bijection between $\mathrm{Irreps}(G)$ and $\mathrm{Conj}(G)$).

\begin{example} The simplest case $G=\Z/2=\{\pm I\}\subset SL(2,\C)$ yields $Y=T^*\C\P^1=\mathcal{O}_{\C\P^1}(-2)$ arising as the blow-up at $0$ of the Veronese variety\footnote{The image of $\nu_2:\C^2/\{\pm 1\}\hookrightarrow \C^3$, $(x,y)\mapsto (x^2,xy,y^2)$.} $\mathbb{V}(XZ-Y^2)\subset \C^3$. The K\"{a}hler form on $T^*\C\P^1$ makes $\C\P^1$ holomorphic and symplectic, unlike the canonical exact symplectic form on $T^*S^2$. Here $H^*(Y)$ has two generators $1,\omega$, and $G$ has two conjugacy classes $I,-I\in \mathrm{Conj}(G)$.
\end{example}

\begin{remark}
This fails for $G\subset GL(2,\C)$, for example for $G\cong \Z/2$ generated by the reflection $(z,w)\mapsto (z,-w)$, then $\C^2/G\cong \C^2$ is already non-singular but does not remember $G$. 
\end{remark}

\subsection{The generalised McKay correspondence}
\label{Subsection The generalised McKay correspondence}
More generally, for any $n\geq 2$, one considers resolutions of quotient singularities for finite subgroups $G\subset SL(n,\C)$,
\begin{equation}\label{Equation pi from Y to X is Cn mod G}
\pi: Y \to X = \C^n/G,
\end{equation}
viewing $\C^n/G=\mathrm{Spec}\, \C[z_1,\ldots,z_n]^G$ as an affine variety.
There is no longer a preferential resolution, so one requires $\pi$ to be \emph{crepant}, meaning the canonical bundles satisfy $K_Y=\pi^*K_X$ (which is therefore trivial). \fix{Resolutions of such $X$ always exist} by Hironaka \cite{Hironaka}, but crepant resolutions may not exist;\footnote{If $Y \to \C^4/\pm 1$ were crepant then by Theorem \ref{Theorem generalised McKay correspondence}, $H^*(Y)$ would have two generators, in degrees $0,4$ (twice the age grading of $\pm 1$), contradicting that $E=\pi^{-1}(0)$ is a projective variety
 with $H^*(E)\cong H^*(Y)$.} when they exist they need not be unique even though they always \fix{admit the same collection of} divisors \cite[Thm.1.4]{Ito-Reid}.
\fix{For $n\leq 3$,} they exist \cite[Thm.1.2]{BridgelandKingReid}. \fix{For $n=3$,} they are related by flops.

The conjecture $\chi(Y)=|\mathrm{Conj}(G)|$ dates back to work of Dixon-Harvey-Vafa-Witten \cite{DHVW}, Atiyah-Segal \cite{Atiyah-Segal} and Hirzebruch-H\"{o}fer \cite{Hirzebruch-Hoefer}. In the early 1990s, the conjecture was refined by Miles Reid \cite{Reid2,Reid} by taking into account the grading of $H^*(Y)$.
Namely, consider the dual action\footnote{%
In the notation of Ito-Reid \cite{Ito-Reid}, we are taking the age grading of $\varphi_g\in \mathrm{Hom}(\mu_r,G)$, $\varphi_g(e^{2\pi i/r})=g^{-1}$, where $\mu_r\subset \C^*$ is the group of $r$-th roots of unity. The inverse reflects the fact that we do not dualise $H^{2k}(Y,\C)$ (compare \cite[Theorem 1.6]{Ito-Reid}). This choice agrees with Kaledin \cite{Kaledin}, where 
a representation $g: \mu_r \to GL(\C^n)$ labels eigensummands $V_j\subset \C^n$ so that the action is $\lambda \cdot x = \lambda^{-b_j}x$ for  $b_j\in [0,r)\cap \Z$; the dual action on the coordinate ring $\mathcal{O}(\C^n)=\C[(\C^n)^*]$ gets rid of that inversion and $\mathrm{age}(g)=\frac{1}{r}\sum b_j\dim V_j$ (so our $a_j=2\pi b_j/r$).
%
}
of $g\in SL(n,\C)$ on $\C^n$, $g\cdot x=g^{-1}(x)$, and let $\lambda_1,\ldots,\lambda_n\in U(1)$
%
%
 denote the unordered eigenvalues (which are $|G|$-th roots of unity). Writing $\lambda_j = e^{i a_j}$ for $a_j\in [0,2\pi)$, define the \emph{age grading} on $\mathrm{Conj}(G)$ by\footnote{This is an integer, as $\sum a_j$ is divisible by $2\pi$ as $\prod \lambda_j = \det g^{-1} = 1$, and it only depends on $[g]\in \mathrm{Conj}(G)$.}
\begin{equation}\label{Equation age}
\mathrm{age}(g) = \tfrac{1}{2\pi}\sum a_j \in [0,n).
\end{equation}
The generalised McKay correspondence, as reformulated by \fix{Reid \cite{Reid2}}, is the following. 

\begin{theorem}\label{Theorem generalised McKay correspondence}
$\dim H^{2k}(Y,\C)=|\mathrm{Conj}_k(G)|$ 
where $\mathrm{Conj}_k(G)$ denotes the conjugacy classes of \emph{age} $k$, and the odd cohomology of $Y$ vanishes.
\end{theorem}

This was proved for $n=3$ by Ito-Reid \cite{Ito-Reid}; for general $n$ and abelian $G$ it was proved using toric geometry by Batyrev-Dais \cite{BatyrevDais}; in full generality it was proved using motivic integration machinery by Batyrev \cite{Batyrev2} and later by Denef-Loeser \cite{Denef-Loeser}.
We refer to Craw's thesis \cite{Craw2} \fix{and the references therein} for an extensive history of the generalisations of the McKay correspondence, in particular on the extensions to a statement about the K-theory of $Y$ in terms of $\mathrm{Rep}(G)$, and more generally about relating the derived categories of coherent sheaves on $Y$ and of $G$-equivariant sheaves on $\C^n$. \fix{In particular, on the latter generalisations, we highlight the work of Bridgeland-King-Reid \cite{BridgelandKingReid} and Bezrukavnikov-Kaledin \cite{BezrukavnikovKaledin}, and we refer the reader to Craw's expository notes \cite{CrawExpository} and the references therein.}
\fix{Finally, we mention that Reid \cite{Reid2} also strengthened the above correspondence statement with the following conjecture:}

\begin{openproblem*}
There is a natural basis of $H^*(Y)$ labelled by the conjugacy classes of $G$.
\end{openproblem*}

Although the precise meaning of natural is not known, reasonable labellings are known for $n=2$ by the classical correspondence; for $n=3$ by Ito-Reid \cite{Ito-Reid};
%
and by Kaledin \cite{Kaledin} for even $n=2m$ when $G$ preserves the complex symplectic form on $\C^{2m}$.

Our approach to the McKay Correspondence, which we describe below, uses only tools from symplectic topology and thus it differs significantly from the above algebraic geometry literature. Our way of thinking about the McKay Correspondence is fundamentally new, and we expect that it will lead to new insights into the crepant resolution conjecture, which we will address in a subsequent paper \cite{McLean-Ritter}.

\subsection{Isolated singularities}
\label{Subsection Isolated singularities}

In this paper we consider the case when the singularity is isolated.
Our approach via Floer theory works in examples of non-isolated singularities, however generalising the proofs is harder as the moduli space of Hamiltonian orbits in $Y$ lying over the singular locus is difficult to pin down. In a subsequent paper \cite{McLean-Ritter}, we will prove the McKay Correspondence in the non-isolated case in a slightly different way, but based on the foundational work of this paper.

\begin{lemma}\label{Lemma G acts freely}
For $G\subset SL(n,\C)$ any finite subgroup, $\C^n/G$ is an isolated singularity if and only if $G$ acts freely away from $0\in \C^n$ (i.e.\;the eigenvalues of $g\neq 1\in G$ are not equal to $1$).
\end{lemma}
\begin{proof}
Given any finite subgroup $Q\subset GL(n,\C)$, the Chevalley-Shephard-Todd theorem \cite{Chevalley} states that $\C^n/Q$ is smooth if and only if $Q$ is generated by quasi-reflections\footnote{A quasi-reflection is a non-identity element $A$ for which $A-I$ has rank one (i.e.\;$\mathrm{codim}_{\C}\,\mathrm{Fix}(A)=1$).}, in which case $\C^n/Q\cong \C^n$.
However, a finite order element of $SL(n,\C)$ cannot be a quasi-reflection.
%
So finite $G\subset SL(n,\C)$ are \emph{small}, i.e.\,contain no quasi-reflections.
%
%
The singular set of $\C^n/G$ is
\begin{equation}\label{Equation singular set}
\mathrm{Sing}(\C^n/G) = \{v\in \C^n: g\cdot v = v \textrm{ for some }1\neq g\in G\}/G.
\end{equation}
Indeed, where $G$ acts freely the quotient is easily seen to be smooth. Conversely, at a point $v$ as above, pick a $\mathrm{Stab}_G(v)$-invariant analytic neighbourhood $V\subset \C^n$ of $v$. Then $V/\mathrm{Stab}_G(v)$ is analytically isomorphic to a neighbourhood of $[v]\in \C^n/G$. By the same theorem, this is isomorphic to $\C^n$ if and only if $\mathrm{Stab}_G(v)$ is generated by quasi-reflections. But the latter fails as $G$ is small, so $v$ is indeed singular. We refer to \cite{Fujiki,Prill} for more precise details.
\end{proof}

\begin{remark}
The Kleinian singularities in Sec.\ref{Section The classical McKay correspondence} are always smoothable. This fails in dimension $n\geq 3$ for the above isolated singularities $\C^n/G$ by Schlessinger's rigidity theorem \cite{Schlessinger}.
\end{remark}

\begin{examples*}
Finite groups $G$ admitting a fixed point free faithful complex representation were classified by Wolf \cite[Theorem 7.2.18]{Wolf} $($cf.\,also the final comments in \cite[Example 1.43]{Hatcher}$)$. Those representations which yield subgroups in $SL(n,\C)$ have been classified by Stepanov \cite{Stepanov}. 
%
%
%
For abelian groups, it forces $G$ to be a cyclic group.\footnote{After a conjugation, one may assume all matrices in $G$ are diagonal, then the projection to the $(1,1)$-entry gives an injective group homomorphism into $S^1$, and finite subgroups of $S^1$ are cyclic.} 
%
%
When $n$ is an odd prime, in particular for $n=3$, the only finite subgroups $G\subset SL(n,\C)$ that give rise to an isolated singularity are cyclic groups, by Kurano-Nishi \cite{KuranoNishi}.
A simple example is $X=\C^3/(\Z/3)$, where $\Z/3$ acts diagonally by third roots of unity, which admits the (unique) crepant resolution $\pi: Y\to X$ given by blowing up $0$, with exceptional divisor $E=\pi^{-1}(0)\cong \C\P^2$. 
%
%
Lens spaces \cite[Example 2.43]{Hatcher} yield a family of examples of cyclic actions, namely $G=\Z/m$ acts on $\C^n$ by rotation by $(e^{2\pi i \ell_1/m},\ldots,e^{2\pi i \ell_n/m})$ where $\ell_j\in \Z$ are coprime to $m$ and $\sum \ell_j \equiv 0 \;\mathrm{mod}\;m$.
For higher dimensional non-abelian examples, we refer to the detailed discussion by Stepanov \cite{Stepanov}.
%
%
%
%
%
\end{examples*}

\subsection{An outline of our proof using Floer theory}
\label{Subsection An outline of our proof using Floer theory}
Let \eqref{Equation pi from Y to X is Cn mod G} be a crepant resolution  of an isolated singularity.
By an averaging argument\footnote{By using an element $h\in SL(n,\C)$ we can change the standard basis of $\C^n$ to a basis of eigenvectors for the $G$-invariant inner product $\tfrac{1}{|G|}\sum_{g\in G} \langle g\cdot,g\cdot \rangle_{\C^n}$. Then the $h$-conjugate of $G$ lies in $SU(n)$.} 
we may assume $G\subset SU(n)$. 
As $Y$ is quasi-projective, it inherits a K\"{a}hler form $\omega$ from an embedding into a projective space. One can modify the K\"{a}hler form so that away from a small neighbourhood of $$E=\pi^{-1}(0)\subset Y$$ it agrees via $\pi$ with the standard K\"{a}hler form on $\C^n/G$ (Lemma \ref{lemma convex symplectic manifold on resolution}). The Floer theory of $(Y,\omega)$ comes into play, as the diagonal $\C^*$-action on $\C^n/G$ lifts to $Y$ (this relies on $Y$ being crepant \cite[Prop.8.2]{Batyrev2}, we will give a self-contained proof in \fix{Proposition \ref{proposition lift of action}}). The underlying $S^1$-action is Hamiltonian, corresponding to the standard Hamiltonian $\frac{1}{2}|z|^2$ on $\C^n$ away from a neighbourhood of $E$. We use this Hamiltonian, rescaled by large constants, to define Floer cohomology groups of $Y$ and their direct limit, symplectic cohomology $SH^*(Y)$.

Loops in $Y\setminus E$ are naturally labelled by $\mathrm{Conj}(G)$ via their free homotopy class,\footnote{If the Floer solutions $\R\times S^1 \to Y$ counted by the Floer differential did not intersect $E$, then the Floer differential would preserve these conjugacy classes. But this assumption is most likely false.}
\begin{equation}
\label{Equation pi1 of Y minus E}
[S^1,Y\setminus E]=\pi_0(\mathcal{L}(Y\setminus E))\cong \mathrm{Conj}(G).
\end{equation}
This follows from an analogous statement for based loops: $\pi_1(Y\setminus E) \cong \pi_1((\C^n\setminus 0)/G)\cong \pi_1(S^{2n-1}/G)\cong G$.
An analogous isomorphism holds also in the non-isolated case.\footnote{
Namely, $\pi_1(Y\setminus E) \cong G$ where $E=\pi^{-1}(\mathrm{Sing}(\C^n/G))$. Indeed, let $F_2$ denote the union of all $\mathrm{codim}_{\C}\geq 2$ fixed point loci in $\C^n$ of all subgroups of $G$. Then any non-identity element in $G$ fixing a point in $\C^n\setminus F_2$ must be a quasi-reflection, but there are no quasi-reflections in a finite subgroup $G\subset SL(n,\C)$. So $G$ acts freely on $\C^n\setminus F_2$. Thus $\pi_1((\C^n\setminus F_2)/G)\cong G$ (note that $\C^n\setminus F_2$ is simply connected due to the codimension of $F_2$). Finally $Y\setminus E \cong (\C^n\setminus F_2)/G$ via $\pi$, using \eqref{Equation singular set}. On the other hand, $\pi_1(Y)=1$ is a general feature of resolutions of quotient singularities \cite[Theorem 7.8]{Kollar}.}
\begin{lemma}\label{Lemma primitive orbits of the S1 action}
Any eigenvector $v\in \C^n\setminus 0$ of $g\in G$ yields a closed orbit $x_g$ of the $S^1$-action, corresponding to ${\bf g}=[g]\in \mathrm{Conj}(G)$ via \eqref{Equation pi1 of Y minus E}. Namely, if  $g(v) = e^{i\ell}v$ for $0< \ell\leq  2\pi$,
\begin{equation}\label{Equation xg orbit corresponding to v}
x_g(t)=e^{i\ell t}v: S^1 \to (\C^n\setminus 0)/G\cong Y\setminus E, \textrm{ where } x_g(0)=[v]=[gv]=[e^{i\ell}v]=x_g(1).
\end{equation}
 Conversely, ${\bf g}\in \mathrm{Conj}(G)$ can be uniquely recovered from $[v]$ \fix{and the eigenvalue $e^{i\ell}$.}
\end{lemma}
\begin{proof}
We check that $[v]$ determines ${\bf g}\in \mathrm{Conj}(G)$. If $g_1\neq g_2\in G$ with $g_1(v) = \fix{ e^{i\ell}v} = g_2(v)$, then $g_2^{-1}g_1\in \mathrm{Stab}(v)$ implies $v$ is singular by \eqref{Equation singular set}, yielding the contradiction $v=0$ (the isolated singularity). Conjugation $g\mapsto hgh^{-1}$ corresponds to changing eigenvectors by $v\mapsto h(v)$.
\end{proof}

Recall that the eigenvalue $e^{i\ell}$ above contributes $\frac{1}{2\pi}a = 1-\frac{\ell}{2\pi}$ to the $\mathrm{age}(g)$ in \eqref{Equation age}.
Given ${\bf g}\in \mathrm{Conj}(G)$, call $e^{i\ell}$ a {\bf minimal eigenvalue} of ${\bf g}$ if $0<\ell\leq 2\pi$ achieves the minimal possible value amongst eigenvalues of ${\bf g}$. 
If $v\in S^{2n-1}$ satisfies $g(v)=e^{i\ell}v$, and $e^{i\ell}$ is minimal, then we call $x_g(t)=[e^{i\ell t}v]\in S^{2n-1}/G$ a {\bf minimal Reeb orbit}.

To simplify our outline, let us assume that we are using the quadratic radial Hamiltonian $H=\tfrac{1}{4}R^2$ 
%
%
%
to define Floer cohomology, where $R=|z|^2$ on $\C^n/G$, and that this agrees via $\pi$ with the Hamiltonian used on $Y\setminus E$ (Section \ref{Subsection Vanishing of the symplectic cohomology of crepant resolutions} discusses these details). This implies that the time-$t$ Hamiltonian flow on $Y\setminus E$ equals multiplication by $e^{i R t}$ on each {\bf slice}\footnote{$\mathcal{S}(R)\cong S^{2n-1}/G$ is an $S^1$-equivariant isomorphism, as the region $R>0$ avoids the isolated singularity.}
\begin{equation}
\label{Equation slice S}
\mathcal{S}(R)=\pi^{-1}\{[z]\in \C^n/G: |z|^2= R\}\cong S^{2n-1}/G.
\end{equation}
Under this identification, each $1$-periodic Hamiltonian orbit $y:S^1 \to Y\setminus E$ corresponds uniquely to a {\bf Reeb orbit} $x_{{\bf g}} \subset S^{2n-1}/G$ satisfying \eqref{Equation xg orbit corresponding to v} (where we also allow $\ell\geq 2\pi$), and ${\bf g}\in \mathrm{Conj}(G)$ is the class that $y$ determines via \eqref{Equation pi1 of Y minus E}. Fixing ${\bf g}$ and $\ell\in (0,\infty)$ defines
\begin{equation}
\label{Equation moduli space Oa of Ham orbits}
\mathcal{O}_{{\bf g},\ell}=\{ \textrm{parametrized Hamiltonian }1\textrm{-orbits in }\mathcal{S}(\ell)\textrm{ in the class }{\bf g}\} \subset S^{2n-1}/G,
\end{equation}
where the ``inclusion" is defined by taking the initial point of the corresponding Reeb orbit, so $y\mapsto x(0)$. Although the $\mathcal{O}_{{\bf g},\ell+2\pi k}$ yield the same subset of $S^{2n-1}/G$ via \eqref{Equation slice S} for each $k\in \N$, they consist of $1$-orbits in $Y\setminus E$ arising in different slices. So, loosely, $\N$ copies of $\mathcal{O}_{{\bf g},\ell}\subset S^{2n-1}/G$ for $0< \ell \leq 2\pi$ contribute to the symplectic chain complex $SC^*(Y)$.

\begin{example} Continuing the Example $G=\Z/2$, $Y=T^*\C\P^1$, the chain complex $SC^*(Y)$ is
$$
\xymatrix@C=10pt@R=10pt{ 
\textcolor{black}{\bf (2)} & & \boxed{\textcolor{blue}{\bf +1}} \ar@{-->}^-{}[ll] & \boxed{\textcolor{red}{\bf -1}}
\ar@/^0.4pc/@{-->}^-{}[dlll]
 & \textcolor{blue}{\bf -3} \ar@{->}^-{}[dll] & \textcolor{red}{\bf -5} \ar@{->}^-{}[dll] & \textcolor{blue}{\bf -7} \ar@{->}^-{}[dll] & \textcolor{red}{\bf -9} \ar@{->}^-{}[dll] & \textcolor{blue}{\bf -11} \ar@{->}^-{}[dll]  & \ldots \ar@{->}_-{\ldots}[dll] 
\\
\textcolor{black}{\bf(0)} & & \textcolor{blue}{\bf -2} & \textcolor{red}{\bf -4} & \textcolor{blue}{\bf -6} & \textcolor{red}{\bf -8} & \textcolor{blue}{\bf -10} & \textcolor{red}{\bf -12} & \textcolor{blue}{\bf -14}  & \ldots
%
%
\\ 
}
$$
Those numbers are the Conley-Zehnder indices\footnote{Which is a $\Z$-grading on $SC^*(Y)$ as $Y$ is Calabi-Yau by the triviality of $K_Y$.} of the orbits (Appendix C). 
The ``zero-th column'' is a Morse complex for $Y$ and computes $H^*(\C\P^1)$ (this arises from constant orbits).
The other columns are the local Floer contributions of $\mathcal{O}_{-I,\pi},\mathcal{O}_{+I,2\pi}$, $\mathcal{O}_{-I,3\pi}$, etc.\;Each $\mathcal{O}_{{\bf g},\ell}$ equals $\mathcal{S}(\ell) \cong S^3/G\cong \R\P^3$ as any point in the slice yields an orbit. For \textcolor{red}{\bf even multiples} of $\pi$ the orbits lift to iterates of great circles in $S^3$, for \textcolor{blue}{\bf odd multiples} they lift to non-closed orbits in $S^3$ travelling that odd number of half-great circles. These correspond to $\textcolor{red}{\bf +1}$ and $\textcolor{blue}{\bf -1}$ eigenvectors in $\C^2$, for $\textcolor{red}{\bf g=+I}$ and $\textcolor{blue}{\bf g=-I}$, and (disregarding the zero-th column) they arise in the \textcolor{red}{\bf even columns} and the \textcolor{blue}{\bf odd columns}. Using a Morse-Bott model (Appendix E) each column is a copy of $H^*(\mathcal{O}_{{\bf g},\ell})\cong H^*(\R\P^3,\C)$ with grading suitably shifted, so two generators separated by 3 in grading. The jump by $4=2n$ in grading every two columns is due to a full rotation of $\varphi_t^*K_Y$ along the $S^1$-action $\varphi_t$ compared to the standard trivialisation of $K_{\C^n}$. 
\end{example}

Loosely, the positive complex $SC^*_+(Y)$ is the quotient of $SC^*(Y)$ by the Morse subcomplex of constant $1$-periodic Hamiltonian orbits, which appear in $E=\pi^{-1}(0)\subset Y$. The Morse subcomplex computes $H^*(Y)$ and thus gives rise to the long exact sequence 
\begin{equation}\label{Equation LES for SH and SH+}
\cdots \to H^*(Y) \stackrel{c^*}{\to} SH^*(Y) \to SH^*_+(Y) \to H^{*+1}(Y) \to \cdots
\end{equation}
This construction is known for exact convex symplectic manifolds \cite{Viterbo,Bourgeois-Oancea}, in which case the Floer action functional $A_H:\mathcal{L}Y \to \R$ provides the necessary filtration to make the argument rigorous. As our K\"{a}hler form $\omega$ on $Y$ is non-exact (so $A_H$ becomes multi-valued), we construct a novel filtration in Appendix D in order achieve the same result for any convex symplectic manifold. Then $SC^*_+(Y)$ is generated by the union $\cup \mathcal{O}_{{\bf g},\ell}$ over ${\bf g}\in \mathrm{Conj}(G)$, $\ell\in (0,\infty)$.

\begin{example} Continuing above, $SH^*(Y)=SH^*(\mathcal{O}_{\C\P^1}(-2))=0$ by \cite{Ritter4}, so the chain complex is acyclic and all arrows are isomorphisms.\footnote{Those orbits are also the generators when using the canonical exact symplectic form on $T^*S^2$, and the grading is consistent with the Viterbo isomorphism $SH^*(T^*S^2) \cong H_{2-*}(\mathcal{L}S^2)$ provided differentials vanish.}
In $SC^*_+(Y)$, the zero-th column is quotiented, so the two boxed generators survive to $SH^*_+(Y)$: they are the maxima of the first two Morse-Bott manifolds of $1$-orbits; in $S^3$ they become a half-great circle and a great circle. Working over $\C$\fix{, and using the notation $A[d]$ to mean $A$ with grading shifted down by $d$ for any $\Z$-graded group $A=\oplus A_m$, so $(A[d])_m = A_{m+d}$, we deduce that}
\begin{align*}
SH^{*-1}_+(Y)=\C[\textcolor{blue}{\bf -1}][-1]
\oplus \C[\textcolor{red}{\bf +1}][-1] 
\cong
\textcolor{blue}{\C[-2]}\oplus \textcolor{red}{\C}
\cong
 \textcolor{blue}{H^2(Y,\C)}
\oplus \textcolor{red}{H^0(Y,\C)} \cong H^*(Y,\C).
\end{align*}
\end{example}

\begin{theorem}\label{Theorem SH is zero introduction}
$SH^*(Y)=0$, so there is a canonical isomorphism 
$$SH^{*-1}_+(Y) \to H^{*}(Y).$$
\end{theorem}

That vanishing follows by mimicking the argument in \cite{Ritter2} (analogously to $\C^n$ \cite[Sec.3]{OanceaEnsaios}): the Hamiltonians $L_k=k R$ 
%
yield the flow $e^{ikt}$, and for generic \fix{$k\in \R_{>0}$} the only period $1$ orbits are constant orbits in $E$ (Lemma \ref{Lemma generic slope gives constant orbits}) whose Conley-Zehnder index becomes unbounded as $k\to \infty$ (Theorem \ref{Theorem SH=0}). As $SH^*(Y)$ is the direct limit of $HF^*(L_k)$ under grading-preserving maps, in any given finite degree no generators appear for large $k$.

\begin{theorem}\label{Theorem rank of SH+}
$SH^*_+(Y)$ has rank $|\mathrm{Conj}(G)|$. More precisely,
$$\mathrm{rank}\, SH^{2k-1}_+(Y) = |\mathrm{Conj}_k(G)|.$$
%
\end{theorem}

The remainder of this Section will explain the proof of the above \fix{theorem}. A Morse-Bott argument (Appendix E) yields a convergent spectral sequence
\begin{equation}\label{Equation Morse-Bott sp seq}
E_1^{*,*}=\bigoplus H^*(\mathcal{O}_{{\bf g},\ell})[-\mu_{{\bf g},\ell}] \Rightarrow SH_+^*(Y),
\end{equation}
where $\mu_{{\bf g},\ell}$ denotes the shift in grading that needs to be applied to $H^*(\mathcal{O}_{{\bf g},\ell})$ in the Morse-Bott model for the symplectic chain complex ($\mu_{{\bf g},\ell}$ represents the grading of the orbit corresponding to the minimum of $\mathcal{O}_{{\bf g},\ell}$). It turns out that $\mu_{{\bf g},\ell}$ is always an even integer (Equation \eqref{Equation mu B of Morse Bott submfds}).

In particular, in the Morse-Bott model, each ${\bf g}\in \mathrm{Conj}(G)$ gives rise to a maximum $x_g$ (i.e.\,\,top degree generator) of $H^*(\mathcal{O}_{{\bf g},\ell})$, which is an orbit associated to the minimal eigenvalue $e^{i\ell}$ of ${\bf g}$, and its Conley-Zehnder index $\mu(x_g)$ satisfies
\begin{equation}\label{Equation grading of the max of Ogl}
\tfrac{1}{2}(\mu(x_g)+1)=\mathrm{age}(g).
\end{equation}

As the Example illustrated, without knowing $H^*(Y)$ it would be difficult to predict which generators survive in the limit of the spectral sequence \eqref{Equation Morse-Bott sp seq}. \fix{We can} however compute the $S^1$-equivariant analogue of \eqref{Equation Morse-Bott sp seq}, 
because in that case all generators on the $E_1$-page have odd total degree so
the spectral sequence degenerates on that page. Our goal is to recover $SH^*_+(Y)$ from this fact. 

The $S^1$-equivariant theory $ESH^*=SH^*_{S^1}$ was defined by Seidel \cite[Sec.(8b)]{Seidel}. It was constructed in detail by Bourgeois-Oancea \cite{BourgeoisOanceaGysin} and we review it in Appendix B. 
In the equivariant setup, the generators involve the moduli spaces of unparametrized orbits $\mathcal{O}_{{\bf g},\ell}/S^1$ and we prove in Theorem \ref{Theorem equiv coh of Morse Bott submanifolds} that  these can be identified with $\P_{\C}(V_{g,\ell})/G_{g,\ell}$, where we projectivise the $e^{i\ell}$-eigenspace $V_{g,\ell}$ of $g$, and $G_{g,\ell}\subset G$ is the largest subgroup which maps $V_{g,\ell}$ to itself (in fact $G_{g,\ell}=C_G(g)$ is the centraliser, by Lemma \ref{Lemma cyclic groups Gv}). If the characteristic of the underlying field does not divide $|G|$, that quotient by $G_{g,\ell}$ does not affect cohomology (Remark \ref{Remark cohomology of finite quotients}), so
$$
H^*(\mathcal{O}_{{\bf g},\ell}/S^1)\cong H^*(\P_{\C}(V_{g,\ell}))\cong H^*(\C \P^{\dim_{\C}\! V_{g,\ell}-1}).
$$
Up to an additional grading shift by one, which we will explain later, \fix{when working over a field of characteristic zero} we deduce that the $S^1$-equivariant spectral sequence has generators in odd degrees as claimed. \fix{The case of positive characteristic is discussed in Remark \ref{Remark Coefficients}.}

\begin{example}
%
%
Continuing the above example, the equivariant complex $ESC^*_+(Y)$  becomes
$$
\xymatrix@C=10pt@R=0pt{ 
 & & \boxed{\textcolor{blue}{\bf +1}} & \boxed{\textcolor{red}{\bf -1}}
 & \textcolor{blue}{\bf -3}  & \textcolor{red}{\bf -5}  & \textcolor{blue}{\bf -7}  & \textcolor{red}{\bf -9}  & \textcolor{blue}{\bf -11}   & \ldots
 \\
  & & \textcolor{blue}{\bf -1} & \textcolor{red}{\bf -3} & \textcolor{blue}{\bf -5} & \textcolor{red}{\bf -7} & \textcolor{blue}{\bf -9} & \textcolor{red}{\bf -11} & \textcolor{blue}{\bf -13}  & \ldots  
}
$$ 
where each column is a shifted copy of $H^{*}(\R\P^3/S^1)\cong H^{*}(\C\P^1)$, for example the second column is shifted up by $1+\mu_{I,2\pi}=-3$.
\end{example}

Before we can continue the outline, we need a remark about the coefficients used in the equivariant theory.
In the Calabi-Yau setup ($c_1(M)=0$), symplectic cohomology $SH^*(M)$ is a $\Z$-graded $\K$-module, where $\K$ is the {\bf Novikov field},
\begin{equation}\label{Equation Novikov field}
\K=\left\{\sum_{j=0}^{\infty} n_j T^{a_j}: a_j\in \R, a_j\to \infty, n_j\in \mathcal{K}\right\}.
\end{equation}
 Here $\mathcal{K}$ is any given field of characteristic zero (we discuss non-zero characteristics later), and $T$ is a formal variable in grading zero. 
Let $\K(\!(u)\!)$ denote the formal Laurent series in $u$ with coefficients in $\K$, where $u$ has degree $2$, and abbreviate by $\F$ the $\K[\![u]\!]$-module 
$$
\F = \K(\!(u)\!)/u\K[\![u]\!] \cong H_{-*}(\C\P^{\infty}).
$$
Here $u^{-j}$ in degree $-2j$ formally represents $[\C\P^j]\in H_{-*}(\C\P^{\infty})$ negatively graded, and the $\K[\![u]\!]$-action is induced by the (nilpotent) cap product action by $H^*(\C\P^{\infty})=\K[u]$.

In Appendix B, we construct the $S^1$-equivariant symplectic cohomology as a 
$\K[\![u]\!]$-module $ESH^*(Y)$ together with a canonical $\K[\![u]\!]$-module homomorphism
\begin{align*}
c^*:EH^*(Y)\cong H^*(Y)\otimes_{\K} \F \to ESH^*(Y),
\end{align*}
where in general $EH^*(Y)$ denotes the locally finite $S^1$-equivariant homology $H_{2n-*}^{\mathrm{lf},S^1}(Y)$, not $H^*_{S^1}(Y)$. It becomes $H^*(Y)\otimes_{\K}\F$ above, as the $S^1$-action is trivial on constant orbits. 
\fix{Similar to} Theorem \ref{Theorem SH is zero introduction},\footnote{The vanishing follows by the spectral sequence for the $u$-adic filtration (see \eqref{Equation spectral sequence u adic filtration}), and \eqref{Equation ESH+ is ordinary EH} follows by the equivariant analogue of \eqref{Equation LES for SH and SH+} (Corollary \ref{Corollary LES for H SH and SHplus}).} \fix{we have} $ESH^*(Y)=0$ and there is a canonical $\K[\![u]\!]$-module isomorphism
\begin{equation}\label{Equation ESH+ is ordinary EH}
ESH^{*-1}_+(Y)\cong EH^*(Y) \cong H^*(Y)\otimes_{\K} \F.
\end{equation}
This implies that the $\K[\![u]\!]$-module $ESH^{*}_+(Y)$ is in fact a free $\F$-module, \fix{and we will see in the proof of Corollary \ref{Corollary computation of ESH+} that its rank equals 
the Euler characteristic of $Y$,
$$|\mathrm{Conj}(G)|=\mathrm{rank}_{\F}\, ESH^*_+(Y)=\dim_{\K} ESH^{-1}_+(Y)=\sum \dim_{\K} H^{2j}(Y) = \chi(Y).$$}%
As anticipated previously, the equivariant analogue of \eqref{Equation Morse-Bott sp seq}, working in characteristic zero, yields the following isomorphism of $\K$-vector spaces (but not as $\K[\![u]\!]$-modules)
$$
ESH^*_+(Y)\cong \oplus EH^*(\mathcal{O}_{{\bf g},\ell})[-\mu_{{\bf g},\ell}]
\cong \oplus H^{*}(\mathcal{O}_{{\bf g},\ell}/S^1)[-1-\mu_{{\bf g},\ell}]\cong \oplus H^{*}(\C \P^{\dim_{\C}\! V_{g,\ell}-1})[-1-\mu_{{\bf g},\ell}]
$$
where we now explain the second isomorphism.
\fix{For any closed orientable manifold} $X$ with an $S^1$-action with finite stabilisers, and working in characteristic zero, $EH^*(X)\cong H^{*-1}(X/S^1)$ as $\K[\![u]\!]$-modules, with $u$ acting by \fix{cup product} with the negative of the Euler class of $X\to X/S^1$ (Theorem \ref{Theorem EH for free action}). In our case, the $u$-action on $H^{*}(\C \P^{\dim_{\C}\! V_{g,\ell}-1})$ \fix{is cup product by $-\mathrm{PD}[H]$, where $H$ is the hyperplane class}.

Whilst the usual symplectic chain complex is generated by $1$-orbits%
\footnote{Strictly, in Floer theory one must pick a reference loop for each $1$-orbit $x:[0,1]\to Y$, as the action functional is multi-valued. In our setup, $Y$ is simply connected \cite[Theorem 7.8]{Kollar}, so one can just pick a smooth filling disc $\widetilde{x}:\mathbb{D}\to Y$, $\partial \widetilde{x}=x$. Two choices of filling disc $\widetilde{x}_1,\widetilde{x}_2$ differ by a sphere $S=[\widetilde{x}_1\#-\widetilde{x}_2]\in H_2(Y)$ and one identifies \fix{$T^{\omega(S)} \widetilde{x}_2 = \widetilde{x}_1$. As $T$ has} grading zero (as $c_1(Y)=0$), these choices do not matter. A canonical choice of filling disc for $x_g$, for an eigenvalue $e^{i\ell}$ of $g\in G$, is obtained by applying the action of 
$\{re^{it}: 0\leq r\leq 1, 0\leq t \leq \ell \}\subset \C^*$ to $x_g(0)$ to define a map $\D \to Y$.}
over a Novikov field $\K$, the equivariant theory is generated by $1$-orbits over the $\K[\![u]\!]$-module $\F$. There is a natural inclusion $SC^*_+(Y)\to ESC^*_+(Y)$ as the $u^0$-part, so a $\K[d]$ summand in $SC^*_+(Y)$ belongs to a copy of $\F[d]=\K[d]\oplus \K[d+2] \oplus \K[d+4] \oplus \cdots$ in $ESC^*_+(Y)$ where the $u$-action translates the copies $\K[d+2j]\to \K[d+2j-2]$. These summands may however unexpectedly disappear in cohomology. We proved above that \fix{$ESH_+^*(Y)$} is a free $\F$-module, so the $\K$-summands appearing in $\fix{ESH^*_+(Y)}$ must organise themselves into free $\F$-summands, in particular the number of $\F$-summands in odd degrees is simply the dimension $\dim_{\K} \fix{ESH^{-2m-1}_+(Y)}$ for sufficiently large $m$ (these dimensions must stabilize). By considering the gradings of the $\mathcal{O}_{{\bf g},\ell}$, we check explicitly in Corollary \ref{Corollary computation of ESH+} that there are $|\mathrm{Conj}(G)|$ free $\F$-summands in $ESH^*_+(Y)$.

\begin{example}
%
%
Continuing the above example, the $ESC^*_+(Y)$ complex must have two $\F$-summands in $ESH^*_+(Y)$, generated by the two orbits labelled by $\pm I\in G=\Z/2$,
\begin{align*} 
ESH^*_{+,\textcolor{blue}{\bf -I}}(Y) = \K[\textcolor{blue}{\bf -1}]\oplus \K[1]\oplus \K[3] \oplus \cdots \cong
\K[\textcolor{blue}{\bf -1}]\oplus \K[\textcolor{blue}{\bf -1}]u^{-1}\oplus \K[\textcolor{blue}{\bf -1}]u^{-2} \oplus \cdots
\cong \mathbb{F}[\textcolor{blue}{\bf -1}]
\\
ESH^*_{+,\textcolor{red}{\bf +I}}(Y) = \K[\textcolor{red}{\bf +1}]\oplus \K[3] \oplus \K[5] \oplus \cdots 
\cong
\K[\textcolor{red}{\bf +1}]\oplus \K[\textcolor{red}{\bf +1}]u^{-1} \oplus \K[\textcolor{red}{\bf +1}]u^{-2} \oplus \cdots 
\cong \mathbb{F}[\textcolor{red}{\bf +1}]
\end{align*} 
Thus we obtain one free $\F$-summand for each conjugacy class.
\end{example}

\begin{example}
In the case of $A_n$ surface singularities  
$\C^2/G$, so $G$ generated by $\left(\begin{smallmatrix} \zeta & 0 \\ 0 & \zeta^{-1} \end{smallmatrix}\right)$ where $\zeta=e^{2\pi i/(n+1)},$
the work of Abbrescia, Huq-Kuruvilla, Nelson and Sultani \cite{nelsonreebdynamics} is an independent computation of $ESH^*_+(Y)$, and indeed it has rank $n+1=|\mathrm{Conj}(G)|$ (their grading is by $\mathrm{CZ}$ whereas our grading is by $n-\mathrm{CZ}$, see Remark \ref{Remark from CZReeb to CZHam}).
\end{example}

To recover $SH^*_+(Y)$ from $ESH^*_+(Y)$ we use the Gysin sequence \cite{BourgeoisOanceaGysin} (see Appendix B)
\begin{equation}\label{Equation Gysin for SH+}
\cdots \to SH^*_+(Y) \stackrel{\mathrm{in}}{\longrightarrow} ESH^*_+(Y) \stackrel{u}{\longrightarrow} ESH^{*+2}_+(Y) \stackrel{b}{\longrightarrow} SH^{*+1}_{+}(Y)\to \cdots
\end{equation}
\fix{We work under the assumption that the characteristic of the field is coprime to $2,3,\ldots,|G|.$}
As $ESH^*_+(Y)$ lives in odd grading, the sequence splits as 
$$
0 \to SH^{\mathrm{odd}}_+(Y) \hookrightarrow ESH^{\mathrm{odd}}_+(Y) \stackrel{u}{\to} ESH^{\mathrm{odd}+2}_+(Y)\to SH^{\mathrm{odd}+1}_{+}(Y) \to 0.
$$
So $SH^{\mathrm{odd}}_+(Y)$ is the $u^0$-part of $ESH^{\mathrm{odd}}_+(Y)$, 
and $SH^{\mathrm{even}}_+(Y)=0$ since $u$ acts surjectively on $ESH^*_+(Y)$ as it is a free $\F$-module. This yields the $\K$-vector space isomorphism
$$
SH^*_+(Y)\, \cong \, \ker \, ( u: ESH^*_+(Y) \to ESH^*_+(Y) ).
$$
Theorem \ref{Theorem rank of SH+} now follows, since we showed that $ESH^*_+(Y)$ has $|\mathrm{Conj}(G)|$ free $\F$-summands. In particular, Corollary \ref{Corollary computation of ESH+} shows that the $u^0$-parts of those $\F$-summands in $ESH^*_+(Y)$ can be labelled by the maxima mentioned in \eqref{Equation grading of the max of Ogl} (the labelling is non-canonical, see Sec.\ref{Subsection Naturality of the basis}).%
%
%
%
\fix{
\begin{remark}[Coefficients]\label{Remark Coefficients} We showed $SH^{*-1}_+(Y)$ recovers the ordinary cohomology $H^{*}(Y,\K)$ over the Novikov field $\K$. But $\K$ is flat over the base field $\mathcal{K}$ (indeed free, a $\mathcal{K}$-vector space) so $H^*(Y,\K)\cong H^*(Y,\mathcal{K})\otimes_{\mathcal{K}} \K$ determines $H^*(Y,\mathcal{K})$.
Thus we recover the McKay correspondence over fields $\mathcal{K}$ of characteristic zero, in particular over $\Q$.
\\ \indent
Our proof also works when the characteristic is coprime to all integers $2,3,\ldots,|G|$.
The key idea is that the claim really only relies on understanding the spectral sequence for $ESH^*_+(Y)$ in odd degrees in $[-1,\dim_{\R} Y - 3]$. In particular the degree $-1$ part suffices to determine the rank of $ESH^*_+(Y)$ over $\F$, which is equal to both $|\mathrm{Conj}(G)|$ and $\chi(Y)$ (see Corollary \ref{Corollary computation of ESH+}).
 The assumption on the characteristic allows us to relate cohomologies of spaces before and after quotienting by finite groups which involve finite stabilisers, whose size is at most $|G|$ (Theorems \ref{Theorem equiv coh of Morse Bott submanifolds} and \ref{Theorem EH for free action}). 
\\ \indent
If $\mathcal{K}$ is any commutative Noetherian ring, 
%
%
the Novikov ring $\K$ is flat over $\mathcal{K}$. The obstruction to running the above proof is the failure of the isomorphism $EH^*(\mathcal{O}_{{\bf g},\ell})\cong H^{*-1}(\P V_{g,\ell})^{G_{g,\ell}}$. If this fails, there may be unexpected contributions in the Floer cohomology.
%
%
\end{remark}
}
Combining Theorems \ref{Theorem SH is zero introduction} and \ref{Theorem rank of SH+}, we deduce the McKay Correspondence:

\begin{corollary}[Generalised McKay Correspondence]\label{Corollary McKay via Floer}
Let $Y$ be any quasi-projective crepant resolution of an isolated singularity $\C^n/G$, where $G\subset SL(n,\C)$ is a finite subgroup. Let $\mathcal{K}$ be any field of characteristic zero, or assume $\mathrm{char}\, \mathcal{K}$ is coprime to all integers\,\,$\leq |G|$. Then $H^*(Y,\mathcal{K})$ vanishes in odd degrees and has rank $|\mathrm{Conj}_k(G)|$ in even degrees $2k$.
\end{corollary}
%
%
%
%
%
%
%
%
\begin{example}
For $X=\C^3/G$, with $G=\Z/3$ acting diagonally by powers of $\zeta=e^{2\pi i/3}$. The blow up $Y$ of $X$ at $0$ is a crepant resolution with exceptional locus $E\cong \C\P^2$. The Morse-Bott submanifolds $\textcolor{blue}{\mathcal{O}_{\zeta,2\pi/3}},\textcolor{red}{\mathcal{O}_{\zeta^2,4\pi/3}}$, $\textcolor{black}{\mathcal{O}_{I,2\pi}},\textcolor{blue}{\mathcal{O}_{\zeta,8\pi/3}},$ etc. are copies of $S^5/G$.
Following analogous notation as in the example of $T^*\C\P^1$, the acylic chain complex $SC^*(Y)$ over $\mathcal{K}=\C$ is:
$$
\xymatrix@C=10pt@R=0pt{ 
\textcolor{black}{\bf (4)} & & \boxed{\textcolor{blue}{\bf -1}} \ar@{-->}^-{}[ddll] & \boxed{\textcolor{red}{\bf +1}}
\ar@/^0.7pc/@{-->}^-{}[dlll]
 & \boxed{\textcolor{black}{\bf +3}} \ar@/^1.6pc/@{-->}^-{}[llll]  & \textcolor{blue}{\bf -7} \ar@{->}^-{}[ddlll] & \textcolor{red}{\bf -5} \ar@{->}^-{}[ddlll] & \textcolor{black}{\bf -3} \ar@{->}^-{}[ddlll] & \textcolor{blue}{\bf -13} \ar@{->}^-{}[ddlll]  & \ldots \ar@{->}^-{\ldots}[ddlll] 
\\
\textcolor{black}{\bf (2)}
\\
\textcolor{black}{\bf(0)} & & \textcolor{blue}{\bf -6} & \textcolor{red}{\bf -4} & \textcolor{black}{\bf -2} & \textcolor{blue}{\bf -12} & \textcolor{red}{\bf -10} & \textcolor{black}{\bf -8} & \textcolor{blue}{\bf -18}  & \ldots
%
%
\\ 
}
$$
where the generators in round brackets yield $H^*(Y)\cong H^*(\C\P^2)\cong \K\oplus \K[-2] \oplus \K[-4]$, and the boxed generators yield $SH^*_+(Y)\cong\textcolor{blue}{\bf \K[+1]}\oplus \textcolor{red}{\bf\K[-1]} \oplus \textcolor{black}{\bf\K[-3]}$. The columns in the Morse-Bott complex for $ESC_+^*(Y)$ are shifted copies of $H^*(\C\P^2,\K)$ instead of $H^*(S^5,\K)$:
$$
\xymatrix@C=10pt@R=0pt{ 
  & & \boxed{\textcolor{blue}{\bf -1}}  & \boxed{\textcolor{red}{\bf +1}}
 & \boxed{\textcolor{black}{\bf +3}}   & \textcolor{blue}{\bf -7}  & \textcolor{red}{\bf -5}  & \textcolor{black}{\bf -3}  & \textcolor{blue}{\bf -13}   & \ldots  
\\ \vspace{-5mm}
   & & \textcolor{blue}{\bf -3}  & \textcolor{red}{\bf -1}
 & \textcolor{black}{\bf +1}   & \textcolor{blue}{\bf -9}  & \textcolor{red}{\bf -7}  & \textcolor{black}{\bf -5}  & \textcolor{blue}{\bf -15}  & \ldots 
\\
  & & \textcolor{blue}{\bf -5} & \textcolor{red}{\bf -3} & \textcolor{black}{\bf -1} & \textcolor{blue}{\bf -11} & \textcolor{red}{\bf -9} & \textcolor{black}{\bf -7} & \textcolor{blue}{\bf -17}  & \ldots
%
%
}
$$
Thus $ESH^*_+(Y)\cong\textcolor{blue}{\bf \F[+1]}\oplus \textcolor{red}{\bf\F[-1]} \oplus \textcolor{black}{\bf\F[-3]}$. The three summands correspond to the three conjugacy classes of $G$. This also holds for any field $\mathcal{K}$ of characteristic coprime to $2$ and $3$. 
\end{example}
%
\subsection{Naturality of the basis}
\label{Subsection Naturality of the basis}
We return to the Open Problem at the end of Section \ref{Subsection The generalised McKay correspondence}. \fix{Assume for now that $\K$ has characteristic zero.}
In Corollary \ref{Corollary computation of ESH+} we showed that given ${\bf g}\in \mathrm{Conj}(G)$, the sum $F_{\bf g}=\bigoplus_{\ell>0} \fix{EH^*}(\mathcal{O}_{{\bf g},\ell})[-\mu_{{\bf g},\ell}]$ stacks together as a $\K$-vector space to yield a copy of the $\K$-vector space $\F$. This does not hold as $\K[\![u]\!]$-modules because the $E_1$-page of the Morse-Bott spectral sequence has forgotten the structure given by multiplication by $u$, yielding only a non-canonical $\K$-linear isomorphism 
$$\bigoplus_{\bf g \in \mathrm{Conj}(G)}F_{\bf g}\cong ESH^*_+(Y)\cong H^{*+1}(Y)\otimes_{\K} \F,$$
using \eqref{Equation ESH+ is ordinary EH}. We conjecture that each maximum $x_g$ mentioned in \eqref{Equation grading of the max of Ogl} gives rise to a generator $[x_g+c_g]\in SH^*_+(Y)$, where $c_g$ is a ``correction term''\footnote{the correction term is not unique, 
%
%
and arises because the $E_1^{*,*}$ page for the equivariant Morse-Bott spectral sequence is isomorphic to the associated graded algebra of $ESH^*_+(Y)$.} with strictly higher $F$-filtration value  than $x_g$ (in the sense of Appendix E). Theorem \ref{Theorem SH is zero introduction} would then yield generators of $H^*(Y)$ labelled by $\mathrm{Conj}(G)$, respecting \eqref{Equation grading of the max of Ogl}. Nevertheless, the correction terms are not canonical.

\fix{When $\K$ has positive characteristic (coprime to $2,3,\ldots,|G|$), the stacking mentioned for $F_{\bf g}$ only holds up to possible torsion summands (by the universal coefficient theorem), but such torsion must eventually cancel out in the spectral sequence as $ESH^*_+(Y)$ is free over $\F$. It is plausible that the Conjecture would persist to hold in positive characteristic.
}
 
Our modified approach \cite{McLean-Ritter} mentioned in Section \ref{Subsection Isolated singularities} is expected to yield a clean naturality statement, in addition to discussing product structures and extending the results to the case of non-isolated singularities.

\fix{The map from \eqref{Equation LES for SH and SH+},
\begin{equation}\label{Equation SH+ to H}
SH^{*-1}_+(Y) \to H^*(Y)\cong H_{2\dim_{\C}Y-*}^{\mathrm{lf}}(Y),
\end{equation}
 is given by applying the Floer boundary operator of the full $SC^*(Y)$. Once we extend our work to non-isolated singularities in \cite{McLean-Ritter}, it would be interesting to investigate how our basis compares via \eqref{Equation SH+ to H} to the bases built for $H^*(Y)$ by Ito-Reid \cite{Ito-Reid} for $n=3$, and by Kaledin \cite{Kaledin} for $n=2m$ and $G\subset \mathrm{Sp}(2m;\C)\cap U(2m)$.
}

\begin{remark}
In the work of Koll{\'a}r-N{\'e}methi \cite[Corollary 29]{kollarnemethi} a natural bijection arose between the conjugacy classes of $G$ and the irreducible components of the space $\mathrm{ShArc}(0\in X)$ of \emph{short complex analytic arcs} $\overline{\mathbb{D}}\to X=\C^n/G$ which hit the singular point only at $0\in \overline{\mathbb{D}}$, where $G\subset GL(n,\C)$ is any finite subgroup acting freely on $\C^n\setminus \{0\}$. This in turn gives rise to a natural map \cite[Paragraph 27]{kollarnemethi} from conjugacy classes to subvarieties in any resolution $Y$ of $X$, by considering the subsets swept out by the lifts of the arcs under evaluation at $0\in \overline{\mathbb{D}}$. 
\fix{The operator in \eqref{Equation SH+ to H}} counts finite energy Floer cylinders $u: \R \times S^1 \to Y$ converging to a Hamiltonian $1$-orbit at the positive end. Such maps have a removable singularity at the negative end,
%
%
and yield an extension $u: \C \to Y$ with $u(0)\in E$. Evaluation at $0$ sweeps the required locally finite pseudo-cycle in $H_*^{\mathrm{lf}}(Y)$. As $u$ is asymptotically holomorphic near $0\in \C$, the projection to $\C^n/G$ should approximate an analytic arc through $0$. A possible approach to obtain genuine analytic arcs, would be to first perform a neck-stretching argument in the sense of Bourgeois-Oancea \cite{Bourgeois-Oancea} so that (the main component of the) Floer solution converges to a holomorphic map $u:\C \to Y$ that is asymptotic to a Reeb orbit at infinity. 
It would be interesting to investigate more closely the relationship between these two points of view.
\end{remark}

\begin{remark}
Abreu-Macarini \fix{\cite[Theorem 1.12]{AbreuMacarini}} proved that the \emph{mean Euler characteristic} $\chi(M,\xi)$ of a Gorenstein toric contact manifold $(M,\xi)$ is equal to half of the Euler characteristic of any crepant toric symplectic filling $Y$. Recall $\chi(M,\xi)={\displaystyle \lim_{k\to \infty}} \frac{1}{2k} \sum_{j=0}^k \dim HC_{2j}(M,\xi)$ is defined in terms of the linearised contact homology, cf.\,Ginzburg-G\"{o}ren \cite{GinzburgGoren}. By Bourgeois-Oancea \cite{Bourgeois-Oancea-S1}, this homology is isomorphic to the positive $S^1$-equivariant symplectic homology. Our Corollary \ref{Corollary McKay via Floer} implies that $\chi(M,\xi)$ is half of the number of $\F$-summands in $ESH^{*-1}_+(Y)\cong H^*(Y)\otimes_{\K} \F$, thus it yields an alternative perspective of the result of Abreu-Macarini.
\end{remark}

\section{Proofs}
\subsection{Symplectic description of quotient singularities}
\label{Subsection symplectic data for the quotient}
Let $X=\C^n/G$ for any finite subgroup $G\subset SU(n)$ acting freely on $\C^n\setminus 0$, and recall Lemma \ref{Lemma G acts freely}. Viewed as a convex symplectic manifold (in the sense of Sec.\ref{Subsection convex symplectic manifolds}), $\C^n$ has data
$$\textstyle
\omega=\sum dx_j\wedge dy_j, \qquad \theta = \tfrac{1}{2}\sum  x_j dy_j - y_j dx_j, \qquad Z = \tfrac{1}{2}\sum x_j \partial_{x_j} + y_j \partial_{y_j}=\tfrac{1}{4}\nabla R,\qquad R = |z|^2 
$$
in coordinates $z_j=x_j+i y_j$. 
As $G\subset SU(n)$ preserves $R$ and the metric, it preserves all of the above data, 
%
%
so that descends to corresponding data $\omega_G$, $\theta_G$, $Z_G$, $R_G$ on $(X\setminus 0)/G$. Call $\pi_G: S^{2n-1}\to S_G=S^{2n-1}/G$ the induced quotient map on the unit sphere $S^{2n-1}\subset \C^n$ (note $S_G$ is smooth as $G$ acts freely). Using terminology from Appendix C, the {\bf round contact form} $\alpha_0=\theta|_{S^{2n-1}}$ yields a contact form $\alpha_G=\theta_G|_{S_G}$ on $S_G$, and they define contact structures $\xi=\ker \alpha_0$ on $S^{2n-1}$, $\xi_G=\ker \alpha_G$ on $S_G$.
The Reeb flow $\phi_{2t}$ for $\alpha_0$ descends to the Reeb flow on $S_G$ for $\alpha_G$, where
\begin{equation}\label{Equation Reeb flow for sphere}
 \phi_{t} : S^{2n-1} \to S^{2n-1}, \quad \phi_{t}(z) = e^{it} z.
\end{equation}
%
\begin{remark} We will from now on refer to $\phi_t$ as the {\bf Reeb flow} (rather than $\phi_{2t}$) so the periods/lengths of Reeb orbits we will refer to are, strictly, the double of their actual values.
\end{remark}
\subsection{The closed Reeb orbits in the quotient}
\label{Subsection The closed Reeb orbits in the quotient}
Let $g\in G$ and let $V$ be the $e^{i\ell}$-eigenspace of $g$ for some given $\ell>0\in \R$,
$$V=V_{g,\ell}=\{v\in \C^n: g(v)=e^{i\ell}v\}.$$
%
%
By Lemma \ref{Lemma primitive orbits of the S1 action}, $h(V)=V_{hgh^{-1},\,\ell}$ is the $e^{i\ell}$-eigenspace of $hgh^{-1}$. We sometimes abusively write $\dim_{\C} V_{\mathbf{g},\ell}$ for a class ${\bf g}=[g]\in \mathrm{Conj}(G)$, since the dimension does not depend on the choice of representative $g$.
Let $\P(V)$ be the complex projectivisation. Define:
$$
\begin{array}{rcl}
G_v &=& \{ h\in G: v \textrm{ is an eigenvector of } h\} \qquad (\textrm{where }v\in \C^n\setminus 0).\\[1mm]
G_V &=& \{h\in G: h(V)\subset V\}.
\\[1mm]
G_{g,\ell} &=& {\displaystyle \bigcup_{v\in V\setminus 0}} G_v = \{h\in G: V_{h,\ell'} \cap V \neq \{0\}  \textrm{ for some }\ell'\in \R\}.
\end{array}
$$ 
Observe that $G_{h(v)}=hG_v h^{-1}$, so we sometimes abusively write $|G_p|$ for $p=[v]\in (\C^n\setminus 0)/G$ as the size of the subgroup $G_v$ does not depend on the choice of representative $v$.
\begin{lemma}\label{Lemma cyclic groups Gv}\strut
\begin{enumerate}
\item \label{Item 1 CGg} $G_v\subset G$ is a cyclic subgroup of size $|G_v|=|\{\lambda\in S^1: h(v)=\lambda v \textrm{ for some }h\in G\}|$.
\item \label{Item 3 CGg} $\{h\in G: V\cap h(V) \neq \{0\}\}=C_G(g)$ recovers the centraliser of $g$.
\item \label{Item 2 CGg}\label{Item 4 CCg} $G_{g,\ell}=G_V=C_G(g)$. 
\item \label{Item 5 CCg} $G_V$ acts on $\P(V)$ with stabilisers $\mathrm{Stab}_{G_V}([v]) = G_v$, and the size $|C_G(g)|/|G_v|$ of the orbit of $[v]$ is the size of the fibre of $\P(V)\to \P(V)/G_V$.
\end{enumerate}
\end{lemma}
\begin{proof}
(1) Consider $\{\lambda\in S^1: h(v)=\lambda v \textrm{ for some }h\in G_v\}$. As this is a finite subgroup of $S^1$, it is cyclic. Pick a generator $\lambda$, associated to $h\in G_v$ say. Then for any $h'\in G_v$, there is a $k\in \N$ satisfying $h'(v)=\lambda^k v = h^k(v)$, which forces $h'=h^k$ since $G$ acts freely on $\C^n\setminus 0$.
\\
\indent (2) $g(hv)=e^{i\ell}hv$ implies $h^{-1}g h v =e^{i\ell}v= gv$, so $h^{-1}gh=g$ ($G$ acts freely). Conversely, if $h\in C_G(g)$, then $h,g$ have a common basis of eigenvectors, so $hv=\lambda v\in V\cap h(V)$.
\\
\indent (3) Let $h$ be the generator of $G_v$ from \eqref{Item 1 CGg}, so $g=h^k$ for some $k$. 
Thus $h$ and $g$ commute, so $h\in C_G(g)$.
%
Conversely, if $h\in C_G(g)$, then $h,g$ have a common basis of eigenvectors, and a subcollection will be a basis for $V$ consisting of eigenvectors of $h$. Thus $h\in G_v$ for any $v$ from this subcollection, and also $h(V)=V$ so $h\in G_V$. Finally \eqref{Item 3 CGg} implies
$G_V\subset C_G(g)$.
\\
\indent (4) Observe that $h\in G_V=C_G(g)$ fixes $[v]\in \P(V)$ precisely if $v$ is an eigenvector of $h$.
\end{proof}

\begin{corollary}\label{Corollary Bgell}
$
B=B_{{\bf g},\ell}=\pi_G(V\cap S^{2n-1})=(V\cap S^{2n-1})/G_{g,\ell}\subset S_G
$
 is a submanifold of real dimension $\dim B= 2\dim_{\C}\!V-1$.
\end{corollary}
\begin{proof}
$\pi_G(V\cap S^{2n-1})=(GV/G)\cap S_G$ where $GV=\cup h(V)$ over all $h\in G$,  and $GV/G \cong V/G_{g,\ell}$ by Lemma \ref{Lemma cyclic groups Gv}. 
%
%
%
%
Finally, $V\cap S^{2n-1}$ is a transverse intersection and $G_{g,\ell}$ acts freely on it (as $G$ acts freely on $\C^n\setminus 0$).
\end{proof}

By Lemma \ref{Lemma primitive orbits of the S1 action}, $B$ is precisely the moduli space $$\mathcal{O}=\mathcal{O}_{{\bf g},\ell}$$ of parametrized closed Reeb orbits in $S_G$ of length $\ell$ associated to the class ${\bf g}\in \mathrm{Conj}(G)$ via \eqref{Equation pi1 of Y minus E}, as $p=[v]\in B$ determines the Reeb orbit $[0,\ell]\to S_G$, $\phi_t(p)=[e^{it}v]\in S_G$ with initial point $\phi_0(p)=p$. We often blur the distinction by identifying $B\equiv \mathcal{O}$. Note however that the subset $B_{{\bf g},\ell+2\pi k}\subset S_G$ does not depend on $k\in \N$, whilst $\mathcal{O}_{{\bf g},\ell+2\pi k}$ does, due to the length. 
The {\bf short Reeb orbits} are those of {\bf length} $\ell \in (0,2\pi]$, and they determine the {\bf age} in \eqref{Equation age}:
\begin{equation}\label{Equation relating Reeb lengths to age}
\mathrm{age}({\bf g})=\frac{1}{2\pi}\sum_{0<\ell\leq 2\pi} (2\pi - \ell)\dim_{\C} V_{{\bf g},\ell},
\end{equation}
in particular the sum of all these lengths counted with dimension-multiplicity is $2\pi(n-\mathrm{age}\,g)$. Recall from the Introduction, the {\bf minimal Reeb orbits} associated to $g$ are the short Reeb orbits occurring for the smallest value of $\ell$ (amongst eigenvalues of $g$). Also, recall from Appendix C the Definition \ref{Definition Morse-Bott submfd} of {\bf Morse-Bott submanifold}.

\begin{lemma} \label{Lemma dimension of B}
$B\subset S_G$ is a Morse-Bott submanifold of real dimension $\dim B=2 \dim_\C V - 1$.
\end{lemma}
\begin{proof}
Condition (1) of Definition \ref{Definition Morse-Bott submfd} is a consequence of Lemma \ref{Lemma primitive orbits of the S1 action}, as explained above. To check condition (2), let $\gamma:[0,\ell]\to S_G$ be a closed Reeb orbit. Lift $\gamma$ to a path in $V$, $\widetilde{\gamma}:[0,\ell]\to S^{2n-1}$, $\widetilde{\gamma}(t)=\phi_t(v)$ with $g(v)=e^{i\ell}v$.
After an $SU(n)$-change of coordinates, we may assume $\widetilde{\gamma}(0)=(0,\ldots,0,1)$, so $\widetilde{\gamma}(t)=(0,\ldots,0,e^{it})$.
Fix the obvious trivialisation $\gamma(0)^*\xi_G\cong \widetilde{\gamma}(0)^*\xi
\cong \C^{n-1} \times 0\subset \C^n$,
%
%
 then this lifts to a trivialisation of $\widetilde{\gamma}(\ell)^*\xi=(g \widetilde{\gamma}(0))^*\xi$ by $g (\C^{n-1}\times 0)$. So the linearised return map in this trivialisation is $z\mapsto (g^{-1} \circ D\phi_{\ell})(z)= g^{-1} (e^{i \ell } z)$ for $z\in \C^{n-1}\times 0\subset \C^n$,
whose $1$-eigenspace is $E=V \cap (\C^{n-1} \times 0)$.
This is a transverse intersection, as $\C\cdot \widetilde{\gamma}(0)=0\times \C\subset V$, thus $\dim 
E= 2 \dim_\C V - 2=\dim B - 1$, moreover the intersection equals $TB\cap \xi|_B$.
%
%
\end{proof}

Each Reeb orbit $\gamma:\R/\ell \Z\to S_G$ defines iterates $\gamma^k: \R / k\ell\Z \to S_G$, $\gamma^k(t) = \gamma(t)$, for $k\in \N$, and the {\bf multiplicity} of $\gamma$ is the largest $k$ with $\gamma = \eta^k$ for some closed Reeb orbit $\eta$.

\begin{lemma}\label{Lemma multiplicity of Reeb orbit}
 The multiplicity of the Reeb orbit of length $\ell$ corresponding to $p\in B$ equals
\begin{equation}\label{Equation multiplicity calculation}
m_p=\max_{b\in (0,\infty)}  \left\{ \tfrac{\ell}{b}  :  h(v) = e^{ib}v  \textrm{ for some } h \in G \textrm{ and }  v\in \pi_G^{-1}(p) \subset S^{2n-1}\right\}=\ell\,|G_p|/2\pi,
\end{equation}
so it may depend on the orbit in $B$.
For short Reeb orbits $(0\!<\!\ell\!\leq\! 2\pi)$, $m_p\leq |G_p|\leq |G|$ and it is determined by $g=h^{m_p}$ where $h$ is the generator of $G_v$ and $p=[v]$.
\end{lemma}
\begin{proof}
The first equality in \eqref{Equation multiplicity calculation} is immediate, as the achieved maximum yields a minimal $b$ for which $[e^{ib}v]=[v]\in S_G$. The second equality follows by Lemma \ref{Lemma cyclic groups Gv}\eqref{Item 1 CGg}, since the generator $h$ of $G_v$ will achieve the maximum in \eqref{Equation multiplicity calculation} and satisfies $h(v)=e^{i 2 \pi /|G_v|}v$, so $b=2\pi/|G_v|$. 
%
%
%
\end{proof}

\subsection{The associated $S^1$-action on the Morse-Bott submanifolds of Reeb orbits}

There are two circle actions on $B=\mathcal{O}$: the $S^1$-action that $B$ inherits as a subset $B\subset \C^n/G$; and the {\bf associated $S^1$-action} of Definition \ref{Definition Morse-Bott submfd}: the circle $S^1_{\ell}=\R/\ell \Z$ acts on $\mathcal{O}$ by time-translation $\gamma\mapsto \gamma(\cdot + c)$ for $c\in S^1_{\ell}$. The orbits of both actions agree geometrically as images in $\C^n/G$\fix{, yielding the same circle $C\subset \C^n/G$, but the degrees of the quotient maps $S^1 \to C$, $t\mapsto [e^{it}v]$ and $S^1_{\ell}\to C$, $t\mapsto \gamma(t)$ can differ. Indeed the degrees are respectively $|G_p|$ and $m_p=\ell\,|G_p|/2\pi$ (by \eqref{Equation multiplicity calculation}), so the two actions coincide only for $\ell = 2\pi$.}

\begin{theorem}\label{Theorem equiv coh of Morse Bott submanifolds} 
One can identify $B/S_{\ell}^1$ with the quotient $\P V/G_V$. The fiber of $\P V \to \P V/G_V$ over $p=[v]$ has size $|C_G(g)|/|G_v|$ which divides $|G|$.
Via $\P V \cong \C\P^{\,\dim_\C V-1}$, we get
$$
 H^*(B/S^1_{\ell})\cong H^*(\C\P^{\,\dim_\C V-1})
$$ 
over any field of characteristic not dividing $|G|$.

The equivariant cohomology of $B=\mathcal{O}$ for the $S^1_{\ell}$-action, in the sense of Appendix B, is
\begin{equation}\label{Equation EH of curly O}
EH^*(\mathcal{O}) \cong H^{*-1}(B/S^1_{\ell})
\end{equation}
in characteristic zero, and for $\ell\leq 2\pi$ also over fields of characteristic coprime to $2,3,\ldots,|G|$.   
\end{theorem}
\begin{proof}
$B/S_{\ell}^1 = (GV\cap S^{2n-1})/(S^1\cdot G)=\P(GV)/G = \P V/G_V$,
%
%
%
%
where the last equality uses Lemma \ref{Lemma cyclic groups Gv} \eqref{Item 3 CGg},\eqref{Item 4 CCg}. The statement about the fiber follows by Lemma \ref{Lemma cyclic groups Gv} \eqref{Item 5 CCg}. Remark \ref{Remark cohomology of finite quotients} implies 
$H^*(\P V/G_V) \cong H^*(\P V)^{G_V}$ by pulling back via the projection $p: \P V \to \P V/G_V$.
But $H^*(\P V)^{G_V}=H^*(\P V)$, indeed $p^*$ is surjective since $p^*\omega_G=\omega$ is the Fubini-Study form on $\P V\subset \C\P^{n-1}$ that generates the ring $H^*(\P V) \cong H^*(\C\P^{\,\dim_\C V-1})$. 
%
%
%
Theorem \ref{Theorem EH for free action} implies \eqref{Equation EH of curly O} (using that the sizes of the stabilisers are $m_p\leq |G|$ when $\ell \leq 2\pi$, by Lemma \ref{Lemma multiplicity of Reeb orbit}).
\end{proof}

\subsection{Conley-Zehnder indices of Reeb orbits in the quotient}

Appendix C is a survey of Conley-Zehnder indices and their properties (CZ1)-(CZ4). For $n\geq 2$, let $\kappa$ be the canonical bundle on $S^{2n-1}$ induced by the standard complex structure $J$ inherited from the inclusion $S^{2n-1}\subset \C^n$. 
Let $E=\sum z_j \partial z_j$ denote the {\bf Euler vector field}, and let $E^*=\sum \overline{z}_j\, dz_j$.
%
%
Then, as complex vector bundles, $T_{\C}\C^n$ and $T_{\C}^*\C^n$ pulled back to $S^{2n-1}$ split as $\xi\oplus \C E$ and dually $\xi^*\oplus \C E^*$ (in particular, the contraction $\iota_E \xi^*=0$), where $\xi$ was defined in Sec.\ref{Subsection symplectic data for the quotient}. 
%
%
It follows that $\Lambda^{n-1}_{\C}(\C^n)^*$ pulled back to $S^{2n-1}$ has a one-dimensional summand determined by the forms which vanish when contracted with $E$. Thus $\kappa$ can be trivialized by the section
\begin{equation} \label{section of canonical bundle}
K = \iota_{E} (dz_1 \wedge \cdots \wedge dz_n).
\end{equation}
%
%
%
%
%
This ensures that a \fix{compatible trivialisation of $\xi$} together with the field\footnote{where $Z,Y$ are the vector fields defined in Sec.\ref{Subsection convex symplectic manifolds}.} $Z+iY$ from Remark \ref{Remark from CZReeb to CZHam} yields, up to homotopy, the standard trivialisation of the anti-canonical bundle $\mathcal{K}_{\C^n}^*$ of $\C^n$. Conversely, observe that the standard trivialisation of $\mathcal{K}_{\C^n}$ by $dz_1\wedge \cdots \wedge dz_n$ is induced by any complex frame for $T\C^n$ arising as the image under $SL(n,\C)$ of the standard frame $\partial_{z_1},\ldots,\partial_{z_n}$ (since $\det = 1$). Restricting to $S^{2n-1}\subset \C^n$, if the first vector field of that frame is $E$ then the remaining vector fields of the frame induce the section \eqref{section of canonical bundle} on $\kappa$. 
%
%

Recall from Sec.\ref{Subsection symplectic data for the quotient}, $G\subset SU(n)$ is a finite subgroup acting freely on $\C^n\setminus 0$. Observe that $SU(n)$ preserves the field $E$, the canonical bundle $\mathcal{K}_{\C^n}$,
%
%
and more generally the above splittings; so it preserves $K$.
%
%
%
%
%
So via the quotient $\pi_G: S^{2n-1}\to S_G=S^{2n-1}/G$, we obtain induced data $\alpha_G,\xi_G,J_G,\kappa_G,K_G$ from the analogous data $\alpha,\xi,J,\kappa,K$ defined on $S^{2n-1}\subset \C^n$.

\begin{theorem}\label{Theorem CZ indices for Bott mfds}
For $0< \ell \leq 2\pi$, the Conley-Zehnder index of $B=B_{{\bf g},\ell}$ is 
$$
CZ(B)=n-2\,\mathrm{age}({\bf g}) + 2 \sum_{\ell'<\ell} \dim_{\C} V_{{\bf g},\ell'} + \tfrac{1}{2}\dim B + \tfrac{1}{2},
$$
and the associated grading \eqref{Equation Morse-Bott shift} is
\begin{equation}
\label{Equation mu B of Morse Bott submfds}
\mu(B)=2\,\mathrm{age}({\bf g}) - 2 \sum_{\ell'<\ell} \dim_{\C} V_{{\bf g},\ell'} - \dim B - 1.
\end{equation}
%
%
%
In general, $\mathrm{CZ}(B_{{\bf g},\ell+2\pi k}) = \mathrm{CZ}(B_{{\bf g},\ell}) + 2kn$, and thus $\mu(B_{{\bf g},\ell+2\pi k}) = \mu(B_{{\bf g},\ell}) - 2kn$.

If $e^{i\ell}$ is the minimal eigenvalue of ${\bf g}$ then $\mu(B)= 2\,\mathrm{age}({\bf g})-\dim B -1$, so the maximum of $B$ (in the sense of Remark \ref{Remark from CZReeb to CZHam}) is a minimal Reeb orbit in grading $\mu= 2\,\mathrm{age}({\bf g})-1$.
\end{theorem}
\begin{proof}
We first show that the middle claim follows from the first. Note that $\phi_t^*K=e^{int}K$. As $\phi_{2k\pi}=\mathrm{Id}_{S_G}$, an orbit $\gamma$ in $B_{{\bf g},\ell+2\pi k}$ can be viewed as a concatenation of an orbit $\gamma_1\in B_{{\bf g},\ell}$ together with the orbit $\gamma_2:\R/2\pi k \Z \to S_G$, $t\mapsto e^{it} \gamma_1(\ell)$. Thus $\mathrm{CZ}(\gamma)=\mathrm{CZ}(\gamma_1)+\mathrm{CZ}(\gamma_2)$ by property \ref{CZ Index of Catenation is sum of CZ indices}. By properties \ref{CZ of a direct sum} and \ref{CZ of a path in unitary line} one deduces $\mathrm{CZ}(\gamma_2)=2kn$.
%
%

We now prove the first claim. Abbreviate $d=\dim_{\C} V$.
By an $SU(n)$-change of coordinates to a basis of unitary eigenvectors for $g$, we may assume the Reeb orbit is $\gamma(t)=\pi_G(\widetilde{\gamma}(t))$ for
$\widetilde{\gamma}(t)=(e^{it},0,\ldots,0)$, with $t\in [0,\ell]$, and that the first $d$ standard basis vectors are eigenvectors of $g$ with eigenvalue $e^{i\ell}$ (i.e. a basis for $V=V_{g,\ell}$). Thus
$$
g = \mathrm{diag}(e^{i\ell},\ldots,e^{i\ell}, e^{i\ell'},\ldots)
$$
with $e^{i\ell}$ in the first $d$ diagonal entries, and the other eigenvalues $e^{i\ell'},\ldots$ of $g$ in the remaining diagonal entries, where $0<\ell'\leq 2\pi$. This basis splits $T_{\gamma(t)}\C^n = \C^n = \C \oplus \xi$ where $\C=\C\times 0$, $\xi=0\times \C^{n-1}$.
\fix{Picking a trivialisation of} $\gamma^*\xi_G$ is equivalent to picking a {\bf $G$-compatible trivialisation} of $\widetilde{\gamma}^*\xi$, meaning that if $h(\widetilde{\gamma}(t))=\widetilde{\gamma}(s)$ then the trivialisations of $\widetilde{\gamma}(t)^*\xi$ and $\widetilde{\gamma}(s)^*\xi$ are related by multiplication by $h$. By Lemma \ref{Lemma cyclic groups Gv}, it suffices to ensure compatibility for $h\in G$, with $h^k=g$ (and $k\in \N$ maximal such).
Because we may assume that the above eigenbasis for $g$ arose from an eigenbasis for $h$, the compatibility with $h$ will in fact also be guaranteed by our construction. We define two auxiliary families in $SL(n,\C)$,
\begin{align*}
g_t = \mathrm{diag}(e^{i\ell \frac{t}{\ell}},\ldots,e^{i\ell \frac{t}{\ell}}, e^{i\ell' \frac{t}{\ell}},\ldots) = \mathrm{diag}(e^{it},\ldots,e^{it}, e^{i\ell' \frac{t}{\ell}},\ldots)
\\
a_t = \mathrm{diag}(1,\ldots,1,e^{2\pi i(\mathrm{age}(g)-n)\frac{t}{\ell}},1,\ldots,1)
\end{align*}
for $t\in [0,\ell]$, where for $a_t$ the non-unit entry can be placed in any position except the first (to ensure that the first section $E=(a_t\cdot g_t)(1,0,\ldots,0)$ is the Euler vector field along $\widetilde{\gamma}$). The image under $a_t \cdot g_t$ of the eigenbasis defined at $t=0$ gives a trivialisation of the canonical bundle of $\C^n$ along $\widetilde{\gamma}$, which is compatible with the standard trivialisation as $\det (a_t\cdot g_t)=1$, using \eqref{Equation relating Reeb lengths to age}. It is also $G$-compatible since at $t=\ell$, we have $a_{\ell}\cdot g_{\ell}=g$ (using that $\mathrm{age}(g)\in \N$).
Omitting the first section (the Euler vector field) yields a compatible trivialisation of $\widetilde{\gamma}^*\xi$.
The (linearised) flow $\phi_t=\mathrm{diag}(e^{it},\ldots,e^{it})$ in this trivialisation becomes $(a_t \cdot g_t)^{-1}\cdot\phi_t$, so
$$
a_t^{-1}\cdot \mathrm{diag}(1, \ldots,1,e^{i(\ell-\ell')\frac{t}{\ell}},\ldots).
$$
As the lengths $\ell,\ell'\in (0,2\pi]$, all differences satisfy $|\ell-\ell'|<2\pi$, so the function \eqref{Equation W function for CZ} satisfies $W(\ell-\ell')=+1$ if $\ell'<\ell$, and $-1$ if $\ell'>\ell$. Thus properties \ref{CZ of a direct sum} and \ref{CZ of a path in unitary line} imply
$$
\mathrm{CZ}(B) = 2(n-\mathrm{age}(g)) + \sum_{\ell'<\ell} \dim_{\C} V_{{\bf g},\ell'} -
\sum_{\ell'>\ell} \dim_{\C} V_{{\bf g},\ell'}.
$$
By replacing one copy of $n$ by $n=\dim_{\C} V + \sum_{\ell'\neq \ell} \dim_{\C} V_{{\bf g},\ell'}$, and using $\dim_{\C}V = \frac{1}{2}\dim B + \frac{1}{2}$ from Lemma \ref{Lemma dimension of B}, the claim follows.
\end{proof}

\subsection{Convex symplectic manifold structure for resolutions of isolated singularities}
\label{Subsection Convex symplectic manifold structure for resolutions of isolated singularities}

Continuing with the notation from Sec.\ref{Subsection symplectic data for the quotient}, let $\pi: Y \to X=\C^n/G$ be any resolution, and let $B_\epsilon=\{z: |z|\leq \epsilon \}\subset \C^n/G$ for $0<\epsilon\ll 1$.

\begin{lemma} \label{lemma convex symplectic manifold on resolution}
There is a K\"{a}hler form $\omega_Y$ on $Y$ such that $(Y,\omega_Y)$ is convex symplectic and its data $\omega,\theta,Z,R$ on $Y\setminus \pi^{-1}(B_{\epsilon})$ agrees via $\pi$ with the data $\omega_G$, $\theta_G$, $Z_G$, $R_G$ on $(\C^n\setminus B_{\epsilon})/G$.
\end{lemma}

This Lemma is immediate from Lemma \ref{lemma Q factorial resolution} because $Y$ admits a K\"{a}hler form $\omega_Y$ which agrees with $\pi^*\omega_X$ outside of an arbitrarily small neighbourhood of $\pi^{-1}(0)$.

When $Y$ is a crepant resolution, the natural diagonal $\C^*$-action on $\C^n$ (and thus on $\C^n/G$) lifts to $Y$ by \fix{Proposition \ref{proposition lift of action}.} So $\pi:Y\to \C^n/G$ is $\C^*$-equivariant. The $S^1\subset \C^*$ defines an $S^1$-action $\widetilde{\phi}_t$ on $Y$ that lifts the action $\phi_t[z]=[e^{it}z]$ on $\C^n/G$. By averaging, namely replacing $\omega_Y$ by $\int_{S^1} \widetilde{\phi}_t^*\omega_Y$, we may assume that the $\omega_Y$ above is $S^1$-invariant.\footnote{such an averaging does not affect $\pi^*\omega_G$, which is already $\phi_t$-invariant.}

\subsection{Vanishing of symplectic cohomology of crepant resolutions}
\label{Subsection Vanishing of the symplectic cohomology of crepant resolutions} 
Below we prove directly that $SH^*(Y)=0$ based on \cite{Ritter2}, but we remark that $SH^*(M)=0$ is a general result \cite{Ritter4} for convex symplectic manifolds $M$ with Chern class $c_1(M)=0$ whose Reeb flow at infinity arises from a Hamiltonian $S^1$-action $\phi_t$ with non-zero Maslov index $2I(\phi)\in \Z$ (in our setup, $I(\phi)=n$ is the winding number arising in the proof of Theorem \ref{Theorem SH=0}). Indeed, the assumptions imply that $SH^*(M)$ is $\Z$-graded, $c^*:QH^*(M)\to SH^*(M)$ is a quotient $\K$-algebra homomorphism, and the $S^1$-action defines a $\K$-linear automorphism $\mathcal{R}_g\in \mathrm{Aut}(SH^*(M))$ of degree $2I(\phi)$, but $\dim_{\K}QH^*(M)=\dim_{\K} H^*(M,\K)<\infty$ so $SH^*(M)=0$. 

Recall crepant resolutions $\pi: Y \to X=\C^n/G$ admit a $\C^*$-action, yielding an action $\phi_t$ by $S^1\subset \C^*$, and by averaging the K\"{a}hler form by $\int\phi_t^*\omega_Y$ we may assume $\omega_Y$ is $S^1$-invariant. In general, for a symplectic manifold $(M,\omega)$ with an $S^1$-action, if $\omega$ is $S^1$-invariant the action is symplectic; and if $H^1(M,\R)=0$ (e.g.\,if $M$ is simply connected) the action is Hamiltonian.\footnote{Indeed, the vector field $v=v_t$ that generates the $S^1$-action $\phi_t$ on $M$ defines a form $\sigma=\sigma_t=\omega(\cdot,v)$, which is closed due to the Cartan formula $d(\iota_{v}\omega)+\iota_v d\omega = \mathcal{L}_{v}\omega = \partial_t|_{t=0}\phi_t^*\omega$. Thus $[\sigma]\in H^1(M,\R)=0$, so $\sigma=dH$ and $v=X_H$ is Hamiltonian, where $H=H_t: M\to \R$.}

\begin{lemma}
Let $\pi: Y \to X=\C^n/G$ be a crepant resolution, where $G\subset SL(n,\C)$ is a finite subgroup. If the K\"{a}hler form on $Y$ is $S^1$-invariant, then the $S^1$-action is Hamiltonian.
\end{lemma}
\begin{proof}
Resolutions of quotient singularities are always simply connected \cite[Thm.7.8]{Kollar}.
\end{proof}
Now assume $G\subset SU(n)$ acts freely on $\C^n\setminus 0$, $\omega_Y$ is $S^1$-invariant, and $\omega_Y=\pi^*\omega_G$ except on a neighbourhood $Y_{\epsilon}=\pi^{-1}(B_{\epsilon})$ of $\pi^{-1}(0)$ (see Sec.\ref{Subsection Convex symplectic manifold structure for resolutions of isolated singularities}). Recall $H_G=\frac{1}{2}|z|^2: \C^n/G\to \R$ generates the $S^1$-action $z\mapsto e^{it}z$ on $\C^n/G$. As $\pi$ is $S^1$-equivariant,
the Hamiltonian $H:Y\to \R$ generating the $S^1$-action  on $Y$ agrees with $\pi^*(H_G)$ on $Y\setminus Y_{\epsilon}$ up to an additive constant.

\begin{lemma}\label{Lemma generic slope gives constant orbits}
The period of closed orbits of $H:Y\to \R$ outside $\pi^{-1}(0)$ are integer multiples of $2\pi/|G|$. For $k\not\in \frac{2\pi}{|G|}\Z$, all $1$-orbits of $kH$ are contained in $\pi^{-1}(0)$, and for $k\not\in 2\pi \Q$ they are constant.
\end{lemma}
\begin{proof}
As the map $\pi$ is $S^1$-equivariant, the $1$-orbits in $Y\setminus \pi^{-1}(0)$ agree with the $1$-orbits in $(\C^n\setminus 0)/G$. By Lagrange's theorem, the $|G|$-th iterate of a closed Reeb orbit in $S^{2n-1}/G$ will lift to a closed Reeb orbit in $S^{2n-1}$, whose period must be in $2\pi \Z$. The first two claims follow. For the final claim, consider the initial point $p$ of a $1$-orbit in $\pi^{-1}(0)$. The stabilizer in $S^1$ of $p$ contains $e^{ik}$, which has a dense orbit in $S^1$ for $k\not\in 2\pi \Q$ and so $p$ is fixed by $S^1$.
\end{proof}

\begin{remark}\label{Remark SH is lim HFLk}
The Hamiltonians $L_k=kH$ for $k\not\in 2\pi\Q$ determine $SH^*(Y)=\varinjlim HF^*(L_k)$, but they cannot be used in the construction of $SH^*_+(Y)$ (see Appendix D).
\fix{We briefly clarify the meaning of the notation $HF^*(L_k)$.
Recall that the K\"{a}hler metric on $Y$ is $S^1$-invariant under the Hamiltonian $S^1$-action, and the fixed locus of that action lies in the compact set $\pi^{-1}(0)$. The proof Lemma 1 in Frankel \cite{Frankel} proves that the fixed locus is a compact Riemannian submanifold $C\subset Y$ lying inside $\pi^{-1}(0)$, and that it is a Morse-Bott submanifold for the Hamiltonian.
Observe that the (constant) $1$-orbits of $L_k$ are the points of $C$. When we write $HF^*(L_k)$, it is understood that one either uses a Morse-Bott Floer complex (see Remark \ref{Remark from CZReeb to CZHam}) or one uses a generic small compactly supported perturbation of $L_k$ (in the sense of Hofer-Salamon \cite[Theorem 3.1]{Hofer-Salamon}). Up to isomorphisms, the choice of perturbation does not affect $HF^*(L_k)$ nor the continuation maps that define the direct limit $SH^*(Y)$. By a judicious choice of perturbation, using an auxiliary Morse function $f_S: S\to \R$ on each connected Morse-Bott submanifold $S$ (see \cite[Prop.2.2]{CFHW} and \cite[Appendix B]{KwonvanKoert}) the generators of the Floer complex after perturbation can be identified with the critical points $x$ of the functions $f_S$, and the grading of those generators is $\mu(S)+\mathrm{ind}_{f_S}(x)$ where $\mu(S)$ is as in \eqref{Equation Morse-Bott shift} and $\mathrm{ind}_{f_S}(x)$ is the Morse index of $x\in \mathrm{Crit}(f_S)$. 
}
\end{remark}

\begin{theorem}\label{Theorem SH=0}
The generators of $CF^*(L_k)$, for $k\not\in 2\pi\Q$, lie in arbitrarily negative degree for large $k$, therefore 
$$
SH^*(Y)=0, 
\quad 
ESH^*(Y)=0,
\quad
SH^*_+(Y)\cong H^{*+1}(Y,\K),
\quad
ESH^*_+(Y) \cong H^{*+1}(Y,\K)\otimes_{\K} \F.
$$
\end{theorem}
\begin{proof}
As $Y$ is crepant and\footnote{The second condition ensures that sections of $\mathcal{K}$ agree on the boundary, up to homotopy.}
 $H^1(\partial Y_{\epsilon},\R)\cong H^1(S^{2n-1}/G,\R)=0$ (using Remark \ref{Remark cohomology of finite quotients}), there is a nowhere zero smooth section $s$ of the canonical bundle $\mathcal{K}$ of $Y$ agreeing with the pull-back of the standard section on $Y\setminus Y_{\epsilon}\cong (\C^n\setminus B_{\epsilon})/G$ for $\mathcal{K}_{\C^n}$,
\begin{equation}\label{Equation standard canonical section}
s|_{Y - Y_\epsilon} = dz_1 \wedge \cdots \wedge dz_n.
\end{equation}
The $\C^*$-action $\phi_{w}$ on $Y$ induces a $\C^*$-action on $\mathcal{K}$, thus it defines a function $f_w:Y\to \C^*$ by
$$
\phi_{w}^*(s|_{\phi_w(y)}) = f_w(y)\, s|_y.
$$
By \eqref{Equation standard canonical section}, $f_w(y)=w^{n}$ for $y\in Y\setminus Y_{\epsilon}$. The map $f: \C^*\times Y \to \C^*$, $f(w,y)=f_w(y)$, defines a homotopy class of maps in 
$$
[\C^*\times Y,\C^*] \cong H^1(\C^*\times Y) = (H^1(\C^*)\otimes H^0(Y)) \; \oplus \; (H^0(\C^*)\otimes H^1(Y)).
$$
%
%
Only $H^1(\C^*)$ matters, as $Y$ is connected and $H^1(Y)=0$ (as $Y$ is simply connected). Thus $f$ is homotopic to the map $(w,y)\mapsto w^{n}$.
Given a (constant) $1$-orbit $x$ of $L_k=kH$, let $p=x(0)\in \pi^{-1}(0)$ be the initial point. As $p$ is a fixed point, we may linearise the $\C^*$-action on $T_p Y$:
$$
\C^* \times T_p Y \to T_p Y,\quad 
(w,Z_j) \mapsto w^{m_j}Z_j
$$
for some $m_1,\cdots,m_n \in \Z$, where $Z_1,\cdots,Z_n$ are a basis of $T_p Y$ induced by a choice of $\C$-linear coordinates near $p$ (by a linear change of basis, we diagonalised the action at $p$). Thus, the action on $\mathcal{K}$ is by multiplication by $w^{-\sum m_j}$, so it must equal $w^n$, so $-\sum m_j=n$. 
As the time $t$ flow of $L_k=kH$ is $\phi_{w^k}=\phi_w^k$ for $w=e^{it}$, using Appendix C we deduce
$$\textstyle
\mathrm{CZ}(x) = \sum W(-km_j) \geq \sum 2 \lfloor \tfrac{-km_j}{2\pi} \rfloor \geq \sum  (\tfrac{-km_j}{\pi} -2) =  (\tfrac{k}{\pi}-2)n
$$
(using \ref{CZ of a direct sum}, \ref{CZ of a path in unitary line}, and $W(t)\geq 2 \lfloor \tfrac{t}{2\pi} \rfloor$), so the grading $\mu(x)\leq (3-\frac{k}{\pi})n$. \fix{Thus we conclude that $\mu(x)\to -\infty$ as $k\to \infty$. The final claim then follows, because $SH^*(Y)=\varinjlim HF^*(L_k)$ is a direct limit over grading-preserving maps.}\footnote{\fix{A small perturbation of $L_k$ to a generic Hamiltonian as described in Remark \ref{Remark SH is lim HFLk} will change Conley-Zehnder indices by at most $\dim_{\R}(Y)$, so $\mu(x)\to -\infty$ as $k\to \infty$ still holds even after perturbation.
We remark that one can also prove directly (and more generally) that Conley-Zehnder indices change by at most $\dim_{\R}(Y)$ after perturbation, without appealing to the Morse-Bott argument in Remark \ref{Remark SH is lim HFLk}, by the same argument as in McLean \cite[Lemma 4.10]{McLean}.}}
The same argument applies to $ESH^*(Y)$, using that $\F$ lies in negative degrees (alternatively, it follows from $SH^*(Y)=0$ by \eqref{Equation spectral sequence u adic filtration}). The final two results follow by Corollary \ref{Corollary LES for H SH and SHplus}.
\end{proof}

\subsection{Computation of $ESH^*_+$ of crepant resolutions}
\label{Subsection Computation of positive equivariant symplectic cohomology of crepant resolutions}

We refer to Appendices B and D for the construction of $ESH^*_+$. We choose a specific sequence $H_k$ of {\bf admissible Hamiltonians} (see \ref{Subsection Admissible Hamiltonians}) with final slope $k\not\in 2\pi\Q$, to compute $ESH^*_+(Y)=\varinjlim EHF^*_+(H_k)$. Recall $H:Y \to \R$ is the Hamiltonian generating the $S^1$-action on $Y$, and by Lemma \ref{lemma convex symplectic manifold on resolution} the radial coordinate $R$ on $Y\setminus Y_{\epsilon}$ (in the sense of Sec.\ref{Subsection convex symplectic manifolds}) agrees via $\pi:Y\to \C^n/G$ with the radial coordinate $R_G=|z|^2$ on $(\C^n\setminus B_{\epsilon})/G$; in that region $H^2=\frac{1}{4}R^2$.
Define $H_k = H^2$ except on the region where $H^2$ has slope $\geq k$ in $R$, and extend by $H_k = kR$ outside of that region. By projecting via $\pi$ and then projecting to $S_G=S^{2n-1}/G$, there is a $1$-to-$1$ correspondence between the $1$-orbits defining $ECF^*_+(H_k)$ and the Reeb orbits in $S_G$ of length $\leq k$ 
\fix{
(analogously to Remark \ref{Remark SH is lim HFLk}, it is understood that a small perturbation of $H$ is needed near $\pi^{-1}(0)$, but the resulting generators near $\pi^{-1}(0)$ of the Floer complex will be quotiented out by definition of $ECF^*_+(H_k)$, see Section \ref{Section Appendix D}).
}
So we may abusively write $B_{{\bf g},\ell}\subset Y$ when referring to those orbits in $Y$ (recall $B_{{\bf g},\ell}\subset S_G$ from Lemma \ref{Corollary Bgell}).

\begin{corollary}\label{Corollary computation of ESH+}
Assume $\mathrm{char}\,\K=0$, or more generally $\mathrm{char}\,\K$ coprime to all integers $\leq |G|$.
\\
For ${\bf g}=[g]\in \mathrm{Conj}(G)$, the orbits in $\cup_{\ell>0} B_{{\bf g},\ell}\subset Y$ contribute a copy of \fix{the $\K$-vector space} $\F[-\mu_{\bf g}]$ to the $E_1$-page of the Morse-Bott spectral sequence for\footnote{more precisely, as we only take the direct limit on cohomology, $ECF^*_+(Y,H_k)$ sees at least the summand $\K[-\mu_g]\oplus \K[-\mu_g+2] \oplus \K[-\mu_g+4]\oplus \cdots \oplus \K[-\mu_g+2m n]$ of $\F[-\mu_g]=\K[-\mu_g]\oplus \K[-\mu_g+2]\oplus \cdots$ if $k\geq (m+1)\pi$.} $ESC^*_+(Y)$ (see Appendix E), where
$$
\mu_g=\mu_{{\bf g}}=2\,\mathrm{age}({\bf g})-1.
$$
Moreover, as a $\K$-vector space, $ESH^*_+(Y)$ has one summand $\F[-\mu_{{\bf g}}]$ for each ${\bf g}\in \mathrm{Conj}(G)$.
\end{corollary}
\begin{proof}
Suppose $\mathrm{char}\,\K=0$. By Theorems \ref{Theorem equiv coh of Morse Bott submanifolds}-\ref{Theorem CZ indices for Bott mfds}, each eigenspace $V=V_{g,\ell}$ of $g$ for $0<\ell\leq 2\pi$ yields a submanifold $B=B_{{\bf g},\ell}\subset Y$ which contributes a copy of $EH^*(B) \cong H^{*-1}(\C\P^{\dim_{\C}V-1})$ (by Theorem \ref{Theorem equiv coh of Morse Bott submanifolds}), shifted up in grading by the $\mu(B)$ in \eqref{Equation mu B of Morse Bott submfds}. 
By Theorem \ref{Theorem CZ indices for Bott mfds}, 
if $e^{i\ell}$ is the minimal eigenvalue of ${\bf g}$ then the maximum of $B/S^1$ contributes a generator in grading $\mu=\mu_g$,
$$
\mu_g=\mu(B)+\dim B =2\,\mathrm{age}(g)-1.
$$
So $B/S^1\cong \P(V)/G$, which has dimension $2\dim_{\C} V-2$, contributes one generator in each odd degree in the range $[\mu_g- 2\dim_{\C}V+2,\mu_g]$.
%
%
The next smallest eigenvalue $e^{i\ell'}$ of $g$, corresponding to an eigenspace $V'=V_{{\bf g},\ell'}$ and a submanifold $B'=B_{{\bf g},\ell'}$, will have a maximum in degree $\mu(B')+\dim B'= \mu_g - 2\dim_{\C}V$, so it contributes one generator in each odd degree in the range $[\mu_g-2\dim_{\C}V-2\dim_{\C}V'+2,\mu_g-2\dim_{\C}V]$.
Inductively, the eigenvalues of $g$ will account for one generator in each odd degree in the range $[\mu_g - 2n+2, \mu_g]$. The iteration formula in Theorem \ref{Theorem CZ indices for Bott mfds}, i.e.\,the cases $2k\pi<\ell\leq 2(k+1)\pi$ for $k\in \N\setminus 0$, contribute generators in all odd degrees $[\mu_g - 2n +2-2kn, \mu_g-2kn]=[\mu_g -2(k+1)n +2, \mu_g-2kn]$.

The second claim follows because the Morse-Bott spectral sequence degenerates: all generators are in odd total degree, so all differentials $d_r^{pq}$ on all pages $E^{pq}_r$ for $r\geq 1$ will vanish.

When $\K$ has non-zero characteristic, Theorem \ref{Theorem SH=0} implies that $ESH^*_+(Y)\cong H^{*+1}(Y)\otimes_{\K} \F$ is a free $\F$-module with generators in degrees $*=-1,0,1,\ldots,\dim_{\R} Y -2$ \fix{(not $\dim_{\R} Y -1$ as $H^{\dim_{\R}Y}(Y)=0$ since $Y$ is non-compact, and we can also exclude all even degrees including $\dim_{\R} Y -2$ since the generators of $ESH^*_+(Y)$ are in odd degree, $\F$ lies in even degrees and $\dim_{\R}Y$ is even). Moreover, the number of $\F$-summands in $ESH^*_+(Y)$ equals $$\dim_{\K} ESH^{-1}_+(Y)=\sum \dim_{\K} H^{2j}(Y) = \chi(Y),$$
the Euler characteristic of $Y$, because the generators of $ESH^*_+(Y)$ are in odd degree and $\F$ as a $\K$-vector space has exactly one generator in each non-positive even degree.
}%

\fix{In the range of degrees $-1,1,3,\ldots,\dim_{\R}Y-3$ mentioned above,} only generators corresponding to Reeb orbits of period $\ell\leq 2\pi$ can contribute because those of period $\ell>2\pi$ have grading $\mu_g-2kn\leq -3$ for $k\geq 1$, as $\mu_g=2\,\mathrm{age}(g)-1\leq 2n-3$ using that the age grading lies in $[0,n-1]$ by \eqref{Equation age} (thus their grading and that of their differentials does not land in the range \fix{$[-1,\dim_{\R}Y-3]$}). Finally, under the assumptions on $\mathrm{char}\,\K$, we can apply Theorem \ref{Theorem equiv coh of Morse Bott submanifolds} to the Morse-Bott manifolds of Reeb orbits of period $\ell\leq 2\pi$.
\fix{We refer the reader back to the closely related discussion below \eqref{Equation ESH+ is ordinary EH} for additional clarifications.}
\end{proof}

\section{Appendix A: Weil divisors, Cartier divisors and Resolutions}

In the paper, we work with analytic geometry, so the words \emph{regular}, \emph{rational}, \emph{isomorphism} below are replaced respectively by \emph{holomorphic}, \emph{meromorphic}, \emph{biholomorphic}. In this section, codimension always refers to the complex codimension.
By \fix{a {\bf variety}} $X$ we mean an irreducible normal quasi-projective complex variety.
Recall {\bf normal} means each point has a normal affine neighbourhood, and an affine variety is normal if its coordinate ring $\C[X]$ of regular functions is integrally closed (i.e.\,elements of its fraction field satisfying a monic polynomial over the ring must lie in the ring). Equivalently, all local rings of $X$ are integrally closed. Non-singular quasi-projective varieties are normal, since the local rings are UFDs.
Normality ensures that for a codimension one subvariety $Z\subset X$, there is some affine open of $X$ on which the ideal for $Z$ is principal in $\C[X]$.
It follows \cite[Chp.II.5.1 Thm.3]{Shafarevich} that the subvariety of singular points of $X$ has codimension at least $2$.
Any quotient $Y=X/G$ of a normal affine 
variety $X$ by a finite group $G$ of automorphisms is also normal: if $f\in \C(Y)$ is integral over $\C[Y]=\C[X]^G$ then it is integral over $\C[X]$, so  $f\in \C[X]$ by normality, but functions in $\C(Y)$ are $G$-invariant, so $f\in \C[Y]$.
In particular, $\C^n/G$ is normal for any finite subgroup $G\subset SL(n,\C)$.

%
%
%

\begin{remark}
Normality is equivalent to requiring that rational functions bounded in a neighbourhood of a point must be regular at that point. The removable singularities theorem shows $\C$ is normal, and Hartogs' extension theorem becomes: for any subvariety $V\subset X$ of codimension at least $2$, any regular function on $X\setminus V$ extends to a regular function on $X$.
\end{remark}

As we work with singular varieties, we need to distinguish two notions of divisor, which coincide for non-singular varieties.
A {\bf Weil divisor} is a finite formal $\Z$-linear combination $\sum a_m V_m$ of irreducible closed subvarieties of codimension one. It is {\bf effective} if all $a_m\geq 0$. A rational section $s$ of a line bundle $L\to X$ defines a Weil divisor $(s)=\sum \mathrm{ord}_s(Z)\, Z$, where we sum over irreducible closed subvarieties $Z\subset X$ of codimension one, and $\mathrm{ord}_s(Z)$ is the associated valuation.\footnote{In analytic geometry, in a local trivialisation near a generic point $p\in Z$, $s$ is given by a meromorphic function $f=gz^k$, where $g$ is an invertible holomorphic function, and $z$ is a holomorphic coordinate extending a local basis of holomorphic coordinates for $Z$ near $p$. Then one defines $\mathrm{ord}_s(Z)=k$.}
Similarly, a global non-zero rational function $f$ on $X$ defines a {\bf principal} Weil divisor $(f)$, and in this notation, $(f/g)=(f)-(g)$ for such functions $f,g$. Two \fix{Weil divisors $D_1,D_2$ are \emph{linearly equivalent} if their difference is principal, $D_1-D_2=(f)$. The corresponding equivalence classes of Weil divisors define the \emph{Weil divisor class group} $\mathrm{Cl}(X)$.}
The {\bf support} of a Weil divisor $\sum a_m V_m$ is the subset $\cup V_m$ taking the union over all $a_m\neq 0$. 

A {\bf Cartier divisor} $D$ is defined by \fix{an equivalence class of data:} an open cover $U_i$ of $X$ together with non-zero rational functions $f_i$ on $U_i$, such that $f_i/f_j$ is regular on the overlap $U_i\cap U_j$. One identifies two data sets if one can pass to a common refinement \fix{of the cover or rescale} the $f_i$ by invertible regular functions. The {\bf support} is the union of zeros and poles of the $f_i$. A {\bf principal} Cartier divisor is given by the data $(X,f)$ for a global meromorphic function $f$. Two Cartier divisors are {\bf linearly equivalent} if they differ by a principal Cartier divisor. Cartier divisors up to linear equivalence correspond to complex line bundles on $X$ up to isomorphism. The associated bundle $\mathcal{O}(D)$ is constructed from the $U_j\times \C$ using $f_i/f_j$ as transition function $U_j\times \C \to U_i \times \C$. The line bundle admits a rational section $s$ given by $s=f_j$ on $U_j$, so in particular the Weil divisor $(s)$ agrees locally with $(f_j)$.

As the variety is normal, Cartier divisors can also be defined as the ``locally principal Weil divisors", namely a Weil divisor that locally is equal to $(f)$ for some meromorphic function $f$ on $X$. Explicitly $D=\sum a_m Z_m$ \fix{with $a_m=\mathrm{ord}_{f_i}(Z_m)$} for any $f_i$ satisfying $U_i\cap Z_m\neq \emptyset$. 

A Weil divisor $D$ is {\bf $\Q$-Cartier} if $mD$ is Cartier for some $m\in \N$.
A quasi-projective variety $X$ is {\bf $\Q$-factorial} if all Weil divisors are $\Q$-Cartier.
Algebraic or analytic varieties over $\C$ with only quotient singularities are $\Q$-factorial \cite[Prop.5.15]{kollarmori}. So $\C^n/G$ is $\Q$-factorial for any finite group $G\subset SL(n,\C)$. The idea is that, although in general Weil divisors do not \fix{pull back to Weil divisors,\footnote{Pull-backs of Weil divisors are not usually defined. However, $\pi:\C^n\to \C^n/G$ is a finite flat degree $|G|$ cover over the complement of the singular set which has codimension $\geq 2$ (consisting of points of $\C^n$ with non-trivial stabilizer). So the pre-image $\pi^{-1}(Z)$ of an irreducible codimension $1$ subvariety of $\C^n/G$ is a codimension $1$ subvariety over that complement and can then be uniquely extended to a Weil divisor on $\C^n$.} in this case} a Weil divisor $D$ in $\C^n/G$ pulls back
%
%
 to a Weil divisor in $\C^n$, in particular this is Cartier so locally it is cut out as $(f)$, then the averaged function $\sum_{g\in G} f\circ g^{-1}\in \C[\C^n]^G=\C[\C^n/G]$ locally cuts out $|G|\cdot D$, so $|G|\cdot D$ is Cartier.
%
%
%

Let $\pi: Y \to X$ be a morphism of varieties.
The {\bf push-forward} of Weil divisors is defined by $\pi_*(\sum a_i V_i)= \sum a_i' \overline{\pi(V_i)}$ with $a_i=a_i'$ if the closure $\overline{\pi(V_i)}\subset X$ is a codimension one subvariety, and $a_i'=0$ otherwise. This is in general only a Weil divisor, even if $\sum a_i V_i$ is Cartier.
If $\pi$ is a dominant map (i.e.\,with dense image), then the {\bf pull-back}
 of a Cartier divisor given by data $(U_i,f_i)$ on $X$ is the Cartier divisor $(\pi^{-1}(U_i),\pi^*f_i)$ on $Y$. This corresponds to pull-back for the corresponding line bundles.
 
By a {\bf resolution} of $X$, we mean a non-singular variety $Y$ with a proper birational morphism  $\pi : Y \to X$, such that the restriction $\pi: \pi^{-1}(X_{\mathrm{reg}}) \to X_{\mathrm{reg}}$ is an isomorphism over the smooth locus $X_{\mathrm{reg}}=X\setminus \mathrm{Sing}(X)$ of $X$. 
%
%
%
Resolutions always exist by Hironaka's theorem \cite{Hironaka}.
	
\begin{lemma} \label{lemma Q factorial resolution}
Let $X$ be a $\Q$-factorial variety with only one singular point, at $0 \in X$, and let $\omega_X$ be a K\"{a}hler form on $X \setminus 0$.
\fix{Any resolution} $\pi : Y \to X$ admits a K\"{a}hler form $\omega_Y$ such that $\omega_Y = \pi^* \omega_X$ outside of an arbitrarily small neighbourhood
of $\pi^{-1}(0)$.
\end{lemma}
\begin{proof}
We first make an observation. Given any Cartier divisor $D$ on $Y$, the push-forward $\pi_*D$ is a Weil divisor on $X$, so $m\pi_*D$ is Cartier for some $m\in \N$. Let $f$ be a meromorphic function on $X$ such that $m\pi_*D=(f)$ near $0\in X$. Then $mD-(\pi^*f)$ is a Cartier divisor on $Y$ whose support
 intersects some open neighbourhood $U\subset Y$ of $\pi^{-1}(0)$ only in codimension one subvarieties contained in $\pi^{-1}(0)$.
The Cartier divisor yields a line bundle $\mathcal{O}(mD-(\pi^*f))$ on $Y$ with a meromorphic section\footnote{namely $S=s^{\otimes m}/\pi^*f$ where $s$ is a meromorphic section for the line bundle associated to $D$ with $D=(s)$.} $S$ whose only zeros and poles in $U$ lie in $\pi^{-1}(0)$.
%
%
%
%
%
%

As $Y$ is quasi-projective, we may pick a very ample line bundle $L\to Y$. Let $D$ denote \fix{a choice of associated Cartier divisor.}
The above argument yields a meromorphic section $S$ of $L^{\otimes m}$ whose only zeros and poles in a neighbourhood $U\subset Y$ of $\pi^{-1}(0)$ are contained in $\pi^{-1}(0)$.

As $L^{\otimes m}$ is very ample, we can choose a Hermitian metric $|\cdot|$ on $L^{\otimes m}$ such that the curvature $\Omega$ of the Chern connection determines a positive $(1,1)$-form
$\frac{i}{2\pi} \Omega = \frac{i}{2\pi}\partial \overline{\partial} \log |S|^2$ on $Y$ (using the fact that the latter is the Chern form for any non-zero meromorphic section $S$ of a holomorphic Hermitian line bundle).
%
%
%
%
%
Let $c: Y \to [0,1]$ be a smooth function, with $c=1$ near $\pi^{-1}(0)$ and $c=0$ outside of $U$. The claim follows by taking $\omega_Y=\pi^* \omega_X+\delta\frac{i}{2\pi}\partial \overline{\partial} (c\cdot \log |S|^2)$, and picking $\delta>0$ sufficiently small so that this is a positive form where $0<c<1$ \fix{(observe that our particular choice of section $S$ ensures that $\omega_Y$ is well-defined on $0<c<1$).} Note that $\omega_Y=\pi^*\omega_X$ outside of $U$, and $\omega_Y=\pi^*\omega_X+ \delta \frac{i}{2\pi}\Omega$ where $c=1$. 
\end{proof} 

\begin{remark} \label{Remark can assume piD is zero}
\fix{From the preceding proof, we obtain\footnote{\fix{After relabelling $L^{\otimes m}$, $S$ in the proof of Lemma \ref{lemma Q factorial resolution} by $L$ and $s$ respectively.}} a Weil divisor $D=(s)$ in $Y$ arising from a rational section $s$ of a very ample line bundle $L\to Y$, such that the only zeros and poles of $s$ in some neighbourhood $U\subset Y$ of $\pi^{-1}(0)$ lie in $\pi^{-1}(0)$. 
Thus $D=A+B$ decomposes into a Weil divisor $A$ supported in $\pi^{-1}(0)$, and a Weil divisor $B$ supported in $Y\setminus U$. 
We now show that one can construct $L$ and $s$ so that $B=0$ if one makes the additional assumption that Weil divisors of $X$ supported away from $0$ are torsion in $\mathrm{Cl}(X)$.
By the assumption, $m\cdot \pi_*B=(f)$ for some positive integer $m\geq 1$ and some
meromorphic function $f$ on $X$.
As $\pi$ is a biholomorphism over $X\setminus \{0\}$, the Weil divisor $mD-(\pi^*f)$ is supported in $\pi^{-1}(0)$. Therefore if we replace $L,s$ by $L^{\otimes m}$ and $s^{\otimes m}/\pi^*f$ respectively, we obtain $B=0$ above.}
\end{remark}

As normal varieties $X$ are smooth in codimension one, the
canonical bundle $\Lambda_{\C}^{\mathrm{top}}T^*X_{\mathrm{smooth}}$ defined on the smooth locus extends to a Weil divisor class $\mathcal{K}_X$ on $X$, the {\bf canonical divisor}.
%
%
A variety $X$ is {\bf quasi-Gorenstein} if $\mathcal{K}_X$ is Cartier, i.e.\;there is a line bundle $\omega_X$ 
%
%
%
which restricts to the canonical bundle on the smooth part $X_{\mathrm{smooth}}\subset X$. 
%
%
In particular $\C^n/G$, for finite subgroups $G\subset SL(n,\C)$, are quasi-Gorenstein, as $g\in G$ acts on $\Lambda^{\mathrm{top}}_{\C} T^*\C^n$ by $\det g=1$. (They are in fact Gorenstein, although we will not define this notion here).

For a birational morphism $\pi : Y \dashrightarrow X$, a closed codimension one subvariety $V\subset Y$ is an {\bf exceptional divisor} if $\pi(V)\subset X$ has codimension $\geq 2$ (i.e.\,the Weil divisor $\pi_*(V)=0$). The {\bf exceptional divisor of $\pi$} is the Weil divisor $\sum V_i$ summing over the exceptional $V_i$.

If $\pi$ is a regular birational morphism of quasi-Gorenstein varieties, and $V$ is exceptional and irreducible, then the {\bf discrepancy} of $V$ is $a_V=\mathrm{ord}_f(V)$ where $f=s_Y/\pi^*s_X$ is a rational section of $\omega_Y\otimes \pi^*\omega_X$ determined by a choice of non-zero rational sections $s_X,s_Y$ of $\omega_X,\omega_Y$ such that $\pi^*s_X$ and $s_Y$ agree in the region where $\pi$ is an isomorphism. 
$V$ is called a {\bf crepant divisor} if $\mathrm{ord}_f(V)=0$.
The {\bf total discrepancy} of $X$ is the infimum of the discrepancies over all such possible $V$ and $\pi:Y\to X$. The total discrepancy of a smooth variety $X$ is one \cite[Corollary 2.31]{kollarmori}.
%
The {\bf discrepancy divisor} is the Cartier divisor $\sum a_V \, V$ of $Y$, \fix{summing over irreducible exceptional divisors}. The discrepancies $a_V$ are in fact independent of the choices of $s_X,s_Y,\pi$, 
%
%
%
and $\pi$ is a {\bf crepant resolution} if all $a_V=0$, so $\pi^*\omega_Y=\omega_X$.

%
%
%
%
%
%
\begin{lemma}[Negativity Lemma]\label{Lemma negativity lemma}
 Let $\pi:Y \to X$ be a resolution, where $X$ has only one singular point at $0 \in X$. Let $D$ be a homologically trivial $\Q$-Cartier divisor in $Y$ with $\pi_*(D)=0$. Then $D=0$.
\end{lemma}
\begin{proof}
 We can apply the negativity lemma \cite[Lemma 3.39]{kollarmori} to the divisors $\pm D$ (\fix{they are both $\pi$-nef}, since they are homologically trivial, and $\pi_*(\pm D)$ are effective since zero). It follows that $\pm D$ are effective, therefore $D=0$.
%
%
%
%
\end{proof}

\begin{lemma}\label{Lemma c1 is zero for crepant}
Let $\pi:Y \to X$ be a resolution, where we assume $X$ has only one singular point at $0 \in X$ and $\mathcal{K}_X$ is $\Q$-Cartier. Then $Y$ is crepant if and only if $c_1(Y)|_U\! = \!0$ on a neighbourhood $U \!\subset\! Y$\,of $\pi^{-1}(0)$.
\end{lemma}
\begin{proof}
Suppose first $\mathcal{K}_X$ is Cartier.
As $\pi$ is an isomorphism away from $\pi^{-1}(0)$, all exceptional divisors lie in $\pi^{-1}(0)$. As $\pi^{-1}(0)$ is of codimension one, the irreducible components $V_i$ of $\pi^{-1}(0)$ are the exceptional divisors, and $\sum V_i$ is the exceptional divisor of $\pi$. In the notation above, 
we may pick $s_X$ so that near $0\in X$ it is regular and nowhere-vanishing. If $\pi$ is crepant then $f=s_Y/\pi^*s_X$ has no zeros or poles along the $V_i$ so $s_Y$ is regular and nowhere-vanishing near $\pi^{-1}(0)$, so $\omega_Y$ is trivial near $\pi^{-1}(0)$ as required. Conversely, suppose $c_1(Y)|_U=0$. Near $0\in X$ we can pick a local nowhere vanishing section $s_X$ of $\omega_X$. Then $\pi^*s_X$ defines a rational section for $\omega_Y|_U$ (shrinking $U$ if necessary).
 By construction, the support of the divisor $(\pi^*s_X)$ on $U$ lies entirely in $\pi^{-1}(0)$. Then $c_1(Y)|_U=0$ implies $(\pi^*s_X)$ is homologically trivial on $U$. Lemma \ref{Lemma negativity lemma} implies $(\pi^*s_X)=0$ on $U$. Thus $\pi^*s_X$ is a regular nowhere vanishing section for $\omega_Y|_U$, as required (taking $s_Y=\pi^*s_X$ in the definition of crepant).
%
%
When $\mathcal{K}_X$ is $\Q$-Cartier, say $m\mathcal{K}_X$ is Cartier, one considers $f^{\otimes m}$, $\omega_Y^{\otimes m}$, $\omega_X^{\otimes m}$ instead of $f,\omega_Y$, $\omega_X$.
\end{proof}

\begin{proposition} \label{proposition lift of action}
Let $X$ be a $\Q$-factorial variety with only one singularity,
at $0 \in X$, admitting a regular \fix{$\C^*$-action $\mu:\C^* \times X \to  X$} which fixes $0$.
\fix{Assume that the Weil divisors of $X$ supported away from $0$ are torsion in $\mathrm{Cl}(X)$.}
Suppose $\pi : Y \to X$ is a crepant resolution.
Then $\mu$ lifts to a $\C^*$-action on $Y$, so that $\pi$ is a $\C^*$-equivariant morphism.
\end{proposition}
\begin{remark}
\fix{Let $X=\C^n/G$ for a finite subgroup $G\subset SL(n,\C)$ acting freely on $\C^n\setminus \{0\}$. Then $X$ with the standard $\C^*$-action satisfies the assumptions of Proposition \ref{proposition lift of action}. Indeed, given a Weil divisor $D$ in $\C^n/G$ supported away from $0$, we can define a Weil divisor $\widetilde{D}$ in $\C^n$ by picking a ``lift'' of $D$ via the quotient $\psi:\C^n \to \C^n/G$, meaning each subvariety $S$ arising in $D$ gets replaced by a choice of lift of $S$ via $\psi$. There are $|G|$ distinct choices of such a lift of $S$, and the lifted subvarieties are freely permuted by $G$. By construction, $\psi_*\widetilde{D}=D$.
Weil divisors in $\C^n$ are known to be principal,\footnote{\fix{More generally this holds whenever the coordinate ring is a unique factorization domain \cite[Prop.II.6.2]{Hartshorne}, which in our case is $\C[x_1,\ldots,x_n]$.}} so $\widetilde{D}=(\widetilde{f})$ for a meromorphic function $\widetilde{f}$ on $\C^n$. Observe that for any $g\in G$ the Weil divisor $(g^*\widetilde{f})$ also has push-forward $\psi_*(g^*\widetilde{f})=D$, since $G$ permutes the lifted subvarieties. Thus 
the averaged $G$-invariant meromorphic function $\sum_{g\in G} g^*\widetilde{f}$ on $\C^n$ descends to a well-defined meromorphic function $f$ on $\C^n/G$ with associated Weil divisor $(f)= |G|\cdot D$. Thus $|G|\cdot D$ is a principal divisor, so $D$ is torsion in $\mathrm{Cl}(X)$.}%
%
%
%

\fixtwo{The above argument can also be adapted to the situation when $G$ does not act freely on $\C^n\setminus \{0\}$, as follows. Let $S,D$ be as before. Recall that $\C^n/G$ is normal, so the singular set $\mathrm{Sing}(\C^n/G)$ has codimension at least two.
Therefore the intersection $S_{\mathrm{sing}}=S\cap \mathrm{Sing}(\C^n/G)$ has codimension at least one in $S$. As $G$ acts freely on $\C^n\setminus \psi^{-1}(\mathrm{Sing}(\C^n/G))$, there are $|G|$ choices of lifts of $S\setminus S_{\mathrm{sing}}$ to $\C^n$. We make one such choice of lift, and then we take the Zariski closure, call it $\widetilde{S}\subset \C^n$. Summing over $S$, the sum of these subvarieties $\widetilde{S}$ defines a Weil divisor $\widetilde{D}$ on $\C^n$ such that $\psi_*\widetilde{D}=D$. The remainder of the previous argument then holds verbatim, showing that $|G|\cdot D$ is principal, so $D$ is torsion in $\mathrm{Cl}(X)$.}
\end{remark}
\begin{proof}[Proof of Proposition \ref{proposition lift of action}]
We first lift the action over the smooth locus, to obtain a rational map
$\mu : \C^* \times Y \dashrightarrow Y.$
%
%
%
Let $\mu_w=\mu(w,\cdot) : Y \dashrightarrow Y$
be the restriction to $\{w\} \times Y$, for $w\in \C^*$.
The set of points at which $\mu_w$ is not regular is of codimension at least two (the proof of \cite[Chp.II.3 Thm.3]{Shafarevich}
%
%
applies to the quasi-projective non-singular variety $Y$). Let $Y'\subset Y$ be the locus where $\mu_w$ is regular. The differential $d\mu_w: T_{\C}Y'|_y \to T_{\C}Y|_{\mu_w(y)}$ induces a rational section $S=(\Lambda^{\mathrm{top}}_{\C} d\mu_w)^{\vee}$ of $\mathcal{L}=\omega_Y \otimes (\mu_w^*\omega_Y)^{\vee}$ 
%
%
which is regular over $Y'$. \fix{It is locally} the determinant of the Jacobian matrix for $\mu_w$. 
Suppose by contradiction that $\mu_w$ has an exceptional divisor $V$. 
%
%
%
%

Because $\pi: Y\setminus \pi^{-1}(0)\to X\setminus 0$ is a $\C^*$-equivariant isomorphism, it follows that $V \subset \pi^{-1}(0)$.
%
%
%
%
By construction, the section $S$ must vanish along $V$. 
The effective divisor $(S)$ on $Y'$ yields an effective divisor on $Y$ by taking the closure (recall $\mathrm{codim}\, Y\setminus Y'\geq 2$). 
As before, its support lies in $\pi^{-1}(0)$.
If $(S)$ were null-homologous on $U$ then, since $\pi_*(\pm(S))=0$, Lemma \ref{Lemma negativity lemma} would imply $(S)=0$ on $U$, contradicting that $(S)$ involves a strictly positive multiple of $V$. Therefore $(S)$ is not null-homologous. This implies that $c_1(\mathcal{L})|_U\neq 0$. Finally, we check that this is false. As in the proof of Lemma \ref{Lemma c1 is zero for crepant}, it suffices to consider the case when $\mathcal{K}_X$ is Cartier (if $m\mathcal{K}_X$ is Cartier we consider the bundle $\mathcal{L}^{\otimes m}$ etc.). Let $s_X$ be a rational section of $\omega_X$ that is regular and nowhere zero near $0\in X$. Then $s_X \cdot (\underline{\mu}_w^*s_X)^{\vee}$ is a rational section of $\omega_X \otimes (\underline{\mu}_w^*\omega_X)^{\vee}$ where $\underline{\mu}_w$ is the $\C^*$-action on $X$. As $\pi$ is $\C^*$-equivariant, $\sigma=\pi^*s_X \cdot (\mu_w^*\pi^*s_X)^{\vee}$ is a rational section of $\mathcal{L}$. \fix{As $\pi$ is crepant}, $s_Y/\pi^*s_X$ has trivial orders of vanishing near $\pi^{-1}(0)$, so $\sigma$ trivialises $\mathcal{L}$ near $\pi^{-1}(0)$, thus $c_1(\mathcal{L})|_U=0$ for a neighbourhood $U$ of $\pi^{-1}(0)$, the required contradiction.
%
%

Thus $\mu_w$ has no exceptional divisors for all $w \in \C^*$.
%
%
%

%
%
\fix{By Remark \ref{Remark can assume piD is zero} (and using the assumption about Weil divisors) we can construct a very ample line bundle $L$ on $Y$ with a rational section $s$ whose only zeros and poles lie in $\pi^{-1}(0)$, in particular $\pi_*(D)=0$ as $\pi$ collapses the divisors in $\pi^{-1}(0)$. From now on, by divisor we mean the equivalence class of the divisor.}

As $\mu_w$ has no exceptional divisors,
the $D_w = (\mu_w)_*(D)$ define a smooth family of Weil divisors parameterized by $w \in \C^*$. In particular all $D_w$ are homologous. Also note that $\pi_*(D_w)=0$, since $\pi_*D=0$ and $\pi$ is $\C^*$-equivariant.

Abbreviate $D_{w,w'}=D_w-D_{w'}$ for any $w,w'\in \C^*$. As both $\pm D_{w,w'}$ are null homologous and $\pi_*(\pm D_{w,w'})=0$, Lemma \ref{Lemma negativity lemma} implies
$D_{w,w'}=0$, so $D_w = D_1$ for all $w \in \C^*$.

Recall that a section of $L$ is equivalent to
a rational function $f$ such that
$(f) + D$ is effective.
%
%
%
Since $D_w = D$ for all $w \in \C^*$, we deduce that
$\mu_w$ pulls back sections of $L$
to sections of $L$ by pulling back its respective rational functions.
Thus we have an action
$\mu^* : \C^* \times H^0(L)^* \to H^0(L)^*$
and each map $\mu^*_w = \mu^*|_{\{w\} \times H^0(L)^*} : H^0(L)^* \to H^0(L)^*$ is linear.
As $L$ is very ample, it is also relatively very ample (i.e.\;the restriction to fibers of $\pi$ is very ample),
%
%
so we have a natural embedding
$\iota: Y \to \mathrm{Proj}(\oplus_{k\geq 0} H^0(L^{\otimes k})^*)$
%
%
%
%
%
%
%
%
 and the induced action of $\mu^*$ on $H^0(L^{\otimes k})^*$ preserves $\text{image}(\iota)$
and its restriction to $Y$ is $\mu$.
This extends $\mu$ to an action
$\C^* \times Y \to Y$ compatibly with the $\C^*$-action on $X$ via projection.
\end{proof}
%
\section{Appendix B: Equivariant symplectic cohomology}
\subsection{Classical $S^1$-equivariant cohomology}
\label{Subsection Classical S1-equivariant cohomology}
\fix{Recall that $\K$ is the Novikov field from \eqref{Equation Novikov field}, and (co)homology is computed with coefficients in $\K$ unless indicated otherwise.}
For a topological space $X$ with an $S^1$-action, $H^*_{S^1}(X)=H^*(ES^1\times_{S^1} X)$ is a module over $H_{S^1}^*(\mathrm{point})=H^*(BS^1)$ by applying the functor $H^*_{S^1}(\cdot)$ to $X\to \mathrm{point}$. We take $ES^1=S^{\infty}$ to be the direct limit of $S^1\subset S^3\subset S^5\subset  \cdots$ where $S^{2n-1}\subset \C^n$, then $BS^1=ES^1/S^1=\C\P^{\infty}$ and we identify $H^*(\C\P^{\infty})=\K[u]$, with $u$ in degree $2$. So $H^*_{S^1}(X)$ is naturally a $\K[u]$-module. 

Similarly, $H_*^{S^1}(X)=H_*(ES^1\times_{S^1}X)$ admits a cap product action $u:H_*^{S^1}(X)\to H_{*-2}^{S^1}(X)$ making $H^{S^1}_{-*}(X)$ a $\K[u]$-module (notice the negative grading). We can identify $H^{S^1}_{-*}(\mathrm{point})=H_{-*}(\C \P^{\infty})\cong \K[u^{-1},u]/u\K[u]$ as $\K[u]$-modules, 
where $u^{-j}$ represents the class $[\C\P^j]$ graded negatively.
Equivalently, completing in $u$, we can view them as $\K[\![u]\!]$-modules
$$
\F = \K(\!(u)\!)/u\K[\![u]\!] \cong H_{-*}(\C\P^{\infty}),
$$
where $\K[\![u]\!]$ \fix{and} $\K(\!(u)\!)=\K[\![u]\!][u^{-1}]$ are respectively formal power series and Laurent series. 
\\[1mm]
\indent {\bf Motivation.}\,\emph{One wants an equivariant Viterbo theorem \cite{Viterbo}: for closed oriented spin $N$, we want a $\K[\![u]\!]$-module isomorphism $ESH^*(T^*N)\cong H_{n-*}^{S^1}(\mathcal{L}N)$ $($using the natural $S^1$-action on the free loop space $\mathcal{L}N=C^{\infty}(S^1,N))$, compatibly with the inclusion of constant loops $EH^*(T^*N)\cong H_{n-*}^{S^1}(N)\to H_{n-*}^{S^1}(\mathcal{L}N)$ via $c^*:EH^*(T^*N)\to ESH^*(T^*N)$ $($the equivariant analogue of the canonical map $c^*:H^*(T^*N) \to SH^*(T^*N))$. As the $S^1$-action on constant loops is trivial, $H_{-*}^{S^1}(N)\cong H_{-*}(N)\otimes H_{-*}(\C\P^{\infty})\cong H_{-*}(N)\otimes_{\K} \F$.}
\subsection{$S^1$-complexes and equivariant symplectic cohomology}
\label{Subsection S1complexes and equivariant symplectic cohomology}
%
Let $C^*=CF^*(H)$ be a Floer chain complex used in the construction of symplectic cohomology $SH^*(M)$ (for $M$ as in Sec.\ref{Subsection convex symplectic manifolds}).
Following Seidel \cite[Sec.(8b)]{Seidel}, $C^*$ admits degree $1-2k$ maps 
$$\delta_k: CF^*(H)\to CF^*(H)[1-2k]$$
for $k \in \N$, where $\delta_0$ is the usual Floer differential
%
%
 and $\sum_{i+j=k} \delta_i \circ \delta_j=0$. In general such data $(C^*,\delta_k)$ is called an {\bf $S^1$-complex}, and we recall the specific Floer construction of $\delta_k$ later.

Given an $S^1$-complex $C^*$, we define the {\bf equivariant complex} by
\begin{equation}\label{Equation equivariant d}
EC^* = C^*\otimes_{\K} \F, \qquad d=\delta_0 + u\delta_1+u^2\delta_2+\cdots
\end{equation}
so $\K$-linearly extending 
\fix{$$d(y u^{-j})=\sum u^{k-j}\delta_k(y)=\delta_0(y)u^{-j}+\delta_1(y)u^{-j+1}+\cdots+\delta_j(y)u^0.$$}
Notice $d$ is naturally a $\K[\![u]\!]$-module homomorphism (but not for $\K[u^{-1}]$),
%
%
%
 in particular $u$ acts by zero on $u^{0}C^* $. The resulting cohomology $EH^*$ is a $\K[\![u]\!]$-module.
 
For the Floer complexes, the direct limit $ESH^*(M)$ of the equivariant Floer cohomologies $EHF^*(H)$ over the class of Hamiltonians admits a canonical $\K[\![u]\!]$-module homomorphism
\begin{equation}\label{Equation canonical equation in equivariant setup}
c^*:EH^*(M)\cong H^*(M)\otimes_{\K} \F \to ESH^*(M),
\end{equation}
where $EH^*(M)$ arises from the Morse-theoretic analogue of the construction 
\eqref{Equation equivariant d} for the $1$-orbits of a $C^2$-small Hamiltonian (these are constant orbits, so involve a trivial $S^1$-action).
%
%
\subsection{Construction of the $\delta_k$ in Floer theory}
\label{Subsection The construction of the deltak in Floer theory}
We follow work of Viterbo \cite[Sec.5]{Viterbo} and \fix{Seidel \cite[Sec.(8b)]{Seidel}}, and for details we refer to Bourgeois-Oancea \cite[Sec.2.3]{Bourgeois-Oancea-S1}. The function 
\begin{equation}\label{Equation f on CPinfinity}
\fix{
\textstyle f:S^{\infty}\to \R, \qquad  f(z)=\sum_{j=1}^{\infty} j |z_j|^2
}
\end{equation}
 induces a Morse function on $\C\P^{\infty}=S^{\infty}/S^1$ with critical points $c_0,c_1,c_2,\ldots$ \fix{in degrees $0,2,4,\ldots$. One} picks a connection on the $S^1$-bundle $S^{\infty}\to \C\P^{\infty}$ that is trivial near all $c_i$ (in a chosen trivialisation of the bundle near each $c_i$). This induces a connection on 
 \begin{equation}\label{Equation E S1 LM} 
 E=S^{\infty}\times_{S^1} \mathcal{L}M \to \C \P^{\infty},
 \end{equation}
 using the natural $S^1$-action on the free loop space $\mathcal{L}M=C^{\infty}(S^1,M)$. One picks a family of Hamiltonians $H_z:M\to \R$ parametrized by $z\in \C\P^{\infty}$ such that locally near $c_i$ the $H_z = h_i$ are some fixed Hamiltonians $h_i: M \to \R$. 
Similarly, one picks generic almost complex structures $J_z$ on $E_z$ that locally near $c_i$ are some fixed $J_i$ on $M$.  It is understood that all Hamiltonians \fix{$h_i$} and almost complex structures \fix{$J_i$} must be of the type allowed by the construction of $SH^*(M,\omega)$ \fix{(so generic time dependent perturbations are tacitly understood)}, and in a neighbourhood $V$ of infinity we require $H_z$ to be radial of the same slope as the given $H$ (so Floer solutions will stay in the compact region $M\setminus V$ by a maximum principle).

If we do not work with a Morse-Bott model, then the given Hamiltonian $H$ has to have been time-dependently perturbed, say $H=H(t,\cdot)$, so as to ensure that Hamiltonian $1$-orbits are non-degenerate. The $H_z$ must then be $S^1$-equivariant: $H_{e^{i\tau}z}(t,\cdot) = H_z(t-\tau,\cdot)$, and similarly for the $J_z$.

We now count pairs $(w,v)$, 
$$
w:\R\to \C\P^{\infty} \qquad \qquad v:\R \to E
$$
where $w$ is a $-\nabla f$ flowline for the Fubini-Study metric on $\C\P^{\infty}$, and $v$ is a lift of $w$ which satisfies the Floer equation $\frac{Dv}{ds} + J_z(\frac{Dv}{dt} - X_{H_{w(s)}})=0$, where the derivatives are induced by the connection on $E$.
More precisely, one fixes asymptotics $c_{i_-},c_{i_+}$ for $w$ and asymptotic
$1$-orbits $x_-,x_+$ for $v$ for the Hamiltonians $h_{i_-},h_{i_+}$, and the moduli space $\mathcal{M}(c_{i_-},x_-; c_{i_+},x_+)$ consists of the rigid solutions $[(w,v)]$ modulo the natural $\R$-action in $s$. 

The shift $\sigma:\C\P^{\infty}\to \C\P^{\infty}$, $(z_0,z_1,z_2,\ldots)\mapsto (0,z_0,z_1,\ldots)$ is compatible with the Fubini-Study form and $\nabla f$ (as $\sigma^*f=f+1$), so we may pick all data compatibly with the natural lifted action $\sigma: E\to E$. 
\fix{So for each $i$, $h_i=H$ and $J_i=J$. The} moduli spaces can be naturally identified if we add the same positive constant to both $i_-,i_+$. Define $\delta_k$ as the $\K$-linear extension of
$$
\delta_k(y) = \sum \# \mathcal{M}(c_k,x; c_0,y) \cdot x
$$
summing over the $1$-orbits $x$ of $H$ for which the moduli spaces are rigid, and $\#$ denotes the algebraic count with orientation signs (and Novikov weights, if present).

Then $d=\delta_0 + u \delta_1 + u^2 \delta_2+\cdots$ operates on 
$
CF^*(H)\otimes_{\K} C_{-*}(\C \P^{\infty}) \cong CF^*(H) \otimes_{\K} \F
$
using the above Morse model for $\C\P^{\infty}$, so viewing the formal variable
$u^{-j}$ as playing the role of $c_j$ (equivalently $[\C\P^{j}]\in H_{-*}(\C\P^{\infty})$ negatively graded with $u$ acting by cap product). \fix{For a detailed description of the Morse-Bott construction of the differentials, we refer to Seidel \cite[Sec.(8b)]{Seidel}, Bourgeois - Oancea \cite{Bourgeois-Oancea,BourgeoisOanceaGysin}, and Kwon - van Koert \cite[Appendix B]{KwonvanKoert}.}%
%
%
\subsection{The $u$-adic spectral sequence}
Following Seidel \cite[Sec.(8b)]{Seidel}, the $u$-adic filtration is bounded below and exhausting, so it gives rise to a spectral sequence converging to $ESH^*(M)$ with $E_1^{**} = H^*(C^*,\delta_0) \cong SH^*(M)\otimes_{\K}\F$. Dropping $u^{-j}$ to avoid confusion, as this is only a spectral sequence of $\K$-vector spaces, and adjusting gradings,\footnote{Abbreviating the total degree by $k=p+q$, the filtration is $F^p (EC^k) = C^{k-2p}u^p + C^{k-2p-2}u^{p+1} + \cdots$, which vanishes for $p>0$, and $E_0^{pq}=F^p (EC^k)/F^{p+1}(EC^k)\cong C^{k-2p}u^p$ with $d_0^{pq}=\delta_0$, so $E_1^{pq}\cong SH^{k-2p}(M)$.}
\begin{equation}\label{Equation spectral sequence u adic filtration}
E_1^{pq} \Rightarrow ESH^*(M), \textrm{ where }
\fix{E_1^{pq} = SH^{q-p}(M)} \textrm{ for }p\leq 0, \textrm{ and }E_1^{pq} =0\textrm{ for }p>0.
\end{equation}
%
%
%
\subsection{Gysin sequence}
\label{Subsection Gysin sequence}
Following Bourgeois-Oancea \cite{BourgeoisOanceaGysin}, any $S^1$-complex $(C^*,\delta_k)$ admits a short exact sequence $0 \to C^*  \stackrel{\mathrm{in}}{\longrightarrow}  EC^* \stackrel{u}{\longrightarrow}  EC^{*+2} \to 0$ using the  natural inclusion of $C^*$ as $u^0C^*$ (recall the differential $\delta_0$ on $C^*$ agrees with $d$ on $u^0C^*$), and using the $\K[\![u]\!]$-module action by $u$ on $EC^*$. The induced long exact sequence is called {\bf Gysin sequence}, 
\begin{equation}\label{Equation Gysin sequence explicitly}
\cdots \longrightarrow
H^* \stackrel{\mathrm{in}}{\longrightarrow} 
EH^* \stackrel{u}{\longrightarrow} 
EH^{*+2} \stackrel{b}{\longrightarrow} 
H^{*+1} \longrightarrow \cdots 
\end{equation}
The boundary map $b$ is induced by the maps 
$b(yu^{-j})=\delta_{j+1}(y)$
for $j\geq 0$, since that is the $u^0$-term of $d$ applied to the preimage $yu^{-j-1}$ of $yu^{-j}$ under multiplication by $u$.
%
%

In our setup above, this exact sequence becomes
$$
\cdots \longrightarrow
HF^*(H) \stackrel{\mathrm{in}}{\longrightarrow} 
EHF^*(H) \stackrel{u}{\longrightarrow} 
EHF^{*+2}(H) \stackrel{b}{\longrightarrow} 
HF^{*+1}(H) \longrightarrow \cdots 
$$
then taking the direct limit over continuation maps yields
$$
\cdots \longrightarrow
SH^*(M) \stackrel{\mathrm{in}}{\longrightarrow} 
ESH^*(M) \stackrel{u}{\longrightarrow} 
ESH^{*+2}(M) \stackrel{b}{\longrightarrow} 
SH^{*+1}(M) \longrightarrow \cdots 
$$
The positive symplectic cohomology version is analogous, yielding \eqref{Equation Gysin for SH+}.
\begin{remark} For context, the classical Gysin sequence for an $S^1$-bundle $\pi:E \to M$ is
\begin{equation}\label{Equation Gysin sequence classical case}
\cdots \to H_*(E) \stackrel{\pi_*}{\longrightarrow} H_*(M)\stackrel{\cap e}{\longrightarrow} H_{*-2}(M) \longrightarrow H_{*-1}(E) \to \cdots
\end{equation}
%
%
where the middle arrow is cap product with the Euler class of the bundle. Now consider the free loop space $\mathcal{L}N$. \fix{If $M=\mathcal{L}N\times_{S^1} S^{\infty}$} (and $E=\mathcal{L}N\times S^{\infty}$), then \eqref{Equation Gysin sequence classical case} becomes
$$
\cdots \to H_*(\mathcal{L}N) \stackrel{}{\to} H_*^{S^1}(\mathcal{L}N) \stackrel{}{\to} H_{*-2}^{S^1}(\mathcal{L}N) \to H_{*-1}(\mathcal{L}N)\to \cdots
$$
%
so resembles the Gysin sequence \eqref{Equation Gysin sequence explicitly} via the Viterbo isomorphism $H_*(\mathcal{L}N)\cong SH^{n-*}(T^*N)$.
\end{remark}
\subsection{Construction of the $\delta_k$ in Morse theory}
For a $C^2$-small Morse Hamiltonian $H$, taking $J_z=J$ and $H_z=H$, the Floer theory in Sec.\ref{Subsection The construction of the deltak in Floer theory} reduces to Morse theory: the Hamiltonian orbits and the Floer solutions become time-independent, so $v:\R \to E$ will solve $\frac{Dv}{ds}=-\nabla H_z$ using the Riemannian metric $g_z=\omega(\cdot,J_z\cdot)$ on $M$ over $z$. 
As the two equations for $(w,v)$ have decoupled, this gives rise to the isomorphism $EH^*(M)\cong H^*(M)\otimes_{\K}H_{-*}(\C\P^{\infty})$. 

Let $X$ be an oriented closed manifold with an $S^1$-action and a Morse function $H:X\to \R$, or a convex symplectic manifold (Sec.\ref{Subsection convex symplectic manifolds}) with an $S^1$-action with a $C^2$-small Morse Hamiltonian $H$ radial at infinity with positive slope (so $-\nabla H$ is inward pointing at infinity).

Then replacing $E$ in \eqref{Equation E S1 LM} by $E=S^{\infty}\times_{S^1} X$ gives a Morse theory analogue of Sec.\ref{Subsection The construction of the deltak in Floer theory}, and via Sec.\ref{Subsection S1complexes and equivariant symplectic cohomology} yields a $\K[\![u]\!]$-module $EH^*(X)$. One can identify the Morse complex with the associated complex of pseudo-manifolds obtained by taking the stable manifold $W^s(p)$ of each critical point $p$ with grading $|W^s(p)|=2n-|p|$ (recall this is how one can classically prove that Morse cohomology recovers locally finite homology, and thus the ordinary cohomology by Poincar\'{e} duality). It follows that there is an isomorphism to $S^1$-equivariant lf-homology,
\begin{equation}\label{Equation EH as S^1 lf theory}
EH^*(X) \cong H^{\mathrm{lf},S^1}_{2n-*}(X).
%
%
%
\end{equation}
When $X$ is a closed manifold, $H^{\mathrm{lf},S^1}_{2n-*}(X)=H^{S^1}_{2n-*}(X)$.

When the $S^1$-action is trivial, $EH^*(X)\cong H^{\mathrm{lf}}_{2n-*}(X)\otimes_{\K} \F \cong H^*(X)\otimes_{\K}\F$ (in Sec.\ref{Subsection Classical S1-equivariant cohomology}
we had $EH^*(T^*N)\cong H^*(T^*N)\otimes_{\K} \F$ and by Poincar\'{e} duality $H^*(T^*N)\cong H^*(N)\cong H_{n-*}(N)$).
%
%
%
%
\begin{remark} \eqref{Equation EH as S^1 lf theory} is not $H_*^{S^1}(X)$ because Poincar\'{e} duality in the equivariant setup is only well-behaved using lf-homology, and the $\K[u]$-actions by cup product on $H^*(\C\P^{\infty})$ and by cap product on $H_*(\C\P^{\infty})$ are substantially different.
\end{remark}
\subsection{Equivariant cohomology for $S^1$-actions with finite stabilisers}
%
\begin{theorem}\label{Theorem EH for free action}
\fix{Let $X$ be a closed oriented smooth manifold with a free $S^1$-action. Let $e$ denote the Euler class of the $S^1$-bundle $X\to X/S^1$. There are $\K[\![u]\!]$-module isomorphisms}
$$EH^*(X) \cong \fix{H_{2n-*}(X/S^1)} \cong H^{*-1}(X/S^1),$$
\fix{where the second isomorphism is Poincar\'{e} duality for $X/S^1$, where $u$ acts on $H_*(X/S^1)$ by cap product by $-e$, and $u$ acts on $H^*(X/S^1)$ by cup product by $-e$.}

This result holds more generally if the $S^1$-action has finite stabilisers, provided the characteristic of the underlying field of coefficients is coprime to the sizes of the stabilisers.
%
%
%
\end{theorem}
\begin{proof}

Let $E=(S^{\infty}\times X)/S^1$ with $S^1$ acting by $(z,x)\mapsto (ze^{it},e^{-it}x)$. The natural projection map $p:E\to X/S^1$ is a homotopy equivalence since the fibres $S^{\infty}$ are contractible. By naturality, the following $S^1$-bundles have the Euler classes labelled on the vertical maps
$$
\xymatrix@R=14pt{
 S^{\infty}
 \ar@{->}^{u}[d] \ar@{->}^{}[r] 
 &
  S^{\infty} \times X
 \ar@{->}^{u-e}[d] \ar@{->}^-{\simeq}[r] 
 &
  X
 \ar@{->}^{-e}[d]  \\
 \C\P^{\infty}
  \ar@{->}^{}[r] 
 &
  E=S^{\infty}\times_{S^1}X 
 \ar@{->}^-{\simeq}[r] 
 &
  X/S^1
}
$$
where on $X$ we have reversed the $S^1$-action.
By naturality of the classical Gysin sequence \eqref{Equation Gysin sequence classical case}, this implies that the cap product action of $u$ on $H_*^{S^1}(X)=H_*(E)$ is identified with cap product by $-e$ on $H_*(X/S^1)$ via  $p_*:H_*(E)\cong H_*(X/S^1)$.
The claim follows upon identifying the classical Gysin sequence \fix{\eqref{Equation Gysin sequence classical case} with \eqref{Equation Gysin sequence explicitly}, by first using \eqref{Equation EH as S^1 lf theory} to identify }
$$
\fix{EH^*(X)\cong H_{2n-*}^{\mathrm{lf},S^1}(X) = H_{2n-*}^{S^1}(X) = H_{2n-*}(E) \cong H_{2n-*}(X/S^1),}$$
\fix{and then applying Poincar\'e duality: $H_{2n-*}(X/S^1)\cong H^{2n-1-(2n-*)}(X/S^1)=H^{*-1}(X/S^1)$.}
%
%
\fix{The final claim, about the case of finite stabilisers, is proved analogously but requires three technical lemmas to justify why the above techniques also work when $X/S^1$ has finite quotient singularities. We prove those lemmas separately, below.}
\end{proof}
%
%
\fix{Before proving the lemmas required to justify the final part of Theorem \ref{Theorem EH for free action}, we recall the following motivation behind the assumption on the coefficients.}
\begin{remark}\label{Remark cohomology of finite quotients}
For any quotient $M/G$ of a Hausdorff topological space $M$ by a finite group action $G$ (not necessarily acting freely), the projection $p: M\to M/G$ induces an isomorphism $$p^*:H^*(M/G)\cong H^*(M)^G\subset H^*(M)$$ over any field of characteristic coprime to $|G|$, where $H^*(M)^G$ \fix{is the group of $G$-invariant elements.} When the action is free, this follows by considering the \fix{classical transfer homomorphism $C_*(M/G) \to C_*(M)$ which sends a singular simplex $\sigma$ in $M/G$ to the sum $\sum_{g\in G} g(\widetilde{\sigma})$ of the possible lifts of $\sigma$ to $M$ (where $\widetilde{\sigma}$ is any choice of such a lift).} 
%
%
%
In the non-free case one needs to consider an analogous map acting on the sheaf of coefficients, this is proved for example in \cite[Corollary to Prop.5.2.3]{Grothendieck} or \cite[Thm.II.19.2]{Bredon}. \fix{The result also holds for cohomology with compact supports.\footnote{\fix{In \cite[Thm.II.19.2]{Bredon}, we use the family of supports on $M/G$ given by compact subsets, whose preimage in the sense of \cite[Definition I.6.3]{Bredon} is the family of compact supports in $M$ since $M\to M/G$ is a proper continuous map (which in turn follows from the fact that it is a closed continuous map with compact fibres).}} 
%
%
If we assume in addition that $M$ is locally compact, then the homology version holds:\footnote{\fix{By \cite[Discussion below Proposition V.19.2]{Bredon} (and \cite[Paragraph above Sec.V.2]{Bredon}) one can construct a transfer map on locally finite homology, $\mu_*:H_*^{\mathrm{lf}}(X/G) \to H_*^{\mathrm{lf}}(X)$ such that $\mu_*\circ p_*=\sum g_*$ is the averaging operator by the $G$-action, and $p_*\circ \mu_*$ is multiplication by $|G|$. Under our assumptions on the coefficients, $p_*\circ \mu_*$ is an isomorphism, and it follows that $\mu_*$ is an injection onto the $G$-invariant classes.}} there is an isomorphism $H_*^{\mathrm{lf}}(X/G) \cong H_*^{\mathrm{lf}}(X)^G \subset H_*^{\mathrm{lf}}(X)$ induced by a transfer map (and when $M$ is compact, recall that lf-homology is just ordinary homology).}
%
%
%
%
%
\end{remark}
\fix{In the following lemmas, we assume that $X$ is a closed oriented smooth manifold with an $S^1$-action whose stabilisers are finite, and that (co)homology is taken with coefficients in a field whose characteristic is coprime to the sizes of the stabilisers. We again consider the natural projection $p:E_X\to X/S^1$, where $$E_X=(S^{\infty}\times X)/S^1$$ with
$S^1$ acting by $(z,x)\mapsto (ze^{it},e^{-it}x)$ on $S^{\infty}\times X$. The fibres are $p^{-1}(x)=S^{\infty}/\mathrm{Stab}(x)$, so under the assumption on the field $\K$ of coefficients the fibres have trivial cohomology:
\begin{equation}\label{Equation fibers are cohomologically contractible}
H^*(p^{-1}(x))=H^*(S^{\infty}/\mathrm{Stab}(x))\cong H^*(S^{\infty})=\K\cdot 1 \cong \K
\end{equation}
living in degree zero, by Remark \ref{Remark cohomology of finite quotients}.
\begin{lemma}\label{Lemma MV argument for fibration}
The quotient map induces an isomorphism $p_*:H_*(E_X)\cong H_*(X/S^1)$.
\end{lemma}
\begin{proof}
We will prove the statement by an inductive Mayer-Vietoris argument \cite[Sec.5]{BottTu}, by inducting on a cover of $X/S^1$. The inductive step is the following. Assume that $U,V$ are open subsets of $X/S^1$ such that the claim holds for $X$ replaced by any of $U$, $V$ or $U\cap V$. Then we can prove the claim for $X'=U\cup V$, as follows. By naturality, the two Mayer-Vietoris sequences for the open covers $X'=U\cup V$ and $X'/S^1=U/S^1 \cup V/S^1$ fit into a commutative diagram via the map $p_*$. By the assumption and the five-lemma \cite[Exercise 5.5]{BottTu}, $p_*: H_*(X')\to H_*(X'/S^1)$ is also an isomorphism, as required. We now build the cover.
\\
\indent
Observe that $X=\cup_{n\geq 1}X_n$ can be stratified by considering the sizes of the stabilisers (which are cyclic subgroups of $S^1$), by defining
$$
X_n=\{x\in X: |\mathrm{Stab}(x)|=n\}.
$$
Abbreviate $X_{\leq n}=\cup_{m\leq n} X_m$.
Fixing $n$, a sequence of points in $X_n$ cannot converge to a point in $X_m$ for $m<n$, by a continuity argument. So $X_n\subset X_{\leq n}$ is a closed subset, and $X_{\leq n}\subset X$ is an open subset.
\\
\indent
For any subset $Y\subset X$ on which $S^1$ acts freely, the natural projection map $p:E_Y\to Y/S^1$ is a homotopy equivalence since the fibres $S^{\infty}$ are contractible, so the claim holds for $Y$.
For example, this applies to the case $Y=X_1$.
\\
\indent
We claim that $X_n\subset X$ is a smooth submanifold. Pick any Riemannian metric on $X$. By an averaging argument, we may assume that the Riemannian metric is $S^1$-invariant. The exponential map and the $S^1$-action $\psi_t$ for time $t\in S^1=\R/\Z$ therefore satisfy
\begin{equation}\label{Equation exp preserves flow}
\psi_t\circ \exp_p =
\exp_{\psi_t(p)} \circ \, d_p\psi_t.
\end{equation}
Suppose a point $p\in X$ is fixed by $\psi_t$. Consider the chart near $p\in X$ given via $\exp_p:T_p X \to X$ in a neighbourhood $N$ of $0\in T_p X$. The induced $S^1$-action on $N$ becomes\footnote{\fix{\,$\exp_p^{-1}\circ (\psi_t \circ \exp_p) = 
\exp_p^{-1}\circ (\exp_{p} \circ \, d_p\psi_t) = 
d_p\psi_t$, using \eqref{Equation exp preserves flow} and $\psi_t(p)=p$.}} the linear action by $d_p\psi_t$.
%
%
%
%
It follows that $X_n$ correponds locally via $\exp_p$ to the \emph{linear} subspace of $T_p X$ of vectors that have stabiliser of size $n$, therefore $X_n\subset X$ is a smooth submanifold. \fixtwo{In particular, $X_n$ is a submanifold of the open submanifold $X_{\leq n}\subset X$.}
\\
\indent
We now consider the inductive Mayer-Vietoris argument for $X_{\leq 2}=U\cup V$, where $U=X_1$ and $V$ is an open $S^1$-invariant tubular neighbourhood of $X_2\subset X_{\leq 2}$ (we apply the exponential map to a neighbourhood of the zero section of the normal bundle of $X_2\subset X_{\leq 2}$). The claim holds for $U$ and $U\cap V$ since the $S^1$ action is free there, so once we prove the claim for $V$ we deduce it also for $X_{\leq 2}$. By construction, $V$ deformation retracts $S^1$-equivariantly onto $X_2$, so it remains to prove the claim for $X_2$. For $X_2$, we can quotient the $S^1$ action by $\Z/2$ (without affecting the claim), which reduces us again to the known case of a free $S^1$ action.
\\
\indent
By induction, we can assume that the claim is known for open manifolds for which the stabilisers are at most of size $n-1$, and we now prove it for $X_{\leq n}=U \cup V$ taking $U=X_{\leq n-1}$ and $V$ an $S^1$-invariant tubular neighbourhood of $X_n\subset X_{\leq n}$. By induction the claim holds for $U$ and $U\cap V$, so we reduce to proving it for $V$. It suffices to prove it for $X_n$ since we can $S^1$-equivariantly deformation retract $V$ onto $X_n$. For $X_n$ we may use the quotiented action by $S^1/(\Z/n)$, which is a free action, therefore the claim holds as required.
\end{proof}
\begin{lemma}
The quotient map $X \to X/S^1$ admits a Gysin sequence \eqref{Equation Gysin sequence classical case} which corresponds to the Gysin sequence for the circle bundle $S^{\infty}\times X \to E_X$ via the projection isomorphism $p_*$. In particular, there is a well-defined Euler class in $H^2(X/S^1)$ which agrees with the usual Euler class in $H^2(X_1/S^1)$ over the locus $X_1\subset X$ where $S^1$ acts freely.
\end{lemma}
\begin{proof}
One approach is to define the Gysin sequence for $X\to X/S^1$ as the Gysin sequence for the circle bundle $S^{\infty}\times X \to E_X$ after applying the projection isomorphisms $H_*(S^{\infty}\times X)\cong H_*(X)$ and $H_*(E_X)\cong H_*(X/S^1)$ (using Lemma \ref{Lemma MV argument for fibration}). One can alternatively construct a Gysin sequence directly by applying the inductive argument as in the previous proof as follows, by using the fact that Gysin sequences are natural with respect to maps of spaces. The classical Gysin sequence applies to any subset $Y\subset X$ on which $S^1$ acts freely, since in that case $Y\to Y/S^1$ is an $S^1$-fibre bundle. 
Similarly to the previous proof, we assume that the claim is known for open manifolds for which the stabilisers are at most of size $n-1$, and we then prove it for $X_{\leq n}=U \cup V$ taking $U=X_{\leq n-1}$ and $V$ an $S^1$-invariant tubular neighbourhood of $X_n\subset X_{\leq n}$. By assumption, the claim holds for $U$ and $U\cap V$. For $V$, since $V$ deformation retracts $S^1$-equivariantly onto $X_n$, we can replace $V$ by $X_n$. For $X_n$ we first consider the quotiented action by $S^1_n=S^1/(\Z/n)$, which is a free action and thus yields a Gysin sequence. 
\\ \indent
We claim that the Euler class constructed for the $S^1_n$-action on $X_n$ (so for the circle bundle $X_n \to X_n/S^1_n$), after viewing it as a class in $H^2(V/S^1)$ via the deformation retraction, will pull back to $n$ times the Euler class constructed for $U\cap V \to (U\cap V)/S^1$. One way to see this, is to construct the Euler class by considering a certain edge homomorphism\footnote{\fix{Over real coefficients this is carried out in Bott-Tu \cite[Sec.14 above Proposition 14.33]{BottTu}, where the Euler class can also be constructed as a de Rham form, as an angular form which equals $1$ under integration along fibres.}} in a spectral sequence construction of the Gysin sequence, then along 
a fibre of the circle bundle the Euler class corresponds to a generator of $H^1(S^1)$. The factor $n$ we mentioned above is then caused by the fact that the quotient map $S^1\to S^1_n=S^1/(\Z/n)$ has degree $n$. A more concrete way to prove this, is to define the Euler class in terms of the Thom class of a $2$-disc bundle, namely the mapping cylinder of the projection map of the circle bundle \cite[Below Theorem 4D.10]{Hatcher}. In this case, along a fibre the Euler class corresponds to a generator of $H^2(D^2,S^1)$, and again the factor of $n$ above arises because of the degree of the map $S^1\to S^1_n$.
\\ \indent
We work with (co)homology with coefficients in a field of characteristic coprime to $n$ (since $n$ arises as the size of a stabiliser, otherwise $X_n=\emptyset$ and there is nothing to prove). So we may rescale by $n$ the Euler class in the Gysin sequence for $X_n$ without affecting the exactness of the sequence. After this rescaling, and via the deformation retraction, we therefore obtain a Gysin sequence for $V$ which is functorial with respect to the inclusion $U\cap V \subset V$. Therefore the Mayer-Vietoris argument from the previous proof yields a Gysin sequence for $X_{\leq n}$ that satisfies the claim, as required.
\end{proof}
\begin{lemma}
There is a Poincar\'{e} duality isomorphism $H_*(X/S^1) \cong H^{\dim X - 1 - *}(X/S^1)$, under which cap product on homology corresponds to cup product on cohomology.
\end{lemma}
\begin{proof}
\fixtwo{Recall that there is a proof of Poincar\'{e} duality by using an inductive Mayer-Vietoris sequence argument (see \cite[Lemma 3.36]{Hatcher} or \cite[Lemma 5.6]{BottTu}), where the inductive Poincar\'{e} duality statement is reformulated for open orientable manifolds $U$ by using compactly supported cohomology: $H_*(U)\cong H_c^{\dim U - *}(U)$.
%
%
In our setup, we consider the inductive Mayer-Vietoris argument from the proof of Lemma \ref{Lemma MV argument for fibration}, and we prove inductively on $n$ the Poincar\'{e} duality statement for open orientable manifolds whose stabilisers have size at most $n-1$. 
\\ \indent
The initial step of the induction uses the known Poincar\'{e} duality statement $H_*(U/S^1)\cong H_c^{\dim U -1- *}(U/S^1)$ for the open subsets $U\subset X$ on which the $S^1$-action is free (so $U/S^1$ is an open orientable manifold).
In the inductive step, we consider an $S^1$-invariant tubular open neighbourhood $T$ of $X_n\subset X_{\leq n}$, and we need to justify the Poincar\'{e} duality statement for $T/S^1$.
Via the exponential map, we may view $T$ as an open neighbourhood of the zero section of the normal bundle $V \to X_n$ of the submanifold $X_n\subset X_{\leq n}$. 
Observe that $X_n/S^1$ can be identified geometrically with the smooth oriented (but possibly non-compact) manifold $Y=X_n/S_n^1$, where $S^1_n=S^1/(\Z/n)$ has quotiented out the subgroup generated by $1/n\in \R/\Z=S^1$.
A neighbourhood of $Y\subset X_{\leq n}/S^1$ can be identified via the exponential map with an open neighbourhood of the ``zero section'' of the fibre bundle
\begin{equation}\label{Eq fiber over XmodS1}
\pi: V/S^1 \to X_n/S^1_n=Y,
\end{equation}
 where $S^1$ acts on $V$ by $d\psi_t$, where $\psi_t$ is the $S^1$-flow for time $t\in S^1=\R/\Z$. Notice that \eqref{Eq fiber over XmodS1} is not quite a vector bundle: the fiber over a point $[x]$ is\footnote{\fix{Observe that the geodesic $\gamma(s)=\exp_x(s\cdot v)$ maps via $\psi_{1/n}$ to the geodesic $\exp_x(s\cdot d\psi_{1/n}v)$.}} the quotient $V_x/\Gamma_x$ of the vector space $V_x$ by the finite cyclic $n$-group $\Gamma_x$ generated by $d_x\psi_{1/n}$. We also remark that $X_n$ and $Y$ need not be orientable (although they are in our applications by Remark \ref{Remark the Xn are orientable}).\\
 \indent
  Let $\mathcal{O}_{Y}$ denote the orientation sheaf for $Y$, and $\mathcal{O}_{F}$ the orientation sheaf for the fiber bundle \eqref{Eq fiber over XmodS1} (the latter being the orientation sheaf associated%
\footnote{\fixtwo{There are various equivalent ways to define orientation sheaves (e.g. \cite[Sec.IV.7.9]{Bredon}). In the approach of Bott-Tu \cite[End of Sec.7]{BottTu} the orientation sheaf for the bundle $V\to X_n$ is a line bundle on $X_n$ whose transition maps are multiplication by the sign of the determinant of the Jacobian of the transition functions used for the bundle $V$. This descends to a real line bundle $\mathcal{O}_{F}$ on $Y$, as the $S^1$-action is orientation-preserving.}}
 to the vertical tangent bundle $\ker (d\pi)$ of $V$). By construction
their tensor product,
$$\pi^*(\mathcal{O}_{Y})\otimes \mathcal{O}_{F} \cong \mathcal{O}_{V/S^1},$$
recovers the orientation sheaf for $V/S^1$, which is a \emph{constant} sheaf, since a chosen orientation on $X$ determines an orientation for the total space of $V$ and thus for $V/S^1$. Let $r=\dim X-\dim X_n$ denote the rank of $V\to X_n$ (in this proof, $\dim$ will denote the real dimension). Let $V'=V\setminus (\textrm{zero section})$. We now consider the following diagram:
$$
\xymatrix@R=14pt{
  H^{p+r}_c(V/S^1)\cong H^{p+r}_c(V/S^1,V'/S^1;\mathcal{O}_{V/S^1})\qquad\qquad
 \ar@{->}_{\cong}[d] 
 &
  H_{\dim Y-p}(V/S^1)
 \ar@{<-}^{\cong}[d] 
 \\
 H^p_c(Y;\mathcal{O}_{Y})
  \ar@{->}^{\cong}[r] 
 &
  H_{\dim Y-p}(Y)
}
$$
The top-left cohomology groups in the above diagram are isomorphic because $\mathcal{O}_{V/S^1}$ is a constant sheaf.
The bottom horizontal map is the Poincar\'{e} duality isomorphism for the smooth manifold $Y$ (the orientation sheaf corrects the possibility that $Y$ may not be orientable). The right vertical map is an isomorphism since the spaces are homotopy equivalent ($V$ is $S^1$-equivariantly contractible onto its zero section $X_n$). The left vertical map is (a mild generalisation of) a version of the non-orientable Thom isomorphism \cite[Sec.IV.7.9]{Bredon} (compare also the simpler \cite[Theorem 7.10]{BottTu}), where the orientation sheaf $\mathcal{O}_{F}$ corrects for the possibility that the bundle $V$ is non-orientable, and where we twisted the Thom isomorphism by the sheaf $\mathcal{O}_{Y}$ (so taking $\mathcal{B}=\mathcal{O}_{Y}$ in \cite[Sec.IV.7.9 Equation (22)]{Bredon}). The version we are using is slightly more general than \cite[Sec.IV.7.9]{Bredon} since \eqref{Eq fiber over XmodS1} is not quite a vector bundle. The proof of \cite[Sec.IV.7.9]{Bredon} for a rank $r$ vector bundle $U \to Y$ is a Leray-Serre spectral sequence argument, using as sheaf over $Y$ the cohomology of the fibre pair $(U,U\setminus (\textrm{zero section}))$. The key observation is that, on a local trivialisation, the fiber directions are modelled on the pair $(D^r,D^r\setminus 0)$ where $D^r$ is the $r$-disc in $\R^r$, and the relative cohomology of that pair is a copy of the base field in degree $r$ (the identification with the base field is canonical up to sign, and it is precisely the orientation sheaf of the bundle that keeps track of signs). In our setup, the local model is the pair $(D^r/\Gamma,(D^r\setminus 0)/\Gamma)$ where $\Gamma$ is a cyclic group of order $n$ acting by orientation preserving maps. The assumption on the characteristic of our base field ensures by Remark \ref{Remark cohomology of finite quotients} that the relative cohomology of that pair is canonically isomorphic, via pull-back by the quotient map, to the relative cohomology of the pair $(D^r,D^r\setminus 0)$. So the Thom isomorphism also holds in our setting. The above diagram thus yields the Poincar\'{e} duality statement for $T/S^1$, namely $H^{q}_c(V/S^1)\cong H_{\dim X -1-q}(V/S^1)$.
\\ \indent
The Thom isomorphism is given by cup product by the Thom class $\tau$ (so $\alpha \mapsto \tau \cup \pi^*\alpha$ in the left vertical map in the diagram), and the analogous statement for locally finite homology is the isomorphism $H^{\mathrm{lf}}_{p+r}(V/S^1) \to H_p(Y;\mathcal{O}_{Y})$ given by cap product by $\tau$ followed by $\pi_*$ (so $\beta \mapsto \pi_*(\tau \cap \beta)$). In particular, a fundamental class $[V/S^1]\in H^{\mathrm{lf}}_{\dim X -1}(V/S^1)$ is determined by requiring that $\pi_*(\tau \cap [V/S^1])\in H_{\dim X_n-1}^{\mathrm{lf}}(Y;\mathcal{O}_{Y})$ is Poincar\'{e} dual to $1\in H^0(Y)$. The Poincar\'{e} duality statement $H^{q}_c(V/S^1)\cong H_{\dim X -1-q}(V/S^1)$ above is then given by cap product by $[V/S^1]$. In particular, by construction $[V/S^1]$ is a fundamental class which restricts to the local orientation generators $\mu_{[x]} \in H_{\dim X - 1}(X/S^1, X/S^1 - \{[x]\})$ at points $[x]\in V/S^1 \subset X/S^1$, and these generators are determined by the orientation of $X/S^1$ that is canonically induced by a chosen orientation of $X$.
\\ \indent
Recall the inductive step in the previous proof involved $X_{\leq n}=U \cup V$ where $U=X_{\leq n-1}$. Now, our inductive hypothesis is that Poincar\'{e} duality holds for $U/S^1$, in the sense that the Poincar\'{e} duality isomorphism can be described by cap product by a locally finite fundamental class $[U/S^1]$ which is consistent with the local orientation generators induced by the orientation of $X/S^1$. Above, we proved that the same statement holds for $V/S^1$. The consistency of the two Poincar\'{e} duality isomorphisms on the overlap $(U\cap V)/S^1$ is guaranteed by the fact that the fundamental classes $[U/S^1]$ and $[V/S^1]$ can be compared (and agree) with the local orientation generators $\mu_{[x]}$ at points $[x]$ in the overlap. The Mayer-Vietoris proof of Poincar\'{e} duality therefore applies, and yields the Poincar\'{e} duality statement for $X_{\leq n}/S^1$, which completes the proof of the inductive step.
} 
\end{proof}
%
%
}
\begin{remark}\label{Remark the Xn are orientable}
\fixtwo{We will use the above results for complex manifolds with $S^1$ actions arising from $\C^*$-actions. In that case, the submanifold $X_n$ of points with stabiliser of size $n$ is automatically a complex submanifold (via the exponential map argument above, $T_p X_n$ corresponds to the \emph{complex} linear subspace of $T_p X$ of vectors with stabiliser of size $n$ for the \emph{complex-linear} linearised action). Once an orientation is chosen for $X_n$, an orientation for the normal bundle of $X_n\subset X$ can also be canonically determined from the chosen orientations for $X_n$ and $X$ (similarly, orientations for quotients by $S^1$ can be determined canonically using the canonical orientation for $S^1$).}
\end{remark}
%
\section{Appendix C: Conley-Zehnder indices}

Let $(C^{2n-1},\xi,\alpha)$ be a {\bf contact manifold} admitting a global contact form: 
%
%
so $\alpha$ is a $1$-form on $C$ such that $\alpha\wedge (d\alpha)^{n-1}$ is a volume form, and 
$\xi = \text{ker}(\alpha)$ is the contact structure. The {\bf Reeb vector field} $Y$ on $C$ is determined by $\alpha(Y)=1$, $d\alpha(Y,\cdot)=0$, and it defines the {\bf Reeb flow}. Let $J$ be a complex structure on $\xi$ compatible with $d\alpha|_{\xi}$.
The {\bf anti-canonical} bundle $\kappa^*=\Lambda_{\C}^{\mathrm{top}}\xi$ is the highest exterior power of this complex bundle, and its dual $\kappa$ is called the {\bf canonical bundle}. Now assume $\kappa$ is trivial and fix%
\footnote{
The argument below would similarly apply if we were given a nowhere zero section $K^*$ of the anti-canonical bundle $\kappa^*$, in which case we use a trivialisation $\tau$ with the property that the induced map $\wedge^{n-1}(\tau)^* : (\R / \ell\Z) \times (\Lambda^{\mathrm{top}}_{\C}\C^{n-1}) \to \kappa_J$
 satisfies  $\Lambda^{n-1}(\tau)^*(z_1 \wedge \cdots \wedge z_{n-1}) = K^*$.} 
 a nowhere zero section $K$ of $\kappa$.

By {\bf Reeb orbit $\gamma$ of length $\ell$} we mean a periodic orbit of period $\ell$ of the Reeb vector field, so $\gamma : \R / \ell \Z \to C$.
Up to homotopy, there is a unique trivialisation
$\tau : \gamma^* \xi \to (\R / \ell\Z) \times \C^{n-1}$ 
so that the corresponding trivialisation
$\Lambda^{n-1} \tau : \gamma^*\kappa \to (\R / \ell\Z) \times (\Lambda^{\mathrm{top}}_{\C}\C^{n-1})^*$ 
satisfies
$(\Lambda^{n-1} \tau)(K) = dz_1 \wedge \cdots \wedge dz_{n-1}$.
Expressing the derivative of the Reeb flow $\phi_t : C \to C$ in the trivialisation at $\gamma(0)$ yields a family of symplectic matrices
\begin{equation} \label{Trivialization Equation}
(\text{pr}_2 \circ \tau|_{\gamma(t)}) \circ D\phi_t \circ (\text{pr}_2 \circ \tau|_{\gamma(0)})^{-1} : \C^{n-1} \to \C^{n-1}
\end{equation}
where $\text{pr}_2 : \R / \ell \Z \times \C^{n-1} \to \C^{n-1}$ is the projection.
The {\bf Conley-Zehnder index} of $\gamma$ is the Conley-Zehnder index of this family, which \fix{is a half-integer satisfying} the following properties (e.g. see \cite{Salamon} \fix{and \cite{Gutt}}):
\begin{CZ}
\item \label{CZ Index of Catenation is sum of CZ indices}
 If $A_t,B_t$ are two paths of symplectic matrices then the Conley-Zehnder index of their catenation
is the sum $CZ(A_t) + CZ(B_t)$.
\item \label{CZ of a direct sum}
If $A_t,B_t$ are two paths of symplectic matrices then $CZ(A_t \oplus B_t) = CZ(A_t) + CZ(B_t)$.
\item \label{CZ does not change after reparameterization}
 \fix{The Conley-Zehnder index is invariant under homotopies with fixed end points.}
\item \label{CZ of a path in unitary line}
The Conley-Zehnder index of $(e^{is})_{s \in [0,t]}$ is
$W(t)$, where
\begin{equation}
\label{Equation W function for CZ}
W \co \R \to \N, \quad
W(t) = \left\{
\begin{array}{ll}
2 \lfloor t/2\pi \rfloor + 1 & \text{if} \ t \notin 2\pi\Z \\
t/\pi & \text{if}  \ t \in 2 \pi \Z
\end{array}
\right. .
\end{equation}
\end{CZ}

Let $\phi_t : C \to C$ be the Reeb flow.
The {\bf linearized return map} associated to the Reeb orbit $\gamma$ of length $\ell$ is the restriction $D\phi_{\ell} : \xi|_{\gamma(0)} \to \xi|_{\gamma(0)}$.

\begin{definition}
\label{Definition Morse-Bott submfd}	
A {\bf Morse-Bott submanifold} $B \subset C$
of length $\ell$ is a submanifold such that:
\begin{enumerate}
\item Through each point of $B$, there is a Reeb orbit of length
$\ell$ contained in $B$.
\item The linearized return map for each such Reeb orbit has $1$-eigenspace equal to $TB \cap \xi|_B$.
%
%
\end{enumerate}
So the Reeb flow satisfies $\phi_t(B)\subset B$ and $\phi_{\ell}|_B=\mathrm{id}$, and the real dimension of the $1$-eigenspace of each return map is $\dim B-1$. 
%
%
We call $\phi_{\ell t}|_B$ the {\bf associated $S^1$-action} on $B$.

\fix{For a connected Morse-Bott submanifold $B\subset C$, we define $\mathrm{CZ}(B)=\mathrm{CZ}(p)\in \Z$, for any point $p\in B$ (independence of the choice of $p$ is a consequence of property \ref{CZ does not change after reparameterization} above).
}
\end{definition}

\begin{remark} Convex symplectic manifolds $M$ (with $J$ as in Sec.\ref{Subsection convex symplectic manifolds}) contain a contact hypersurface $\Sigma$ with a complex splitting $TM=\xi\oplus \C$ where $\C\cong \R Z \oplus  \R Y$ for the vector fields $Z,Y$ defined in Sec.\ref{Subsection convex symplectic manifolds}. If the canonical bundle $\mathcal{K}_M=\Lambda_{\C}^{\mathrm{top}}T^*M$ is trivial, then so is the canonical bundle $\kappa$ for $\Sigma$. 
For a convex symplectic manifold $M^{2n}$ with trivial canonical bundle $\mathcal{K}$ and a choice of trivialisation, our {\bf Conley-Zehnder grading} on the Floer complex $CF^*(H)$ is 
\begin{equation}\label{Equation CZ for Ham case}
\mu(x)=n - \mathrm{CZ}_H(x),
\end{equation}
 where $\mathrm{CZ}_H(x)$ is computed analogously to the above, except now $x$ is a Hamiltonian $1$-orbit for $H$ so we consider the linearisation $D\phi_H^t$ of the Hamiltonian flow in a trivialisation of $x^*TM$ that is compatible with the given trivialisation of $x^*\mathcal{K}$.
If we chose a different trivialisation of $\kappa$, involving a section of $\kappa$ that is obtained from the restriction of the section for $\mathcal{K}_M$ multiplied by a function $f:\Sigma \to \C^*$, then the CZ-index of a Reeb orbit in a free homotopy class $c$, so a conjugacy class of $\pi_1(\Sigma)$, changes by $-2\langle [f],c \rangle$ where $[f]\in H^1(\Sigma,\Z)\cong [\Sigma,\C^*]$, and the grading on $SH^*$ changes by $+2\langle [f],c\rangle$. 
%
%
A similar argument applies to the choice of trivialisation of $\mathcal{K}_M$; that choice will not matter if $M$ is simply connected.
%
\end{remark}

\begin{remark}\label{Remark from CZReeb to CZHam}
\fix{The convention in \eqref{Equation CZ for Ham case}} ensures\footnote{This agrees with \cite[Exercise 2.8]{Salamon} despite the sign, because we use the convention $\omega(\cdot,X_H)=dH$.} that for a $C^2$-small Morse Hamiltonian $H$, a critical point $x$ of $H$ will have Morse index $\mu(x)$. In the notation of Sec.\ref{Subsection convex symplectic manifolds} on the end $\Sigma \times [1,\infty)$, using a radial Hamiltonian $H=h(R)$, a $1$-orbit $x$ in the slice of slope $h'(R)=\ell$ corresponds to a Reeb orbit $\gamma(t)=x(t/\ell)$ of length $\ell$ in $\Sigma$. We pick a basis of sections to trivialize $x^*\xi$ so that together with $Z,Y$ we obtain a trivialisation of $x^*TM\cong x^*\xi\oplus (\R Z\oplus \R Y)$ that is compatible with the given trivialisation of $\mathcal{K}$. If $h''(R)>0$, then the family of symplectic matrices obtained for $\varphi_H^t$ can be identified with the family obtained by \eqref{Trivialization Equation}, together with a shear of type $\left( \begin{smallmatrix} 1 & 0 \\ \mathrm{positive} & 1 \end{smallmatrix}\right)$ in the $(Z,Y)$-plane contributing $\frac{1}{2}$ to $\mathrm{CZ}_H$ \fix{(e.g. see \cite[Prop.4.9]{Gutt})}. Thus $\mathrm{CZ}_H(x)=\fix{\mathrm{CZ}(\gamma)}+\frac{1}{2}$. The correction $+\frac{1}{2}$ however will cancel out, once one takes into account that there is an $S^1$-family of $1$-orbits $x(\cdot+\textrm{constant})$. Indeed if $x$ is transversally non-degenerate then in $CF^*(H)$ it would give rise, after perturbation, to a copy of $H^*(S^1)$ \fix{shifted up by $n-\fix{\mathrm{CZ}(\gamma)}-1$} \cite[Prop.2.2]{CFHW}. In the case of a \fix{connected} Morse-Bott submanifold $S\subset \Sigma$ of orbits, using a Morse-Bott Floer complex for $H$ as in \cite{Bourgeois-Oancea,Bourgeois-Oancea2}, one analogously obtains a copy of $H^*(S)$ shifted up in degree by 
\begin{equation}\label{Equation Morse-Bott shift}
\mu(S)=n-\mathrm{CZ}(S)-\tfrac{1}{2}-\tfrac{1}{2}\dim S
\end{equation}
because half of the signature of the Hessian of an auxiliary Morse function $f_S:S\to \R$ used to perturb the moduli space of $1$-orbits $S$ would contribute to $\mathrm{CZ}_H$ \cite[Section 3.3]{OanceaEnsaios}. In our conventions, a radial Hamiltonian with $h''(R)>0$ on $\C^n$ gives rise to a Morse-Bott submanifold $S=S^{2n-1}$ when the flow undergoes one full rotation, and $\mathrm{CZ}(S)=2n$, 
so \eqref{Equation Morse-Bott shift} equals
%
$\mu(S)=-2n$ (the grading of $\min f_S$), and $\max f_S\in H^{\mathrm{top}}(S)$ contributes in grading $-1$ and has non-trivial Floer differential exhibiting the unit $1\in SH^0(\C^n)$ as a boundary.
\end{remark}

\section{Appendix D: the $F$-filtration and positive symplectic cohomology}
\label{Section Appendix D}

\subsection{Convex symplectic manifolds}
\label{Subsection convex symplectic manifolds}
We consider non-compact symplectic manifolds $(M,\omega)$ where $\omega$ is allowed to be non-exact, but outside of a bounded domain $M_0\subset M$ there is a symplectomorphism
$
(M\setminus M_0,\omega|_{M\setminus M_0}) \cong (\Sigma\times [1,\infty),d(R\alpha)),
$
where $(\Sigma,\alpha)$ is a contact manifold, and $R$ is the coordinate on $[1,\infty)$ ({\bf radial coordinate}).
The {\bf Liouville vector field} $Z=R\partial_R$ is defined at infinity via $\omega(Z,\cdot)=\theta$. The {\bf Reeb vector field} $Y$ on the $\Sigma$ factor is defined by
$
\alpha(Y)=1, \; d\alpha(Y,\cdot)=0.
$
By ``the'' contact hypersurface $\Sigma \subset M$ we mean the level set $R=1$. By {\bf Reeb periods} we mean the periods of Reeb orbits on this $\Sigma$.
The almost complex structure $J$ is always assumed to be $\omega$-compatible and of {\bf contact type at infinity} (meaning $JZ=Y$ or equivalently $\theta=-dR\circ J$).
Let $H: M \to \R$ be smooth. 
By {\bf 1-orbits} we mean $1$-periodic Hamiltonian orbits (i.e.\,using the Hamiltonian vector field $X_H$ where $\omega(\cdot,X_H)=dH$).
The data $(H,J,\omega)$ determines the Floer solutions which define the Floer chain complex $CF^*(H)$. We recall (e.g.\,see  \cite{Ritter3}) that when $H$ is a linear function of $R$ at infinity of slope different than all Reeb periods, the $R$-coordinate of Floer solutions satisfies a maximum principle. This ensures that its cohomology $HF^*(H)$ is defined (to avoid technicalities in defining Floer cohomology, one assumes $M$ satisfies a weak monotonicity condition \cite{Ritter2}, for example this holds if $c_1(M)=0$).
\begin{figure}[ht]
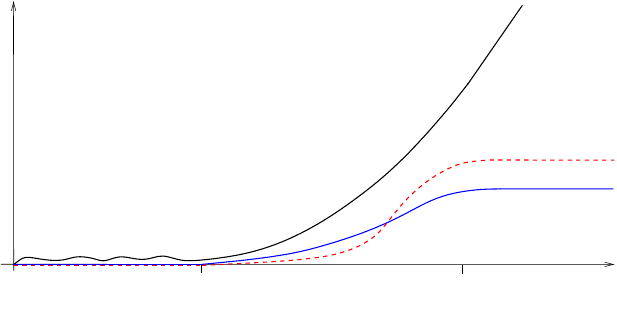
\caption{An illustration of the graphs of $H$, $\phi$, and $f$.}
\label{Figure Graph H,phi,f}
\end{figure}
\subsection{Admissible Hamiltonians}
\label{Subsection Admissible Hamiltonians}
\fix{Observe Figure \ref{Figure Graph H,phi,f}. We assume that for some $R_1>R_0>0$,} 
\begin{enumerate}
\item \label{Appendix item 1} $J$ is of contact type for $R\geq R_0$,
\item $H=h(R)$ only depends on the radial coordinate for $R\geq R_0$,
\item $h'(R_0)>0$ is smaller than the minimal Reeb period,
\item  \label{Appendix item 4} $h''(R_0)>0$, and $h''(R)\geq 0$ for $R\geq R_0$ (so $h'$ is increasing),
\item if $h''(R)=0$ for some $R\geq R_0$ then we require that $h'(R)$ is not a Reeb period,
\item \label{Appendix item 6} \fix{for $R\geq R_1$,} $h'(R)$ is a constant (and not equal to a Reeb period),
\item \label{Appendix item 7 get QH} For $R\leq R_0$, $H$ is Morse and $C^2$-small so that all $1$-orbits in $R\leq R_0$ are constant (i.e.\,critical points of $H$) and the Floer \fix{complex} generated by these $1$-orbits is quasi-isomorphic to the Morse complex, and on cohomology recovers $QH^*(M,\omega)$.
\end{enumerate}

The last condition is not strictly necessary, but one can apply a Floer continuation isomorphism to homotope $H$ on $R \leq R_0$ to ensure that condition. \fix{That the complex in \eqref{Appendix item 7 get QH} is well-defined follows from a maximum principle and it is known that one recovers quantum cohomology \cite{Ritter3}. Moreover, the filtration argument in Sec.\ref{Subsection Overview} will show that this complex is a subcomplex $CF^*_0(H)$ of the Floer complex $CF^*(H)$ of $H$.} 

\fix{Let $H_s: M \to \R$ depend on $s\in \R$ on a compact subset of $\R$, with $H_s=H_-$ for $s\ll 0$ and $H_s=H_+$ for $s\gg 0$, where $H_{\pm}: M \to \R$ are admissible (i.e. satisfy the above conditions).
Then $H_s$ is an {\bf admissible homotopy} of Hamiltonians if $H_s=h_s(R)$ on $R\geq R_0$ such that
\begin{enumerate}
\addtocounter{enumi}{7}
\item $\partial_s h_s'\leq 0$ (to ensure that the maximum principle for Floer solutions applies \cite{Ritter3}), 
\item each $h_s$ satisfies the above conditions \eqref{Appendix item 1}-\eqref{Appendix item 4},
\item and $h_s'$ is constant for $R\geq R_1$ (but may depend on $s$, or be equal to a Reeb period).
\end{enumerate}}%
\noindent Also $J=J_s$ may vary with $s$, subject to the above compatibility and contact type conditions.

\subsection{Cut-off function}
\fix{Observe Figure \ref{Figure Graph H,phi,f}.} Let $\phi: \R \to [0,1]$ be a smooth and increasing function such that
\fix{
\begin{enumerate}
\item $\phi=0$ for $R\leq R_0$,\; $\phi>0$ for $R>R_0$,
\item\label{Item phi' positive} $\phi'>0$ for $R_0< R < R_1$ (recall $h'(R)$ is constant for $R\geq R_1$),
\item and $\phi=1$ for large $R$.
\end{enumerate} 
}
One could omit \eqref{Item phi' positive} at the cost of losing the strictness of the filtration in Theorem \ref{Theorem filtration 1}.

The cut-off function determines an exact two-form on $M$,
$$
\eta = d(\phi(R)\alpha) = \phi(R)\, d\alpha + \phi'(R)\, dR\wedge \alpha,
$$
and an associated $1$-form $\Omega_{\eta}$ on the free loop space $\mathcal{L}M=C^{\infty}(S^1,M)$ given by
\begin{equation}\label{Eqn1Form}
\Omega_{\eta}\co T_x\mathcal{L}M = C^{\infty}(S^1,x^*TM)\to \R,\;\; \xi \mapsto -{\textstyle\int} \eta(\xi,\partial_t x - X_H)\, dt.
\end{equation}
\begin{lemma}\label{Lemma 1form is negative on Floer}
The $1$-form $\Omega_{\eta}$ is negative (or zero) on Floer trajectories $u: \R \times S^1 \to M$. 
\end{lemma}
\begin{proof} Substituting the Floer equation $\partial_t u - X_H = J\partial_s u$, and abbreviating $\rho=R\circ u$,
$$
\begin{array}{rcl}
\eta(\partial_s u, \partial_t u - X_H)
& = &
\eta(\partial_s u,J\partial_s u)
\\
& = &
\phi(\rho) \cdot d\alpha(\partial_s u,J\partial_s u) + \phi'(\rho) \cdot (dR \wedge \alpha)(\partial_s u,J\partial_s u)
\\
& = &
\textrm{positive}\cdot \textrm{positive} + \textrm{positive}\cdot (dR \wedge \alpha)(\partial_s u,J\partial_s u)
\end{array}
$$
To estimate the last term, we may assume that $R\geq R_0$ since $\phi'=0$ otherwise. Since $J$ is of contact type for $R\geq R_0$, we can decompose
$$
\partial_s u = C\oplus yY \oplus zZ \in \ker \alpha \oplus \R Y \oplus \R Z
$$
where $Z=R\partial_R$. Thus: \fix{$dR(\partial_s u) = \rho z$} and $\alpha(J\partial_s u) = \alpha(JzZ) = \alpha(zY)=z$. Using $\theta=R\alpha$,
$$
(dR \wedge \alpha)(\partial_s u,J\partial_s u)
=
dR(\partial_s u) \alpha(J\partial_s u) 
+ \alpha(\partial_s u) \theta(\partial_s u)
=
\fix{\rho z^2 + \rho y^2 \geq 0.}
$$
%
%
%
%
The claim then follows from \phantom\qedhere
%
%
\begin{equation}\label{Eqn positivity explicitly}
\eta(\partial_s u, \partial_t u - X_H)=\phi(\rho)\cdot |C|^2 + \rho\,\phi'(\rho)\cdot (z^2 +  y^2)\geq 0. \rlap{\qquad\qedsymbol}
\end{equation}
\end{proof}

\subsection{Filtration functional}
\label{Subsection Overview}
\fix{Observe Figure \ref{Figure Graph H,phi,f}.} Let $f: \R \to [0,\infty)$ be the smooth function defined by 
$$\textstyle f(R) = \int_0^R \phi'(\tau)\, h'(\tau)\, d\tau.$$
Notice it is a primitive for $\phi'(R)\, h'(R)\, dR$, and satisfies\footnote{using that $\phi'\geq 0,h'\geq 0$, and those are strict just above $R=R_0$, and that $\phi'=0$ for large $R$.}
\begin{enumerate}
\item $f=0$ for $R\leq R_0$, $f>0$ for $R>R_0$,
\item \fix{and} $f$ is bounded.
\end{enumerate}

Define the {\bf filtration functional} $F: \mathcal{L}M \to \R$ on the free loop space by
$$
F(x)=-\int_{S^1} x^*(\phi \alpha) + \int_{S^1} f(R\circ x)\, dt.
$$
where $R\circ x$ is the $R$-coordinate of $x(t)$. 
\begin{theorem}\label{Theorem filtration 1}
The filtration functional $F$ satisfies:
\begin{enumerate}
\item \label{Item Theorem filtration 1 Exactness} \emph{Exactness}: $F$ is a primitive of $\Omega_{\eta}$, so
thus $dF\cdot \xi = - \int_{S^1} \eta(\xi,\partial_t x-X_H)\, dt$.
\item \label{Item Theorem filtration 1 Negativity} \emph{Negativity}: $dF\cdot \partial_s u \leq 0$ for any Floer trajectory $u$ for $(H,J,\omega)$, thus $F(x_-)\geq F(x_+)$ if $u$ travels from $x_-$ to $x_+$.
\item \label{Item Theorem filtration 1 Separation} \emph{Separation}: $F=0$ on all loops in $R\leq R_0$, and $F<0$ on the $1$-orbits in $R\geq R_0$.
\item \label{Item Theorem filtration 1 Compatibility} \emph{Compatibility}: $F$ decreases along any Floer continuation solution $u$ for any admissible homotopy of $(H,J)$.
\item \label{Item Theorem filtration 1 Strictness} \emph{Strictness}:
$F(x_-)>F(x_+)$ for any Floer trajectory joining distinct orbits $x_-$, $x_+$ with $R(x_+)\geq R_0$.
\end{enumerate}

Thus $F$ determines a filtration on the Floer chain complex for a given admissible pair $(H,J)$, and this filtration is respected by Floer continuation maps for admissible homotopies of $(H,J)$.
We use cohomological conventions,\footnote{$x_-$ contributes to $\partial x_+$.} so the Floer differential increases the filtration.
\end{theorem}

We prove this in Section \ref{Subsection Proof of filtration theorem}. Recall\footnote{since $X_H=h'(R)Y$ for $R\geq R_0$, a $1$-orbit $x(t)$ corresponds to the Reeb orbit \fix{$x(t/\ell)$ of period $\ell=h'(\rho)$.}} that non-constant $1$-orbits in $R=\rho\geq R_0$ are in $1$-to-$1$ correspondence with closed Reeb orbits in $\Sigma$ of period $\tau=h'(\rho)$; the filtration value is
\begin{equation}\label{Eqn period function T}
F(x) = T(\rho)\equiv -\phi(\rho)\, h'(\rho) + f(\rho).
\end{equation}
Note $T:[0,\infty)\to \R$ satisfies the following properties, arising from the conditions on $\phi,h$,
\begin{enumerate}
\item $T(\rho)=0$ for $\rho \leq R_0$,
\item $T$ is decreasing,\footnote{\fix{Indeed $T'(\rho)= - \phi(\rho)\, h''(\rho) \leq 0$.}}
\item $T(\rho)<0$ for $R> R_0$,\footnote{by integrating $T'$, using that $\phi,h''>0$ just above $R_0$.}
\item $T$ strictly decreases near $\rho\geq R_0$ if $h'(\rho)$ is a Reeb period.\footnote{since then $h''(\rho)>0$.}
\end{enumerate}

From this, we also deduce \fix{that $T$ is} a strict filtration for $R\geq R_0$, as follows.

\begin{corollary}\label{Corollary R is also a filtration}
If $u$ is a Floer trajectory joining distinct $1$-orbits $x_-,x_+$, with $R(x_+)\geq R_0$,\footnote{Since $H=h(R)$ is time-independent for $R\geq R_0$, each $1$-orbit of $H$ arises as an $S^1$-family of orbits, due to the choice of the starting point of the orbit. With a time-dependent perturbation localized near the orbit, one can split the $S^1$-family into two non-degenerate $1$-orbits (lying in the same $R$-coordinate slice), such that locally there are precisely two rigid Floer trajectories connecting these orbits, and they lie in the same $R$-slice and give cancelling contributions to the Floer differential (so this local Floer complex computes $H^*(S^1)$ up to a degree shift). For the purposes of Floer cohomology we can ignore these two Floer trajectories, and with this proviso the above Corollary continues to hold and implies that there cannot be other Floer trajectories connecting the two perturbed $1$-orbits. The same argument holds more generally when there is a Morse-Bott manifold $\mathcal{O}_a$ of orbits, in which case the local Floer complex computes $H^*(\mathcal{O}_a)$ up to a degree shift, and the Corollary refers to Floer trajectories that are not already accounted for in this local Floer cohomology.
}
$$
R(x_-) < R(x_+).
$$
In particular, given a $1$-orbit $y$ in $R=\rho\geq R_0$, the Floer differential $\partial y$ is determined by the $1$-orbits in $R<\rho$ and the Floer trajectories that lie entirely in $R\leq \rho$.
\end{corollary}
\begin{proof}
Note $T(R(x_-))=F(x_-)> F(x_+)= T(R(x_+))$, and now use that $T$ is decreasing (and use the Strictness in Theorem \ref{Theorem filtration 1}). The final claim  follows from the maximum principle.
\end{proof}

Define $CF^*_0(H)\subset CF^*(H)$ to be the subcomplex generated by $1$-orbits with $F\geq 0$ (which by Sec.\ref{Subsection Admissible Hamiltonians}.\eqref{Appendix item 7 get QH} is quasi-isomorphic to $QH^*(M)$), and let $CF^*_+(H)$ be the corresponding quotient complex.
Define {\bf positive symplectic cohomology} as the direct limit $SH^*_+(M)=\varinjlim HF^*_+(H)$ of the cohomologies of $CF^*_+(H)$.

\begin{corollary}
Positive symplectic cohomology does not depend on the choice of $\phi$.
\end{corollary}
\begin{proof}
This follows from the fact that Corollary \ref{Corollary R is also a filtration} does not depend on the choice of $\phi$, and the Floer theory for $(H,J,\omega)$ does not use $\phi$.
\end{proof}

\begin{corollary}\label{Corollary LES for H SH and SHplus}
There is a long exact sequence of $\K$-algebra homomorphisms
$$
\cdots \to QH^*(M) \stackrel{c^*}{\to} SH^*(M) \to SH^*_+(M) \to QH^{*+1}(M) \to \cdots
$$
%
In particular, if $SH^*(M)=0$ then $SH^*_+(M)\cong QH^{*+1}(M)$ canonically as vector spaces.

In the equivariant case, there is a long exact sequence of $\K[\![u]\!]$-module homomorphisms
$$
\cdots \to H^*(M)\otimes_{\K} \F \stackrel{c^*}{\to} ESH^*(M) \to ESH^*_+(M) \to H^{*+1}(M)\otimes_{\K} \F \to \cdots
$$
\end{corollary}
\begin{proof}
The subcomplex yields the long exact sequence $QH^*(M)\to HF^*(H) \to HF^*_+(H) \to QH^{*+1}(M)$ which, using the Compatibility in Theorem \ref{Theorem filtration 1}, yields the claim by taking the direct limit over continuation maps, as we make the final slope of $h$ increase.

The equivariant setup follows analogously from the filtration, provided the $H_z$ in Sec.\ref{Subsection The construction of the deltak in Floer theory} are chosen to belong to the class of admissible Hamiltonians (we fix the cut-off function $\phi$). For Theorem \ref{Theorem filtration 1} \eqref{Item Theorem filtration 1 Compatibility} to apply to the Floer solutions in the equivariant construction, we need $\partial_s h_{w(s)}'\leq 0$ for $R\geq R_0$, where $w:\R\to \C\P^{\infty}$ is any $-\nabla f$ trajectory (this $f$ refers to the function \eqref{Equation f on CPinfinity}). Here $h_z=h(\cdot,z)$, for $h:[R_0,\infty)\times \C\P^{\infty}\to \R$, is $H_z$ in the region $R\geq R_0$. 
We achieve this by requiring\footnote{As we require $J$ to be of contact type not just for $R\geq R_1$ but also on the region $R_0 \leq R \leq R_1$ (due to Lemma \ref{Lemma 1form is negative on Floer}), on this region we cannot perturb $J$ in the $\mathrm{span}(Z,Y)$ directions, so the standard transversality argument \cite[Prop.3.4.1]{McDuffSalamon} may fail there (this issue does not arise for $R\geq R_1$ as Floer solutions do not reach $R\geq R_1$ due to the maximum principle). 
%
%
%
%
%
If the Floer solution $u$ enters $R<R_0$, it suffices to perturb $J$ there. So transversality is only problematic if $u$ is contained in $R_0\leq R \leq R_1$ and the image of $du$ lies in $\mathrm{span}(Z,Y)$. This implies that $u$ lands inside a cylinder in $\Sigma \times [1,\infty)$ and the ends of $u$ wrap \fix{different amounts of time} around the two boundary circles of that cylinder (as $h'$ increased), which is not allowed \fix{for homotopical reasons.}
Alternatively (without using that $h'$ is monotone) one could allow small enough perturbations of $J$ at injective points of Floer solutions in $R_0\leq R\leq R_1$, so as to ensure that the inequalities in Lemma \ref{Lemma 1form is negative on Floer} remain strictly negative at those points. 
%
%
This way the filtration construction will still hold.} that $h_z$ is independent of $z$ for $R\geq R_0$.
%
%
\end{proof}

\subsection{Proof of Theorem \ref{Theorem filtration 1}}
\label{Subsection Proof of filtration theorem}

Define the $\phi$-action by
$$
\mathcal{A}_{\phi}: \mathcal{L}M \to \R, \; \mathcal{A}_{\phi}(x) = -{\textstyle\int_{S^1}} \, x^*(\phi(R)\alpha).
$$
This vanishes on loops $x$ which lie entirely in $R\leq R_0$. Suppose now $x$ is a $1$-orbit that intersects the region $R\geq R_0$. Then $x$ is forced to lie entirely in the region $R\geq R_0$, indeed it lies in some fixed level set $R=\rho$ since $X_H=h'(R)Y$. In this case, $\mathcal{A}_{\phi}(x)=-\phi(\rho)\, h'(\rho).$

A simple calculation shows that 
\begin{equation}\label{Lemma dAphi is the first piece of 1form}
d\mathcal{A}_{\phi} \cdot \xi = - {\textstyle\int_{S^1}} \,\eta(\xi,\partial_t x)\, dt.
\end{equation}
Finally, we need to ensure the exactness of the second term in \eqref{Eqn1Form},
$$
\textstyle \int \eta(\xi,X_H)\, dt = \int  \phi'(\rho) h'(\rho)\, d\rho(\xi)\, dt,
$$
where we used the equality $X_H=h'(R)Y$ (and the fact that $\eta=0$ and $\phi'=0$ where this equality fails).
By definition, $F(x)=\mathcal{A}_{\phi}(x) + \int_x f\circ R$ so \eqref{Lemma dAphi is the first piece of 1form} and the choice of $f$ imply claim \eqref{Item Theorem filtration 1 Exactness}. Lemma \ref{Lemma 1form is negative on Floer} and claim \eqref{Item Theorem filtration 1 Exactness} imply claim \eqref{Item Theorem filtration 1 Negativity}.

In claim \eqref{Item Theorem filtration 1 Separation}, that $F$ vanishes on loops inside $R\leq R_0$ follows from $\phi(R)=f(R)=0$. On a $1$-orbit $x$ lying in $R=\rho$ the value of $F$ is \eqref{Eqn period function T}.
The rest of claim \eqref{Item Theorem filtration 1 Separation} follows from the properties of $T(\rho)$ mentioned under \eqref{Eqn period function T}.
To show claim \eqref{Item Theorem filtration 1 Compatibility}, let 
\fix{
$$f_s(R)= \int_0^R \phi'(\tau)\, h_s'(\tau)\, d\tau,\qquad F_s(x)=\mathcal{A}_{\phi}(x)+\int_x f_s\circ R.$$
}%
Then 
 \fix{
 $$d_xF_s\cdot \xi = -\int \eta(\xi,\partial_t x - X_{H_s})\, dt.$$
 }%
As in Lemma \ref{Lemma 1form is negative on Floer}, one checks $d_u F_s \cdot \partial_s u\leq 0$ on Floer continuation solutions $u$. Now
\fix{
$$
\partial_s (F_s\circ u) = d_u F_s \cdot \partial_s u + (\partial_s F_s) \circ u
$$
}%
where $(\partial_s F_s)(x) = \int_x (\partial_s f_s) \circ R$. But 
\fix{
$$\partial_s f_s (R) = \int_0^R \phi'(\tau)\, \partial_s h_s'(\tau)\, d\tau \leq 0,$$
}%
using that $h_s$ is admissible ($\partial_s h_s'\leq 0$). So $\partial_s (F_s \circ u) \leq 0$.
To prove claim \eqref{Item Theorem filtration 1 Strictness}, note that in \eqref{Eqn positivity explicitly}, if $\eta(\partial_s u,\partial_t u - X_H)=0$ for some $R\geq R_0$, then $C,z,y$ vanish as $\phi(R),\phi'(R)> 0$. Thus $\partial_s u=0$ and so $\partial_t u = X_H$ (by the maximum principle, $u$ does not enter the region $R\geq R_1$). But $x_-,x_+$ are distinct, so $\partial_s u$ cannot be everywhere zero, so strict negativity holds in Lemma \ref{Lemma 1form is negative on Floer} for some $s\in \R$. $\qed$

\section{Appendix E: Morse-Bott spectral sequence}
\label{Subsection Morse-Bott spectral sequence}

The Morse-Bott spectral sequences that we use in the paper are analogous to those due to Seidel \cite[Eqns.(3.2),(8.9)]{Seidel} that arose from $S^1$-actions on Liouville manifolds. 
Morse-Bott techniques in Floer theory go back to Po\'{z}niak \cite{Pozniak} and Bourgeois \cite{BourgeoisPhD}. For Liouville manifolds (i.e.\;exact convex symplectic manifolds), Bourgeois-Oancea \cite{Bourgeois-Oancea,Bourgeois-Oancea2} showed that the Morse-Bott Floer complex for time-independent Hamiltonians $H$ (assuming transversal non-degeneracy of $1$-orbits) computes the same Floer cohomology as when using a time-dependent perturbation of $H$. The Morse-Bott Floer complex introduces auxiliary Morse functions on the copies of $S^1$ arising as the initial points of $1$-orbits, and uses the critical points of the auxiliary Morse functions as generators, with an appropriate degree shift. The differential now counts {\bf cascades} i.e.\,alternatingly following the flows of the negative gradients of the auxiliary functions or following Floer solutions that join two $1$-orbits. This is the natural complex that would arise from a limit, as one undoes small time-dependent perturbations of $H$ localised near those copies of $S^1$ in $M$.
The Morse-Bott complex admits a natural filtration by the action functional. As the functional decreases along Floer solutions, the filtration is exhausting and bounded below, so it induces a spectral sequence converging to $HF^*(H)$ whose $E_1^{pq}$-page consists of the cohomologies of the $S^1$ copies shifted appropriately in degree. It was shown by Cieliebak-Floer-Hofer-Wysocki \cite[Prop.2.2]{CFHW} that a suitable time-dependent perturbation of $H$ localised near such an $S^1$-copy creates a local Floer complex in two generators whose cohomology agrees with the (Morse-Bott) cohomology of $S^1$.

Kwon and van Koert \cite[Appendix B]{KwonvanKoert} carried out a detailed construction of Morse-Bott spectral sequences for symplectic homology of Liouville domains with periodic Reeb flows. So we restrict ourselves to explaining how these ideas generalise for convex symplectic manifolds $M$ (Sec.\ref{Subsection convex symplectic manifolds}), using admissible Hamiltonians $H$ and our new filtration $F$ from Appendix D (our filtration replaces the role of the action functional, which is multi-valued in our setup).

{\bf Assumption.} \emph{The subsets of Reeb orbits in $\Sigma$ are Morse-Bott submanifolds} (see Def.\ref{Definition Morse-Bott submfd}).

Recall the non-constant $1$-orbits $x$ of $H$ arising at slope $h'=\tau$ correspond to Reeb orbits $y(t)=x(t/\tau)$ in $\Sigma$ of period $\tau$. Consider the {\bf slices} $$\mathcal{S}(c)=\{m: R(m)=c\}\subset M,$$ i.e.\;the subset of points where the radial coordinate $R$ of Sec.\ref{Subsection convex symplectic manifolds} equals a given value $c$. Let $R_{-1}<R_{-2}<\cdots$ be the values of $R$ for which $1$-orbits of $H$ appear in $\mathcal{S}(R)$, equivalently the slopes $\tau_p=h'(R_p)$ for $p<0$ are the Reeb periods less than the final slope of $h'$.

Let $\mathcal{O}_p=\mathcal{O}_{p,H}$ be the {\bf moduli space of parametrized $1$-orbits} of $H$ in $\mathcal{S}(R_p)$. These have $F$-filtration value $F_p=-\phi(R_p)h'(R_p)+f(R_p)$ by \eqref{Eqn period function T}. By construction, 
$$
0>F_{-1}>F_{-2}>F_{-3}>\cdots
$$
By considering the initial point of the orbits, we view $\mathcal{O}_p\subset \mathcal{S}(R)$ as a subset, which can be identified with the Morse-Bott submanifold $B_p\subset \Sigma$ of initial points of the Reeb orbits of period $\tau_p$. Denote by $\mathcal{O}_0$ the Morse-Bott manifold of constant orbits of $H$, i.e.\,the critical locus of $H$ (which by admissibility are the $1$-orbits of $H$ in $R\leq R_0$, and determine a Morse-Bott complex for $M$).
We define $F_0=0$, which is the filtration value for $\mathcal{O}_0$, and by convention we define $F_p=p$ for positive integers $p\geq 1$ (there are no $1$-orbits with filtration value $F>0$).

Abbreviate by $\mathbf{k=p+q}$ the {\bf total degree}. 
Let $C^*$ denote the Floer complex $CF^*_+(H)$ or $CF^*(H)$. The filtration is defined by letting $F^p(C^k)$ be the subcomplex generated by $1$-orbits with filtration function value $F\geq F_p$ (in particular, $F^p(C^k)=0$ for $p> 0$ since $F\leq 0$ on all $1$-orbits). Recall the spectral sequence for this filtration has $E_0^{pq}=F^p(C^k)/F^{p+1}(C^k)$. As the filtration is exhaustive and bounded below, it yields convergent spectral sequences
$$
\begin{array}{lll}
E_1^{pq} \Rightarrow HF^*_+(H) &\textrm{where} & E_1^{pq}=HF^k_{\mathrm{loc}}(\mathcal{O}_p,H) \textrm{ for }p<0, \textrm{and 0 otherwise}
\\[1mm]
E_1^{pq} \Rightarrow HF^*(H) & \textrm{as above, except} & E_1^{0q}=H^q(M)
\end{array}
$$
where it is understood, that $E_1^{pq}=0$ for $p\ll 0$, as there are only finitely many $\mathcal{O}_p$ for $H$, and we remark that the same spectral sequences exist in the equivariant setup after replacing $HF$ by $EHF$. Above, $HF^*_{\mathrm{loc}}(\mathcal{O}_p,H)$ refers to the cohomology of the local Morse-Bott Floer complex generated by $\mathcal{O}_p$. By construction, its differential only counts cascades which do not change the filtration value, so the Floer solutions stay trapped in the slice $\mathcal{S}(R_p)$. If one were to make a very small time-dependent perturbation of $H$ supported near $\mathcal{S}(R_p)$, the argument in \cite[Prop.2.2]{CFHW} and \cite[Sec.3.3]{OanceaEnsaios} would show that this is quasi-isomorphic to the local Floer complex for that slice, where one only considers Floer solutions whose filtration value stays bounded within a small neighbourhood of the value $F=F_p$.

Let $B_{p,c}$ label the connected components of $B_p$ (and the labelling by $c$ depends on $p$). These have a Conley-Zehnder index $\mathrm{CZ}(B_{p,c})$ and a grading $\mu(B_{p,c})=n-\mathrm{CZ}(B_{p,c})$ (Appendix C).

\begin{lemma}\label{Lemma local HF is H}
Assume that the linearised Reeb flow is complex linear with respect to a unitary trivialisation of the contact structure along every periodic Reeb orbit in $\Sigma$. Then
\begin{equation}\label{Equation Local Floer cohomology}
HF^*_{\mathrm{loc}}(\mathcal{O}_p,H)\cong \bigoplus_{c} H^{*-\mu(B_{p,c})}(B_{p,c}).
\end{equation}
\end{lemma}

Kwon and van Koert give a detailed discussion of this in \cite[Prop.B.4.]{KwonvanKoert} and explain in \cite[Sec.B.0.2]{KwonvanKoert} that there is an obstruction in $H^1(\Sigma,\Z/2)$ to \eqref{Equation Local Floer cohomology} caused by orientation signs. Indeed \eqref{Equation Local Floer cohomology} always holds if one uses the local system of coefficients on $B_{p,c}$ determined by that $H^1$-class. This obstruction vanishes under \fix{the assumptions of Lemma \ref{Lemma local HF is H} (see \cite[Lemma B.7]{KwonvanKoert}).}
%
%

By letting the slope of $H$ increase at infinity, 
and taking the direct limit over continuation maps, one obtains the following spectral sequences.

\begin{corollary}
Under the assumption of Lemma \ref{Lemma local HF is H}, there are convergent spectral sequences
$$
\begin{array}{lll}
E_1^{pq} \Rightarrow SH^*_+(H) &\textrm{where} &E_1^{pq}=\bigoplus_{c} H^{k-\mu(B_{p,c})}(B_{p,c}) \textrm{ for }p<0, \textrm{and 0 otherwise}
\\[1mm]
E_1^{pq} \Rightarrow SH^*(H) &\textrm{as above, except} &E_1^{0q}=H^q(M)
\\[1mm]
E_1^{pq} \Rightarrow ESH^*_+(H) &\textrm{where} &E_1^{pq}=\bigoplus_{c}EH^{k-\mu(B_{p,c})}(B_{p,c}) \textrm{ for }p<0, \textrm{and 0 otherwise}
\\[1mm]
E_1^{pq} \Rightarrow ESH^*(H) &\textrm{as above, except} &E_1^{0q}=EH^q(M)\cong H^*(M)\otimes_{\K} \F.
\end{array}
$$
(where ordinary cohomology is always computed using $\K$ coefficients.)
\end{corollary}

\end{document}